\documentclass[reqno]{amsart}
\usepackage{amssymb,amsmath,amsfonts,bm}
\usepackage{dsfont}
\usepackage{epsfig}
\usepackage{graphicx}
\usepackage{enumerate}
\usepackage{verbatim}
\usepackage{color}
\usepackage{hyperref}
\usepackage{caption}
\usepackage[margin=2.7 cm]{geometry}
\usepackage{geometry}
\usepackage{tikz}  
\usetikzlibrary{arrows}
\usepackage{rotating}
\usetikzlibrary{shapes,quotes,calc,angles,arrows,through,intersections,shadings}
\usepackage{tkz-euclide}
\tikzstyle{int}=[draw, fill=blue!20, minimum size=2em]
\tikzstyle{init} = [pin edge={to-,thin,black}]
\DeclareMathSizes{10}{9}{6}{5}
\makeatletter
\newcommand{\steps}[2]{%
  \par
  \addvspace{\medskipamount}%
  \noindent\emph{#1:}
}
\makeatother

\date{} 

 \DeclareMathOperator\supp{supp}
 \DeclareMathOperator\jac{Jac}
  \DeclareMathOperator\sign{sgn}
 \DeclareMathOperator*{\esssup}{ess\,sup}

\DeclareMathOperator\card{card}
\DeclareMathOperator\diverg{\textbf{div}}
\theoremstyle{plain}
\newtheorem{theorem}{Theorem}[section]
\newtheorem{definition}[theorem]{Definition}
\newtheorem{proposition}[theorem]{Proposition}

\newtheorem{lemma}[theorem]{Lemma}
\newtheorem{corollary}[theorem]{Corollary}
\newtheorem{remark}[theorem]{Remark}

\numberwithin{theorem}{section}
\numberwithin{equation}{section}
\numberwithin{figure}{section}

\let\oldtocsection=\tocsection
 
\let\oldtocsubsection=\tocsubsection
 
\let\oldtocsubsubsection=\tocsubsubsection
 
\renewcommand{\tocsection}[2]{\hspace{0em}\oldtocsection{#1}{#2}}
\renewcommand{\tocsubsection}[2]{\hspace{1em}\oldtocsubsection{#1}{#2}}
\renewcommand{\tocsubsubsection}[2]{\hspace{2em}\oldtocsubsubsection{#1}{#2}}

\begin{document}

\parskip=1pt

\title[Rigorous derivation of a ternary Boltzmann equation]
{Rigorous derivation of a ternary Boltzmann equation for a classical system of particles}
\author[Ioakeim Ampatzoglou]{Ioakeim Ampatzoglou}
\address{Ioakeim Ampatzoglou, Courant Institute of Mathematical Sciences, New York University}
\author[Nata\v{s}a Pavlovi\'{c}]{Nata\v{s}a Pavlovi\'{c}}
\address{Nata\v{s}a Pavlovi\'{c}, Department of Mathematics, The University of Texas at Austin}
\begin{abstract}
In this paper, we  present a rigorous derivation of a  new kinetic equation describing the limiting behavior  of a classical system of particles with three particle elastic instantaneous interactions, which are modeled using a non-symmetric version of a ternary distance.  
The ternary collisional operator we derive can be seen as the first step towards obtaining a toy model for a non-ideal gas where higher order interactions are taken into account. 
\end{abstract}
\maketitle
\setcounter{tocdepth}{1}
\tableofcontents

\section{Introduction} \label{intro}
 The Boltzmann equation \cite{Boltzmann 1, Boltzmann 2, cercignani boltzmann, cercignani gases} is the central equation of collisional kinetic theory. It is a nonlinear integro-differential equation giving the statistical description of a dilute gas in non-equilibrium in $\mathbb{R}^d$, for $d\geq 2$. It is  given by
\begin{equation}\label{intro-standard Boltzmann}
\partial_t f+v\cdot\nabla_xf=Q_2(f,f),\quad (t,x,v)\in(0,\infty)\times\mathbb{R}^{d}\times\mathbb{R}^d,
\end{equation}
where the unknown function $f:[0,\infty)\times\mathbb{R}^{d}\times\mathbb{R}^d\to\mathbb{R}$ represents the probability density of finding a molecule of the gas in position $x\in\mathbb{R}^d$, with velocity $v\in\mathbb{R}^d$, at time $t\geq 0$. The expression $Q_2(f,f)$ on the right hand side of \eqref{intro-standard Boltzmann}  is the  collisional operator which is an appropriate quadratic integral operator acting on $f$, taking into account \textbf{binary}  interactions of a pair of gas particles.   Its  exact form depends on the type of interaction between particles. Since the gas is assumed to be very dilute, interactions among three particles or higher order interactions are neglected due to much lower probability of occurring compared to binary.

However, when the gas is dense enough, higher order interactions are much more likely to happen, therefore they  produce a significant effect  to the evolution of the gas and one needs to take them into consideration. An example of such a situation is a colloid, which is a homogeneous non-crystalline substance consisting of either large molecules or ultramicroscopic particles of one substance dispersed through a second substance. As pointed out in \cite{ref 26}, multi-interactions among particles significantly contribute to the grand potential of a colloidal gas and are modeled by a sum of higher order interaction terms. A surprising but very important result of \cite{ref 26} is that  interactions  among three  particles actually depend on the sum of the distances between particles, as opposed to depending on different geometric configurations among interacting particles. This observation is apparently of invaluable computational importance since it significantly simplifies numerical calculations on three particle interactions. The results of \cite{ref 26} have been further verified experimentally e.g. \cite{imp3} and numerically e.g. \cite{imp4}.

\subsection{The program introduced and the goal of this paper}
Motivated by the fact that the Boltzmann equation is valid only for very dilute gases and by the observations of \cite{ref 26} that multi-interactions among particles contribute to the colloidal gas
(although in this paper we do not model colloids), we aim to introduce and rigorously derive (from a system of classical particles) a kinetic model which goes beyond binary interactions, by incorporating a sum of higher order interaction terms in \eqref{intro-standard Boltzmann}. Such an equation, which could serve as a toy model for a non-ideal gas,  would be of the form
\begin{equation}\label{intro:generalized Boltzmann}
\partial_t f+v\cdot\nabla_xf=\displaystyle\sum_{k=2}^m Q_k(\underbrace{f,f,\cdots,f\,}_\text{$k$-times}),\quad (t,x,v)\in(0,\infty)\times\mathbb{R}^{d}\times\mathbb{R}^d,
\end{equation}
where, for $k=1,...,m$, the expression $Q_k(f,f,...,f)$ is the $k$-th order collisional operator and $m\in\mathbb{N}$ is the accuracy of the approximation. Notice that for $m=2$, equation \eqref{intro:generalized Boltzmann} reduces to the classical Boltzmann equation \eqref{intro-standard Boltzmann}. 

The task of rigorously deriving an equation of the form \eqref{intro:generalized Boltzmann} from a classical many particle system, even for the case $m=2$, is a challenging problem that has been settled for short times only in certain situations;  for   hard-sphere interactions, the analysis was pioneered by Lanford \cite{lanford} and recently completed by Gallagher, Saint-Raymond, Texier \cite{gallagher}, while  for short-range potentials, it has been done in \cite{king,gallagher, pulvirenti-simonella}. 
 Up to our knowledge, the case $m=3$ i.e. derivation of the equation
\begin{equation}\label{intro:binary-ternary Boltzmann}
\partial_t f+v\cdot\nabla_xf=Q_2(f,f)+Q_3(f,f,f),\quad (t,x,v)\in(0,\infty)\times\mathbb{R}^{d}\times\mathbb{R}^d,\\
\end{equation}
 has not been studied at all. We refer to  \eqref{intro:binary-ternary Boltzmann} as the binary-ternary Boltzmann equation. We  mention that in a recent work with Gamba and Taskovi\'c \cite{gwp}  we proved  global well-posedness of \eqref{intro:binary-ternary Boltzmann} for small initial data near vacuum.
 
  In addition to understanding binary interactions and interactions among three particles,
 derivation of \eqref{intro:binary-ternary Boltzmann} requires careful analysis of their mutual interactions. This challenging task has been carried out in a subsequent work \cite{binary-ternary} since it  requires a deep understanding of interactions between three particles and their connection to binary interactions. For this reason,  in this paper, we focus on understanding interactions among three particles and rigorously deriving a purely ternary equation, which itself brings a lot of challenges due to combinatorial and configurational intricacies of evolving in time interactions among three particles. We derive an equation of the form
\begin{equation}\label{intro:ternary Boltzmann}
\partial_t f+v\cdot\nabla_xf=Q_3(f,f,f),\quad (t,x,v)\in(0,\infty)\times\mathbb{R}^{d}\times\mathbb{R}^d,\\
\end{equation}
where $Q_3(f,f,f)$ is the ternary collisional operator which is an integral operator of cubic order in $f$.  We refer to \eqref{intro:ternary Boltzmann} as the ternary Boltzmann equation. Global well-posedness for small initial data near vacuum holds as a special case of the results of \cite{gwp}.

Let us mention that Maxwell models with multiple particle interactions have been studied in \cite{multiple gamba 2, multiple gamba} using Fourier transform methods.

Also, we note  that attempts for generalization of the Boltzmann equation  using formal density expansions were made by physicists   in the past, see e.g. \cite{choh-uhlenbeck,cohen1,green,hoegy,sengers}, but in a  different context than ours. These attempts have not been further developed since  since the fourth and higher order collisions, terms as well as the virial expansion of the solution, diverged as the number of particles increased.  According to \cite{cohen50}, the divergences
originate from the desire to make a systematic expansion of the macroscopic
properties of a large system consisting of many particles in terms of the
properties of small (isolated) groups of 2,3,4 etc particles, i.e., from the basic
idea of the virial expansion itself. This leads to formal expansions in terms of
collision integrals containing the dynamics of an increasing number of particles.
These integrals diverge, however, in general, due to long range dynamical
correlations between successive collisions of these particles, introduced by the
possibility of unrestricted free motion of particles between successive collisions.

\subsection{Ternary interactions and their scaling}

 In a typical, dilute hard-sphere gas, the probability of a simultaneous contact of three hard-spheres is very small compared to e.g. the situation when one of the three particles is in simultaneous contact with the other two particles. 
 
 Motivated by this observation and the fact that in some physical situations, such as when one considers colloids as in \cite{ref 26},  interactions among three particles are determined by the sum of the distances of the interacting particles, we introduce the notion of an interaction of three particles based on a non-symmetric version of a ternary distance. More precisely,  we introduce the ternary distance:
 \begin{equation}\label{into:ternary distance}
 d(x_1;x_2,x_3):=\sqrt{|x_1-x_2|^2+|x_1-x_3|^2},\quad x_1,x_2,x_3\in\mathbb{R}^d.
 \end{equation}
Having defined the ternary distance, we introduce the notion of a ternary interaction. Let $\epsilon>0$ and consider three particles $i,j,k$ with positions and velocities $(x_i,v_i), (x_j,v_j), (x_k,v_k)\in\mathbb{R}^{2d}$. We say that the particles $i,j,k$ are in $(i;j,k)$ ternary $\epsilon$-interaction\footnote{when not ambiguous, we will refer to $(i;j,k)$ ternary $\epsilon$-interaction as $(i;j,k)$ interaction.}  if the following geometric condition holds:
 \begin{equation}\label{intro:(ijk)}
d^2(x_i;x_j,x_k)=|x_i-x_j|^2+|x_i-x_k|^2=2\epsilon^2.
\end{equation}
The parameter $\epsilon$  above is called interaction zone. The $i$-th particle is called the central collisional particle, while the particles $j,k$ are called adjacent collisional particles.

Heuristically speaking, an $(i;j,k)$ interaction expresses the interaction of  the central particle $i$ with the pair of the uncorrelated adjacent particles $(j,k)$ with respect to the interaction zone $\epsilon$. By uncorrelated, we mean that particles $j,k$ are not directly affected by each other. For example, Figure \ref{non collisional triple}  shows particles that are not in ternary interaction, 
\begin{figure}[htp]
\centering
\begin{tikzpicture}[scale=1]
\coordinate (a) at (0,0);
\node (A) at (a) {{\color{red}$x_i$}};
\coordinate (a') at (0,-0.7);
\coordinate (ab) at (-0.3,1.5);
\node (text) at (ab){\small{$2\epsilon$}};
\coordinate (b) at (3,0);
\node (B) at (b) {{\color{blue}$x_j$}};
\coordinate (c) at (0,3);
\node (C) at (c) {{\color{teal}$x_k$}};
\coordinate (ac) at (1.5,0.3);
\node (text) at (ac){\small{$2\epsilon$}};
\path[-] (A) edge node {} (B);
\path[-] (A) edge node {} (C);
\end{tikzpicture}
\captionsetup{justification=centering}
\caption{ }
\label{non collisional triple}
\end{figure}
while Figure \ref{collisional triple} offers two examples of particles which are in ternary interaction.
\begin{figure}[htp]
\centering
\begin{tikzpicture}[scale=1]
\coordinate (a) at (0,0);
\node (A) at (a) {{\color{red}$x_i$}};
\coordinate (a') at (0,-0.7);
\coordinate (ab) at (-0.3,0.75);
\node (text) at (ab){\small{$\epsilon$}};
\coordinate (b) at (1.5,0);
\node (B) at (b) {{\color{blue}$x_j$}};
\coordinate (c) at (0,1.5);
\node (C) at (c) {{\color{teal}$x_k$}};
\coordinate (ac) at (0.75,0.3);
\node (text) at (ac){\small{$\epsilon$}};
\path[-] (A) edge node {} (B);
\path[-] (A) edge node {} (C);
\end{tikzpicture}\hspace{4cm}
\begin{tikzpicture}[scale=1]
\coordinate (a) at (4,0);
\node (A) at (a) {{\color{red}$x_i$}};
\coordinate (a') at (4,-0.7);
\coordinate (ab) at (3.5,0.92);
\node (text) at (ab){\small{$\frac{\sqrt{6}}{2}\epsilon$}};
\coordinate (b) at (5.06,0);
\node (B) at (b) {{\color{blue}$x_j$}};
\coordinate (c) at (4,1.837);
\node (C) at (c) {{\color{teal}$x_k$}};
\coordinate (ac) at (4.53,0.3);
\node (text) at (ac){\small{$\frac{\sqrt{2}}{2}\epsilon$}};
\path[-] (A) edge node {} (B);
\path[-] (A) edge node {} (C);
\end{tikzpicture}
\captionsetup{justification=centering}
\caption{ }
\label{collisional triple}
\end{figure}

Let us now describe how velocities instantaneously transform when a ternary interaction happens. Consider an $(i;j,k)$ ternary $\epsilon$-interaction.
Let $v_i^*,v_j^*,v_k^*$ denote the velocities of the interacting particles after the interaction. Assuming the particles are of equal mass $m=1$, we consider the interaction to be elastic i.e. the three particle momentum-energy conservation system   is satisfied:
\begin{align}
v_i^*+v_j^*+v_k^*&=v_i+v_j+v_k,\label{intro:conser of momentum}\\
|v_i^*|^2+|v_j^*|^2+|v_k^*|^2&=|v_i|^2+|v_j|^2+|v_k|^2.\label{intro:conser of energy}
\end{align}
Now we introduce the relative positions  re-scaled vectors
$
(\widetilde{\omega}_1,\widetilde{\omega}_2):=\left(\frac{x_j-x_i}{\sqrt{2}\epsilon},\frac{x_k-x_i}{\sqrt{2}\epsilon}\right).
$
Notice that \eqref{intro:(ijk)} implies $(\widetilde{\omega}_1,\widetilde{\omega}_2)\in\mathbb{S}_1^{2d-1}$ i.e.
$
|\widetilde{\omega}_1|^2+|\widetilde{\omega}_2|^2=1.
$
We shall call the vectors $\widetilde{\omega}_1,\widetilde{\omega}_2$ impact directions of the interaction.
Since the $i$ particle interacts with the  pair of uncorrelated particles $(j,k)$,  we assume the velocities $v_j$, $v_k$ transform with respect to the impact directions unit vector i.e.
\begin{equation}\label{intro:assumption on (j,k)}
\begin{pmatrix}
v_j^*\\v_k^*
\end{pmatrix}
=
\begin{pmatrix}
v_j\\v_k
\end{pmatrix}
-c
\begin{pmatrix}
\widetilde{\omega}_1\\
\widetilde{\omega}_2
\end{pmatrix},
\end{equation}
for some $c\in\mathbb{R}$.
We note that once we added condition\footnote{we note that \eqref{intro:assumption on (j,k)} is the ternary analogue of the condition that appears when one considers binary interactions, see e.g. \cite{gallagher}. } \eqref{intro:assumption on (j,k)} to the system \eqref{intro:conser of momentum}-\eqref{intro:conser of energy}, the new system has a unique solution that algebraically characterizes
the conservation of   momentum and energy for the type of ternary interaction defined in \eqref{intro:(ijk)}.
It is straightforward to verify  that \eqref{intro:conser of momentum}-\eqref{intro:assumption on (j,k)} yield that  $v_i^*,v_j^*,v_k^*$ are given by the collisional formulas 
\begin{equation}
\begin{aligned}
v_i^*&=v_i+\frac{\langle\widetilde{\omega}_1,v_j-v_i\rangle+\langle\widetilde{\omega}_2,v_k-v_i\rangle}{1+\langle\widetilde{\omega}_1,\widetilde{\omega}_2\rangle}(\widetilde{\omega}_1+\widetilde{\omega}_2),\\
v_j^*&=v_j-\frac{\langle\widetilde{\omega}_1,v_j-v_i\rangle+\langle\widetilde{\omega}_2,v_k-v_i\rangle}{1+\langle\widetilde{\omega}_1,\widetilde{\omega}_2\rangle}\widetilde{\omega}_1,\\
v_k^*&=v_k-\frac{\langle\widetilde{\omega}_1,v_j-v_i\rangle+\langle\widetilde{\omega}_2,v_k-v_i\rangle}{1+\langle\widetilde{\omega}_1,\widetilde{\omega}_2\rangle}\widetilde{\omega}_2.
\end{aligned}
\end{equation}

\subsection{Phase space and scaling of ternary interactions}
Now we are ready  to describe the evolution of a system of $N$-particles of $\epsilon$-interaction zone.  Recall that in this paper we pursue only ternary interactions analysis, thus the phase space will take into account  only those.
\begin{definition}\label{intro-def of triary collisions} Let $d\in\mathbb{N}$, with $d\geq 2$, $N\in\mathbb{N}$ and $\epsilon>0$. The phase space of the $N$-particle system of $\epsilon$-interaction zone is defined as:
\begin{equation}\label{intro:triary phase space}
\mathcal{D}_{N,\epsilon}=\left\{Z_N=(X_N,V_N)\in\mathbb{R}^{2dN}:d^2(x_i;x_j,x_k)\geq2\epsilon^2\quad\forall 1\leq i<j<k\leq N\right\},
\end{equation}
where
$
d^2(x_i;x_j,x_k)=|x_i-x_j|^2+|x_i-x_k|^2,
$
and
$X_N=(x_1,...,x_N)\in\mathbb{R}^{dN},\quad V_N=(v_1,...,v_N)\in\mathbb{R}^{dN},$
represent the positions and velocities of the $N$-particles.
\end{definition}

In terms of scaling, one could interpret an $(i;j,k)$ of interaction zone $\epsilon$ as a special hard sphere interaction of radius $\sqrt{2}\epsilon$ in $\mathbb{R}^{2d}$, since expression \eqref{intro:(ijk)} can be written as
$$|\bm{x_i}-\bm{x_{j,k}}|_{2d}=\sqrt{2}\epsilon,$$
where $\bm{x_{i,i}}=\displaystyle\binom{x_i}{x_i}$ and $\bm{x_{j,k}}=\displaystyle\binom{x_j}{x_k}$. Then a  $2d$-particle with position $\bm{x_{i,i}}$ would span a volume of order $\epsilon^{2d-1}$ in a unit of time. In order to observe $O(1)$ interaction per unit of time, there are $N^2$ options for the $2d$-particle positioned at $\bm{x_{j,k}}$. We obtain that $N^2\epsilon^{2d-1}= O(1)$ or equivalently
\begin{equation}\label{intro scaling talk}
N\epsilon^{d-1/2}= O(1).
\end{equation} 
This is the new scaling in which we will observe this kind of ternary interactions, see Section \ref{sec_derivation} for the explicit appearance of this scaling in the calculations.

\begin{remark}\label{remark symmetric}
 The phase space \eqref{intro:triary phase space} will produce the kinetic equation \eqref{equation derived intro}, in which the tracked particle is always the central particle of the interactions occurring. Alternatively, by working on  the phase space  
\begin{equation}\label{symmetric phase space}
\begin{aligned}
\widetilde{D}_{N,\epsilon}=\big\{Z_N=(X_N,V_N)\in\mathbb{R}^{2dN}: d_l^2(x_i,x_j,x_k)\geq 2\epsilon^2,\quad\forall (i,j,k,l)\in\widetilde{\mathcal{I}}_N\},&
\end{aligned}
\end{equation}
where 
$$\widetilde{\mathcal{I}}_N=\{(i,j,k,l):(i,j,k)\in\mathcal{I}_N\text{ and }l:\{i,j,k\}\to\{i,j,k\}\text{ is a permutation}\},$$
$$d_l(x_i,x_j,x_k)=\sqrt{|x_{l_i}-x_{l_j}|^2+|x_{l_i}-x_{l_k}|^2},$$
and using  similar arguments as in this paper, one can derive a symmetrized version of \eqref{equation derived intro}, in which the tracked particle can be either central or adjacent. Moreover, it has been shown in \cite{thesis}, that the symmetrized ternary equation satisfies similar statistical and entropy production properties as the classical Boltzmann equation. In particular, it has a weak formulation which yields an $\mathcal{H}$-Theorem and local conservation of mass, momentum and energy.
For simplicity, we opt to work with the phase space \eqref{intro:triary phase space}.  However,  we would like to mention that all the intermediate results needed for the derivation of the symmetrized ternary equation can be obtained  after some minor changes, see \cite{thesis} for more details. 
\end{remark}

\subsection{Global existence of a flow and the Liouville equation}
Let us now describe the evolution in time of  a system of particles in the phase space \eqref{intro:triary phase space}. Consider  an initial configuration $Z_N\in\mathcal{D}_{N,\epsilon}$. The motion is described as follows: \\
(I) Particles are assumed to perform rectilinear motion as long as there is no interaction i.e.
$$ \dot{x}_i=v_i,\quad \dot{v}_i=0,\quad\forall i\in\{1,...,N\}.$$
(II) Assume now that an initial configuration $Z_N=(X_N,V_N)$ has evolved until time $t>0$, reaching $Z_N(t)=(X_N(t),V_N(t))$, and there is an $(i;j,k)$ interaction at time $t$. Then the velocities $(v_i(t), v_j(t), v_k(t))$ instantaneously transform to  $(v_i^*(t), v_j^*(t), v_k^*(t))$.

We remark that it is not at all obvious that  (I)-(II) produce a well defined dynamics, since the evolution is not smooth in time, and the system can possibly run into pathological configurations. In the case of binary interactions, the analogous result has been established in the work of Alexander \cite{alexander}.  Our dynamics will be constructed in a similar spirit to \cite{alexander}. However a distinction between ternary precollisional and postcollisional configurations as well as new geometric estimates are needed in order to control possible emergence of pathological trajectories.

We informally state the first main result of this paper, for a rigorous statement see Theorem \ref{global flow}.

\textbf{\textit{Existence of a global flow}:}\textit{ Let $m\in\mathbb{N}$ and $0<\sigma<<1$. There is a global in time measure-preserving flow $(\Psi_m^t)_{t\in\mathbb{R}}:\mathcal{D}_{m,\sigma}\to\mathcal{D}_{m,\sigma}$  which preserves kinetic energy.
This flow is called the $\sigma$-interaction zone flow of $m$-particles or simply the interaction flow.} 

The main difficulty in proving Theorem \ref{global flow} is the elimination of configurations following pathological trajectories in time. In particular, in order to go from local  to global in time flow we establish the following crucial fact - when an $(i;j,k)$ interaction happens, then the subsequent interaction cannot involve the same triplet of particles.  This observation enables us to develop  ellipsoidal coverings and new geometric estimates to control the measure of these pathological sets.

The global measure-preserving  interaction flow established yields a  Liouville equation (see \eqref{Liouville equation}) for the evolution $f_N$ of an initial $N$-particle of $\epsilon$-interaction zone probability density $f_{N,0}$.

\subsection{The ternary equation derived} Although Liouville's equation is a linear transport equation, efficiently solving it is almost impossible in case where the particle number $N$ is very large. This is why an accurate statistical description is welcome, and to obtain it one wants to understand the limiting behavior of it as $N\to\infty$ and $\epsilon\to 0^+$, with the hope that  qualitative properties will be revealed for a large but finite $N$.
  Letting the number of particles $N\to\infty$ and the interaction zone $\epsilon\to 0^+$ in the \textbf{new scaling}:
\begin{equation}\label{intro:our scaling}
N\epsilon^{d-1/2}= 2^{1-d/2},
\end{equation}
we derive the ternary Boltzmann equation
\begin{equation}\label{equation derived intro}
\partial_t f+v\cdot\nabla_x f=Q_3(f,f,f), \quad (t,x,v)\in (0,\infty)\times\mathbb{R}^d\times\mathbb{R}^d.
\end{equation}
 The expression $Q_3(f,f,f)$ is the ternary  cubic order collisional operator, given by:
\begin{equation}\label{intro:kernel}
Q_3(f,f,f)=\int_{\mathbb{S}_1^{2d-1}\times\mathbb{R}^{2d}}\frac{b_+(\omega_1,\omega_2,v_1-v,v_2-v)}{\sqrt{1+\langle\omega_1,\omega_2\rangle}}\left(f^*f_1^*f_2^*-ff_1f_2\right)\,d\omega_1\,d\omega_2\,dv_1\,dv_2,
\end{equation}
where 
\begin{equation}\label{intro:parameters boltzmann}
\begin{aligned}
b(\omega_1,\omega_2,v_1-v,v_2-v):=\langle\omega_1,v_{1}-v\rangle+\langle\omega_2, v_{2}-v\rangle,\quad b_+=\max\{b,0\},&\\
 f^*=f(t,x,v^*),f=f(x,t,v),f_i^*=f_i^*(t,x,v_i^*),f_i=f(t,x,v_i) \text{ for  }i\in\{1,2\}.&
\end{aligned}
\end{equation}
\begin{remark}
The ternary collisional operator could be written in a more general form as:
\begin{equation*}
Q_3(f,f,f)=\int_{\mathbb{S}_1^{2d-1}\times\mathbb{R}^{2d}}B(\bm{u},\bm{\omega})\left(f^*f_1^*f_2^*-ff_1f_2\right)\,d\omega_1\,d\omega_2\,dv_1\,dv_2,
\end{equation*}
where
$\bm{u}=\binom{v_1-v}{v_2-v}\in\mathbb{R}^{2d}$, $\bm{\omega}=\binom{\omega_1}{\omega_2}\in\mathbb{S}_1^{2d-1}$ are the vectors of relative velocities and scaled relative positions of the colliding particles.
Of particular interest would be the power law potentials: 
$$B(\bm{u},\bm{\omega})=|\bm{u}|^\gamma\widetilde{b}(\bm{\widehat{u}\cdot\omega},\langle\omega_1,\omega_2\rangle),$$ where $\widetilde{b}$ is the differential cross-section and $\bm{\widehat{u}}$ is the unit vector in the direction of $\bm{u}$. In this paper, we derive equation \eqref{equation derived intro} for the case 
$$\gamma=1,\quad \widetilde{b}(\bm{\widehat{u}\cdot\omega},\langle\omega_1,\omega_2\rangle)=\frac{(\bm{\widehat{u}\cdot\omega})_+}{\sqrt{1+\langle\omega_1,\omega_2\rangle}}.$$ For a study of the global well-posedness of \eqref{equation derived intro} for power law potentials with $\gamma\in(-2d+1,1]$, see \cite{gwp}.

\end{remark}

\subsection{Strategy of the derivation and statement of the main result} Now the natural question is: how do we pass from the $N$-particle dynamics to  the kinetic equation \eqref{equation derived intro}? We implement the program  pioneered
 by Lanford \cite{lanford} and recently refined by Gallagher, Saint-Raymond, Texier \cite{gallagher} for deriving, for short times, the classical Boltzmann equation \eqref{intro-standard Boltzmann} for hard-spheres in the Boltzmann-Grad \cite{Grad 1, Grad 2} scaling
 $
 N\epsilon^{d-1}\simeq 1.
$
 This program has been implemented in the case of short range potentials too e.g.  \cite{king,gallagher,pulvirenti-simonella}.  However, to the best of our knowledge, the program has not been explored outside of  the context of binary interactions. By generalizing the program to allow consideration of ternary particle interactions, we illustrate that the program is universal enough. However to make it applicable to ternary interactions we follow evolution in time of ternary particle interactions, that inform new mathematical arguments described below.

We first derive a finite two-step\footnote{the two-step refers to the coupling between the $k$-th element of the hierarchy and the $(k+2)$-th element of the hierarchy. } coupled hierarchy of equations for the marginals densities of the solution to the Liouville equation, which we call the  BBGKY\footnote{Bogoliubov, Born, Green, Kirkwood, Yvon.} hierarchy. We then formally let $N\to\infty$ and $\epsilon\to 0^+$ in the scaling \eqref{intro:our scaling} to obtain an infinite two-step coupled hierarchy of equations, which we call the Boltzmann hierarchy. It can be observed that for factorized initial data,  the Boltzmann hierarchy  reduces to the ternary Boltzmann equation \eqref{equation derived intro}.  This observation connects the Boltzmann hierarchy with the ternary Boltzmann equation.  

To make this argument rigorous, we first need to show that the BBGKY and Boltzmann hierarchy are well-posed, at least for short times, and then that  if the BBGKY initial data converge to the Boltzmann hierarchy initial data, then this convergence propagates in time in the scaling \eqref{intro:our scaling}. Local well-posedness is shown in Section \ref{sec_local}, see Theorem \ref{well posedness BBGKY}, Theorem \ref{lwp Boltzmann hier}. Showing convergence is a very challenging task and is the heart of our contribution. We informally state our main result here. For a rigorous statement of the result see Theorem \ref{convergence theorem}.\\
\\
\textbf{\textit{Statement of the main result}:} \textit{Let $F_0$ be  initial data for the Boltzmann hierarchy, and $F_{N,0}$ be some BBGKY hierarchy initial data  which ``approximate" $F_0$ as $N\to\infty$, $\epsilon\to 0^+$ under the scaling \eqref{intro:our scaling}. Let $\bm{F_N}$ be the solution to the BBGKY hierarchy  with initial data $F_{N,0}$, and $\bm{F}$ the solution to the Boltzmann hierarchy, with initial data $F_0$, up to short time $T>0$. Then $\bm{F_N}$ converges in observables to $\bm{F}$ in $[0,T]$ as $N\to\infty$, $\epsilon\to 0^+$, under the scaling \eqref{intro:our scaling}. In the case of H\"older continuous $C^{0,\gamma}$,$\gamma\in(0,1]$ tensorized Boltzmann hierarchy initial data and  approximation by conditioned BBGKY hierarchy initial data, we obtain convergence to the solution of the ternary Boltzmann equation \eqref{equation derived intro} with a rate $O(\epsilon^r)$ for any $0<r<\min\{1/2,\gamma\}.$}\\
\\
The proof of this result is achieved by repeatedly using Duhamel's formula for the finite and infinite hierarchy respectively and comparing the corresponding series expansions. However this a delicate point because of the divergence between the finite particle flow and the free flow, due to the ternary interactions of particles in the finite particle case.  The  problem of divergence is present in the derivation of the classical Boltzmann equation as well, see \cite{lanford,gallagher}, but  our case is significantly harder due the  complexity of ternary interactions. To overcome this problem, we develop new geometric and combinatorial estimates, that help us  extract small measure sets of initial data which lead to these diverging trajectories. In particular the main difficulty is to control post-collisional configurations and it requires completely new treatment. To achieve that, we need to explicitly calculate the Jacobian of ternary interactions with respect to impact directions,  and  estimate the surface measure of sets of the form $(K_{\rho}^d\times\mathbb{R}^d)\cap\mathcal{S}$, where $K_{\rho}^d$ is a $d$-dimensional solid cylinder of radius $\rho$ and $\mathcal{S}$ is an appropriate  ellipsoid in $\mathbb{R}^{2d}$. These results are thoroughly presented in Section \ref{sec_geometric}.

\subsection{Further discussion}
 While this paper models ternary interactions among particles via a concept of a ternary distance (namely when \eqref{intro:(ijk)} holds), we note that a more physical way would be to employ a three-body potential of a small interaction zone. In particular, one could 
consider $\Phi:\mathbb{R}^{2d}\to\mathbb{R}$ non-negative, smooth and supported in the unit ball $B_1^{2d}$. Then one would work in the entire space with Newton's equations
$$\dot{x}_i=v_i\quad \dot{v}_i=-\frac{1}{\epsilon}\sum_{\substack{i,j,k\in\{1,...,N\}\\i\neq j\neq k}}\nabla\Phi\left(\frac{x_i-x_j}{\epsilon},\frac{x_i-x_k}{\epsilon}\right).$$ 
Although we did not pursue  analysis of this model, we expect that the relevant scaling \eqref{intro:our scaling}  and the techniques introduced in this paper might be helpful in that context as well.

\subsection*{Acknowledgements} I.A. and N.P.  acknowledge support  from NSF grants DMS-1516228, DMS-1840314 and DMS-2009549. I.A. acknowledges support from the Simons Collaboration on Wave Turbulence.  Authors are thankful to Thomas Chen, Irene M. Gamba, Philip Morrison,  Maja Taskovi\'{c} and Joseph K. Miller for helpful discussions regarding physical and mathematical aspects of the problem. Authors would like to thank Ryan Denlinger for his  constructive suggestions regarding geometric estimates in this paper.  Finally, authors are grateful to the reviewers for their in depth comments and remarks which significantly improved the manuscript.
\subsection{Notation} For convenience, we introduce some basic notation which will be  used throughout the manuscript:\\
$\bullet$ We write 
$x\lesssim y$ if there exists $C_d>0$ with $ x\leq C_d y$.\\
$\bullet$  Given $n\in\mathbb{N}$, $\rho>0$ and $w\in\mathbb{R}^n$, we write $B_\rho^n(w)$ for the $n$-closed ball of radius $\rho>0$, centered at $w\in\mathbb{R}^n$.
In particular, we write $B_\rho^n:=B_\rho^n(0)$
for the $\rho$-ball centered at the origin.\\
$\bullet$  Given $n\in\mathbb{N}$ and $\rho>0$, we write $\mathbb{S}_\rho^{n-1}$ for the $(n-1)$-sphere of radius $\rho>0$.\\
$\bullet$   We write $x<<y$, when $x<cy$ for some number $0<c<1$ small enough.

\section{Collisional transformation of three particles}
\label{sec_collision}
In this  section, we define the collisional transformation of three particles  induced by a pair of impact directions, and investigate its properties.  

For convenience, given $(\omega_1,\omega_2,v_1,v_2,v_3)\in\mathbb{S}_1^{2d-1}\times\mathbb{R}^{3d}$, let us write
\begin{equation}\label{definition of c}
c_{\omega_1,\omega_2,v_1,v_2,v_3}=\frac{\langle\omega_1,v_2-v_1\rangle+\langle\omega_2,v_3-v_1\rangle}{1+\langle\omega_1,\omega_2\rangle}.
\end{equation}
Notice that $c_{\omega_1,\omega_2,v_1,v_2,v_3}$ is well-defined for all $(\omega_1,\omega_2,v_1,v_2,v_3)\in\mathbb{S}_1^{2d-1}\times\mathbb{R}^{3d}$, since
\begin{equation}\label{estimate on denominator}
1+\langle\omega_1,\omega_2\rangle\geq 1-|\omega_1||\omega_2|\geq 1-\frac{1}{2}\left(|\omega_1|^2+|\omega_2|^2\right)=\frac{1}{2}.
\end{equation}

\begin{definition}
Consider  impact directions $(\omega_{1},\omega_{2})\in\mathbb{S}_{1}^{2d-1}$. We  define the collisional transformation induced by $(\omega_1,\omega_2)\in\mathbb{S}_1^{2d-1}$ as
$ T_{\omega_1,\omega_2}:(
v_1,
v_2,
v_3
)\in\mathbb{R}^{3d}\longrightarrow(
v_1^*,
v_2^*,
v_3^*)\in\mathbb{R}^{3d},
$
where
\begin{equation}\label{formulas of collision}\begin{cases}
v_1^*=v_1+c_{\omega_{1},\omega_{2}, v_1,v_2,v_3}(\omega_{1}+\omega_{2}),\\
v_2^*=v_2-c_{\omega_{1},\omega_{2}, v_1,v_2,v_3}\omega_{1},\\
v_3^*=v_3-c_{\omega_{1},\omega_{2}, v_1,v_2,v_3}\omega_2,\\
\end{cases}
\end{equation}
and $c_{\omega_{1},\omega_{2}, v_1,v_2,v_3}$ is given by \eqref{definition of c}.
\end{definition}
In the following definition, we introduce the notion of the cross-section which will have a prominent role in the rest of the paper.
\begin{definition}
We define the cross-section\footnote{We } $b:\mathbb{S}_1^{2d-1}\times\mathbb{R}^{2d}\to\mathbb{R}$ as:
\begin{equation}\label{cross}
b(\omega_1,\omega_2,\nu_1,\nu_2)=\langle\omega_1,\nu_1\rangle+\langle\omega_2,\nu_2\rangle,\quad (\omega_1,\omega_2,\nu_1,\nu_2)\in\mathbb{S}_1^{2d-1}\times\mathbb{R}^{2d}.
\end{equation}
\end{definition}
Notice that by \eqref{definition of c}, \eqref{cross} we have
\begin{equation}\label{relation cross-c}
b(\omega_1,\omega_2,v_2-v_1,v_3-v_1)=\left(1+\langle\omega_1,\omega_2\rangle\right)c_{\omega_1,\omega_2,v_1,v_2,v_3}.
\end{equation}
Direct algebraic calculations illustrate the main properties of the collisional tranformation.
\begin{proposition}\label{operator properties}Consider a pair of impact directions $(\omega_{1},\omega_{2})\in\mathbb{S}_{1}^{2d-1}$. The induced collisional transformation $T_{\omega_{1},\omega_{2}}$ has the following properties:\\
  {\it(i)} Conservation of momentum
  \begin{equation}\label{conservation of momentum tran}
  v_1^*+v_2^*+v_3^*=v_1+v_2+v_3.
  \end{equation}
  {\it(ii)} Conservation of energy:
  \begin{equation}\label{conservation of energy tran}
  |v_1^*|^2+|v_2^*|^2+|v_3^*|^2=|v_1|^2+|v_2|^2+|v_3|^2.
  \end{equation}
  {\it(iii)} Conservation of relative velocities magnitude:
\begin{equation}\label{rel.vel}
|v_1^*-v_2^*|^2+|v_1^*-v_3^*|^2+|v_2^*-v_3^*|^2=|v_1-v_2|^2+|v_1-v_3|^2+|v_2-v_3|^2.
\end{equation}
  {\it(iv)} Micro-reversibility  of the cross-section:
\begin{equation}\label{symmetrid}
b(\omega_{1},\omega_{2},v_2^*-v_1^*,v_3^*-v_1^*)=-b(\omega_{1},\omega_{2},v_2-v_1,v_3-v_1).
\end{equation}
  {\it(v)} $T_{\omega_{1},\omega_{2}}$ is a linear involution i.e. $T_{\omega_{1},\omega_{2}}$ is linear, $T_{\omega_1,\omega_2}^{-1}=T_{\omega_1,\omega_2}$.
In particular $|\det T_{\omega_1,\omega_2}|=1$, thus $T_{\omega_1,\omega_2}$ is measure-preserving.
\end{proposition}
\begin{proof} \textit{(i)} and \textit{(ii)} are guaranteed by construction.
\textit{(iii)} comes immediately after combining \textit{(i)} and \textit{(ii)}. To prove \textit{(iv)}, we use  \eqref{formulas of collision} to obtain
\begin{equation*}
v_2^*-v_1^*=v_2-v_1-2c\omega_1-c\omega_2,\quad v_3^*-v_1^*=v_3-v_1-2c\omega_2-c\omega_1.
\end{equation*}
Using the fact that $(\omega_1,\omega_2)\in\mathbb{S}_1^{2d-1}$, and recalling \eqref{relation cross-c}, we get
\begin{equation*}
\begin{aligned}
b^*=\langle\omega_1, v_2^*-v_1^*\rangle+\langle \omega_2,v_3^*-v_1^*\rangle &=\langle \omega_1,v_2-v_1\rangle+\langle \omega_2,v_3-v_1\rangle-2c_{\omega_1,\omega_2,v_1,v_2,v_3}\left(1+\langle\omega_1,\omega_2\rangle\right)=-b,
\end{aligned}
\end{equation*}
where we use the notation $b:=b(\omega_1,\omega_2,v_2-v_1,v_3-v_1)$, $b^*:=b(\omega_1,\omega_2,v_2^*-v_1^*,v_3-v_1^*)$.
To prove \textit{(v)}, first notice that $T_{\omega_{1},\omega_{2}}$ is  linear in velocities. Recalling notation from \eqref{relation cross-c}, \textit{(iv)} implies that $c^*=-c$ where $c^*:=c_{\omega_1,\omega_2,v_1^*,v_2^*,v_3^*}$, $c:=c_{\omega_1,\omega_2,v_1,v_2,v_3}$. This observation and \eqref{formulas of collision}  directly imply that
$T_{\omega_{1},\omega_{2}}^{-1}=T_{\omega_{1},\omega_{2}}$. Clearly $\left|\det T_{\omega_{1},\omega_{2}}\right|=1$
 and $T_{\omega_1,\omega_2}$ is measure-preserving.
\end{proof}

\section{Dynamics of $m$-particles}\label{sec_dynamics}

In this section we rigorously define the dynamics of $m$-particles of small interaction zone $0<\sigma<<1$. Heuristically speaking particles perform free motion as long as they are not interacting, and instantaneously transform velocities according to the collisional transformation, defined in Section \ref{sec_collision}, when they interact. Intuitively, the dynamics is well-defined as long as we have well-separated in time interactions, such that each of those interactions  involves only one triplet. Here, we show that a flow can be defined for almost all initial configurations.

Throughout this section we consider $m\in\mathbb{N}$  and $0<\sigma<<1$. We assume $m\geq 3$ unless stated.
 
 \subsection{Phase space definitions}
Consider the set $\mathcal{I}_m:=\left\{(i,j,k)\in \{1,...,m\}^3: i<j<k\right\}$ of ordered triples in $\{1,...,m\}$.
We define the phase space of the $m$-particles of $\sigma$-interaction zone  as
\begin{equation}\label{phase space}
\mathcal{D}_{m,\sigma}:=\left\{Z_m=(X_m,V_m)\in\mathbb{R}^{2dm}:d^2(x_i;x_j,x_k)\geq 2\sigma^2,\quad\forall (i,j,k)\in\mathcal{I}_m\right\},
\end{equation}
where 
$
X_m=(x_1,...,x_m)\in\mathbb{R}^{dm}$, $V_m=(v_1,...,v_m)\in\mathbb{R}^{dm}$ represent the positions and velocities of the $m$-particles respectively, and
\begin{equation}\label{distance}
d(x_i;x_j,x_k)=\sqrt{|x_i-x_j|^2+|x_i-x_k|^2},
\end{equation}
is the distance in positions of the particles $i,j,k$. 
Finally, we also define
$\mathcal{D}_{1,\sigma}\equiv\mathbb{R}^{2d}$, $\mathcal{D}_{1,\sigma}^X\equiv\mathbb{R}^{d}$.   Elements of $\mathcal{D}_{m,\sigma}$ are called phase space configurations.

 The phase space $\mathcal{D}_{m,\sigma}$ decomposes to the interior and the boundary:
\begin{align}\mathring{\mathcal{D}}_{m,\sigma}&=\left\{Z_m=(X_m,V_m)\in\mathbb{R}^{2dm}:d^2(x_i;x_j,x_k)>2\sigma^2,\quad\forall (i,j,k)\in\mathcal{I}_m\right\},\label{phase space interior}\\
\partial \mathcal{D}_{m,\sigma} &=
\bigcup_{(i,j,k)\in\mathcal{I}_m}\Sigma_{ijk},\text{ }\Sigma_{ijk}:=\left\{Z_m=(X_m,V_m)\in\mathcal{D}_{m,\sigma}:d^2(x_i;x_j,x_k)=2\sigma^2\right\}\label{phase space boundary}.
\end{align}

We further decompose the boundary to simple collisions and multiple collisions respectively:
\begin{align}\partial_{sc}\mathcal{D}_{m,\sigma}&=\left\{Z_m=(X_m,V_m)\in\partial\mathcal{D}_{m,\sigma}:\mbox{there is unique }(i,j,k)\in\mathcal{I}_m:Z_m\in\Sigma_{ijk}\right\},\label{simple collisions}\\
\partial_{mc}\mathcal{D}_{m,\sigma}&=\left\{Z_m=(X_m,V_m)\in\partial\mathcal{D}_{m,\sigma}:\mbox{there are } (i,j,k)\neq (i',j',k')\in\mathcal{I}_m:Z_m\in\Sigma_{ijk}\cap\Sigma_{i'j'k'}\right\}.\label{multiple collisions}
\end{align}
Notice that in the special case $m=3$, we have 
$\partial_{mc}\mathcal{D}_{3,\sigma}=\emptyset$ and $\partial\mathcal{D}_{3,\sigma}=\partial_{sc}\mathcal{D}_{3,\sigma},$
i.e. there are no multiple collisions when we consider only three particles.

\begin{definition} Consider  $Z_m\in\partial_{sc}\mathcal{D}_{m,\sigma}$. Then there is a unique triplet $(i,j,k)\in\mathcal{I}_m$ such that $Z_m\in\Sigma_{ijk}$. In this case we will say that $Z_m$ is an $(i;j,k)$ simple collision and we will write
\begin{equation}\label{simple collision surfaces}
\Sigma_{ijk}^{sc}:=\left\{Z_m=(X_m,V_m)\in\partial_{sc}\mathcal{D}_{m,\sigma}: Z_m\mbox{ is $(i;j,k)$ simple collision}\right\}.
\end{equation} 
\end{definition}
\begin{remark} Notice that $\Sigma_{ijk}^{sc}\cap\Sigma_{i'j'k'}^{sc}=\emptyset,\quad\forall (i,j,k)\neq(i',j',k')\in\mathcal{I}_m,$ and $\partial_{sc}\mathcal{D}_{m,\sigma}$ decomposes to
$\partial_{sc}\mathcal{D}_{m,\sigma}=\bigcup_{(i,j,k)\in\mathcal{I}_m}\Sigma_{ijk}^{sc}$.
\end{remark}

For the purposes of defining a global flow, throughout this section we use the following notation:
\begin{definition} Let  $(i,j,k)\in\mathcal{I}_m$ and $Z_m\in\Sigma_{ijk}^{sc}$. We introduce 
\begin{equation}\label{def of omega}
(\widetilde{\omega}_{1},\widetilde{\omega}_2):=\frac{1}{\sqrt{2}\sigma}\left(x_j-x_i,x_k-x_i\right)\in\mathbb{S}_1^{2d-1}.
\end{equation}
 Therefore, each $(i;j,k)$ simple collision naturally induces impact directions $(\widetilde{\omega}_{1},\widetilde{\omega}_{2})\in\mathbb{S}_1^{2d-1}$, and a collisional transformation $T_{\widetilde{\omega}_{1},\widetilde{\omega}_{2}}$.
\end{definition}

We also give the following definition:
\begin{definition}
Let $(i,j,k)\in\mathcal{I}_m$ and $Z_m=(X_m,V_m)\in\Sigma_{ijk}^{sc}$. We denote
$Z_m^*=(X_m,V_m^*),$
where
\begin{equation*}
V_m^*=(v_1,...,v_{i-1},v_i^*,v_{i+1},...,v_{j-1},v_j^*,v_{j+1},...,v_{k-1},v_k^*,v_{k+1},...,v_{m}),
\end{equation*}
and
$(v_i^*,v_j^*,v_k^*)=T_{\widetilde{\omega}_{1},\widetilde{\omega}_{2}}(v_i,v_j,v_k),\quad (\widetilde{\omega}_{1},\widetilde{\omega}_{2})\in\mathbb{S}_1^{2d-1}\text{ are given by \eqref{def of omega}}.$
\end{definition}
\subsection{Classification of simple collisions}
Now, we classify simple collisions in order to eliminate collisions which graze under time evolution. Informally speaking, a simple collisional configuration will be precollisional when the three interacting particles have the velocities which led them to the interaction  and postcollisional when the velocities have already changed by the collision according to the transformation \eqref{formulas of collision}. As we will see in Lemma \ref{elem dyn step}, a simple collisional configuration can be characterized by the sign of the cross-section. More specifically, we introduce the following language:
\begin{definition}\label{collision class}
Let  $(i,j,k)\in\mathcal{I}_m$ and $Z_m\in\Sigma_{ijk}^{sc}$. The configuration $Z_m$ is called:\\
{\it(i)} pre-collisional when $b(\widetilde{\omega}_{1},\widetilde{\omega}_{2},v_j-v_i,v_k-v_i)<0,$\\
{\it(ii)} post-collisional when $b(\widetilde{\omega}_{1},\widetilde{\omega}_{2},v_j-v_i,v_k-v_i)>0,$\\
{\it(iii)} grazing when $b(\widetilde{\omega}_{1},\widetilde{\omega}_{2},v_j-v_i,v_k-v_i)=0,$\\
where $(\widetilde{\omega}_{1},\widetilde{\omega}_{2})\in\mathbb{S}_1^{2d-1}$ is given by \eqref{def of omega} and $b$ is given by \eqref{cross}. 
\end{definition}
\begin{remark}\label{remark on pre-post}
 Let  $(i,j,k)\in\mathcal{I}_m$ and $Z_m\in\Sigma_{ijk}^{sc}$. Using \eqref{symmetrid}, we obtain the following:\\
{\it(i)} $Z_m$ is pre-collisional iff $Z_m^*$ is post-collisional.\\
{\it(ii)} $Z_m$ is post-collisional iff $Z_m^*$ is pre-collisional.\\
{\it(iii)} $Z_m=Z_m^*$ iff $Z_m$ is grazing.
\end{remark}

We consider the subset of the phase space:
$
\mathcal{D}_{m,\sigma}^*=\mathring{\mathcal{D}}_{m,\sigma}\cup\partial_{sc,ng}\mathcal{D}_{m,\sigma},
$
where
$$\partial_{sc,ng}\mathcal{D}_{m,\sigma}=\left\{Z_m=(X_m,V_m)\in\partial_{sc}\mathcal{D}_{m,\sigma}: Z_m\text{ is non-grazing}\right\}.$$
Notice that $\mathcal{D}_{m,\sigma}^*$ is a full measure subset of $\mathcal{D}_{m,\sigma}$ and $\partial_{sc,ng}\mathcal{D}_{m,\sigma}$ is a full surface measure subset of $\partial\mathcal{D}_{m,\sigma}$.
 \subsection{Construction of the local flow}
  Here, we show that each $Z_m\in\mathcal{D}_{m,\sigma}^*$ follows a well-defined trajectory for short time.
  Next Lemma defines the  flow for any initial configuration  in $\mathcal{D}_{m,\sigma}^*$ up to the first collision time.
\begin{lemma}\label{elem dyn step}
Consider  $Z_m=(X_m,V_m)\in\mathcal{D}_{m,\sigma}^*$. Then there is a time $\tau^1_{Z_m}\in (0,\infty]$ such that defining $Z_m(\cdot):(0,\tau_{Z_m}^1]\to\mathbb{R}^{2dm}$ by:
$$Z_m(t)=\left(X_m\left(t\right),V_m\left(t\right)\right):=
\begin{cases}
(X_m+tV_{m},V_m),\quad\text{if $Z_m$ is non-collisional or post-collisional},\\
(X_m+tV_m^*,V_m^*),\quad\text{if $Z_m$ is pre-collisional},
\end{cases}
$$
 the following hold:\\
{\it(i)} $Z_m(t)\in\mathring{\mathcal{D}}_{m,\sigma},\quad\forall t\in (0,\tau^1_{Z_m})$,\\
{\it(ii)} if $\tau^1_{Z_m}<\infty$, then $Z_m(\tau_{Z_m}^{1})\in\partial\mathcal{D}_{m,\sigma}$,\\
{\it(iii)} If $Z_m\in\Sigma_{ijk}^{sc}$ for some $(i,j,k)\in\mathcal{I}_m$, and $\tau_{Z_m}^1<\infty$, then $Z_m(\tau^{1}_{Z_m})\notin\Sigma_{ijk}$,\\
The time $\tau_{Z_m}^1$ is called the first (forward) collision time of $Z_m$. The first negative collision time can be defined analogously.
\end{lemma}
\begin{proof}
Let us make the convention
$\inf\emptyset=+\infty$.
We define
\begin{equation}\label{first collision time}
\begin{aligned}
\tau_{Z_m}^1=
\begin{cases}
\inf\left\{t>0:X_m+tV_m\in\partial\mathcal{D}_{m,\sigma}\right\},\quad\text{if $Z_m$ is post-collisional} , \\
\inf\left\{t>0:X_m+tV_m^*\in\partial\mathcal{D}_{m,\sigma}\right\},\quad\text{if $Z_m$ is pre-collisional}.
\end{cases}
\end{aligned}
\end{equation}
$\bullet$ Assume that $Z_m\in\mathring{\mathcal{D}}_{m,\sigma}$. Since $\mathring{\mathcal{D}}_{m,\sigma}$ is open and the free flow is continuous, we obtain $\tau_{Z_m}^1>0$, and claims \textit{(i)}-\textit{(ii)} follow immediately from \eqref{first collision time}.\\
 $\bullet$ Assume now that $Z_m\in\partial_{sc,ng}\mathcal{D}_{m,\sigma}$, hence $Z_m$ is a simple non-grazing collision. Therefore we may distinguish the following cases:
 
 (I) $Z_m$ is an $(i;j,k)$ post-collisional configuration: For any $t>0$, we have
 \begin{align} 
 d^2&(x_i+tv_i;x_j+tv_j,x_k+tv_k)
\geq 2\sigma^2+2tb(x_j-x_i,x_k-x_i,v_j-v_i,v_k-v_i)>2\sigma^2,\label{ijk>sigma}
\end{align}
since $b(\widetilde{\omega}_{1},\widetilde{\omega}_{2},v_j-v_i,v_k-v_i)>0$. This inequality and the fact that $Z_m$ is simple collision imply that $\tau_{Z_m}^1>0$, and claim \textit{(i)} holds. Claim \textit{(ii)} follows from \eqref{first collision time} and claim \textit{(iii)} follows from \eqref{ijk>sigma}.

(II) $Z_m$ is  $(i;j,k)$ pre-collisional: We use the same argument for $Z_m^*$ which is $(i;j,k)$ post-collisional.
\end{proof}
Let us make an elementary, but crucial remark which will turn of fundamental importance when extending the  flow globally in time.
\begin{remark}\label{remark on t1-t2} For configurations with $\tau_{Z_m}^1=\infty$ the  flow is globally defined as the free flow. In the case
 where $\tau_{Z_m}^1<\infty$ and $Z_m(\tau_{Z_m}^{1})\in\partial_{sc,ng}\mathcal{D}_{m,\sigma}$, we may apply Lemma \ref{elem dyn step} once more, considering $Z_m(\tau_{Z_m}^1)$ as initial point, and extend the  flow up to the second collision time
 $\tau_{Z_m}^2:=\tau_{Z_{m}(\tau_{Z_m}^1)}^1.$
 Moreover, if $Z_m(\tau_{Z_m}^1)\in\Sigma_{ijk}^{sc}$ for some $(i,j,k)\in\mathcal{I}_m$, part \textit{(iii)} of Lemma \ref{elem dyn step} implies that $Z_m(\tau_{Z_m}^{2})\notin\Sigma_{ijk}.$ 
\end{remark}
\subsection{Extension to a global interaction flow}
Now, we extract a null set from $\mathcal{D}_{m,\sigma}^*$ such that the  flow is globally defined for positive times on the complement.  For this purpose, we consider truncation parameters in the scaling:
\begin{equation}\label{scaling dynamics}
0<\delta R<<\sigma<<1<R<\rho.
\end{equation}

We first assume initial positions are in $B_\rho^{dm}$ and initial velocities in $B_R^{dm}$. 
 We decompose $D_{m,\sigma}^*\cap(B_\rho^{dm}\times B_R^{dm})$ as follows:
 {\small
\begin{equation}\label{set decomposition of the first collision time}
\begin{aligned}
&I_{free}:=\left\{Z_m=(X_m,V_m)\in D_{m,\sigma}^*\cap(B_\rho^{dm}\times B_R^{dm}): \tau_{Z_m}^1>\delta\right\},\\
&I^1_{sc,ng}:=\big\{Z_m=(X_m,V_m)\in D_{m,\sigma}^*\cap(B_\rho^{dm}\times B_R^{dm}): \tau_{Z_m}^1<\delta,\quad Z_m(\tau_{Z_m}^{1})\in\partial_{sc,ng}\mathcal{D}_{m,\sigma},\text{ and}\quad \tau_{Z_m}^2>\delta\big\},\\
&I^1_{sc,g}:=\big\{Z_m=(X_m,V_m)\in D_{m,\sigma}^*\cap(B_\rho^{dm}\times B_R^{dm}): \tau_{Z_m}^1<\delta,\text{ } Z_m(\tau_{Z_m}^{1})\in\partial_{sc}\mathcal{D}_{m,\sigma},\text{ but $Z_m(\tau_{Z_m}^{1})$ is grazing}\big\},\\
&I^1_{mc}:=\left\{Z_m=(X_m,V_m)\in D_{m,\sigma}^*\cap(B_\rho^{dm}\times B_R^{dm}): \tau_{Z_m}^1<\delta,\quad Z_m(\tau_{Z_m}^{1})\in\partial_{mc}\mathcal{D}_{m,\sigma}\right\},\\
&I^2_{sc,ng}:=\big\{Z_m=(X_m,V_m)\in D_{m,\sigma}^*\cap(B_\rho^{dm}\times B_R^{dm}): \tau_{Z_m}^1<\delta,\quad Z_m(\tau_{Z_m}^{1})\in\partial_{sc,ng}\mathcal{D}_{m,\sigma},\text{ but}\quad  \tau_{Z_m}^2\leq\delta\big\}.
\end{aligned}
\end{equation}}
Notice that for $Z_m\in I_{free}\cup I^1_{sc,ng}$, thanks to Lemma \ref{elem dyn step}, the  flow is well defined up to time $\delta$, and there occurs at most one simple non-grazing collision in $(0,\delta)$. 

\subsubsection{Covering arguments} Now, we make an ellipsoid shell covering of the set $I^1_{mc}\cup I^2_{sc,ng}$ in a way that we can estimate the measure of the coverings. 

\begin{lemma}\label{covering} For $m=3$, there holds $I^1_{mc}=I^2_{sc,ng}=\emptyset$. For $m\geq 4$, the following inclusion holds:
\begin{equation}\label{inclusion in shells}
I^2_{sc,ng}\cup I^1_{mc} \subseteq\bigcup_{(i,j,k)\neq (i',j',k')\in\mathcal{I}_m}\left(U_{ijk}\cap U_{i'j'k'}\right),
\end{equation}
\begin{equation}\label{U_ijk}
U_{ijk}:=\left\{Z_m=(X_m,V_m)\in B_\rho^{dm}\times B_R^{dm}:
2\sigma^2\leq d^2(x_i;x_j,x_k)\leq (\sqrt{2}\sigma+4\delta R)^2\right\}.
\end{equation}
\end{lemma}
\begin{proof} 
For $m=3$,  we have  $\partial_{mc}\mathcal{D}_{3,\sigma}=\emptyset$, thus $I^{1}_{mc}=\emptyset.$
Also, since $m=3$, we  obtain $\mathcal{I}_3=\{(1,2,3)\}$, hence Remark \ref{remark on t1-t2} implies that $\tau_{Z_m}^2=\infty$ i.e. there is no other collision in the future, so $I^2_{sc,ng}=\emptyset$.

Let $m\geq 4$. We first assume that either $Z_m\in\mathring{\mathcal{D}}_{m,\sigma}$ or  $Z_m$ is post-collisional. We first prove the inclusion for $I_{sc,ng}^2$.
Assuming that $Z_m(\tau_{Z_m}^{1})\in I_{sc,ng}^2$ is an $(i;j,k)$ non-grazing collision, we have
$$d^2\left(x_i\left(\tau_{Z_m}^{1}\right); x_j\left(\tau_{Z_m}^{1}\right),x_k\left(\tau_{Z_m}^{1}\right)\right)=2\sigma^2.$$
Since there is free motion up to $\tau_{Z_m}^1$ and $\tau_{Z_m}^1\leq\delta$, triangle inequality implies
\begin{equation}\label{fremotion step 1}
|x_i-x_j|\leq |x_i(\tau_{Z_m}^{1})-x_j(\tau_{Z_m}^{1})|+\delta |v_i-v_j|\leq |x_i(\tau_{Z_m}^{1})-x_j(\tau_{Z_m}^{1})|+2\delta R.
\end{equation}
Since there is collision at $\tau_{Z_m}^1$, we have
\begin{equation}\label{sub collision}
|x_i(\tau_{Z_m}^{1})-x_j(\tau_{Z_m}^{1})|^2+|x_i(\tau_{Z_m}^{1})-x_k(\tau_{Z_m}^{1})|^2=2\sigma^2\Rightarrow |x_i(\tau_{Z_m}^{1})-x_j(\tau_{Z_m}^{1})|\leq\sqrt{2}\sigma.
\end{equation}
Combining \eqref{fremotion step 1}-\eqref{sub collision}, we obtain
\begin{equation}\label{fremotion step 2}
|x_i-x_j|^2\leq  |x_i(\tau_{Z_m}^{1})-x_j(\tau_{Z_m}^{1})|^2 +4\sqrt{2}\sigma\delta R+4\delta^2 R^2.
\end{equation}
Using the same argument for the pair $(i,k)$, adding,  and recalling the fact that there is $(i;j,k)$ simple collision at $\tau_{Z_m}^1$, we obtain
\begin{equation}\label{free motion}
2\sigma^2\leq d^2(x_i;x_j,x_k)\leq 2\sigma^2+8\sqrt{2}\sigma R\delta+8\delta R^2\leq (\sqrt{2}\sigma+4\delta R)^2,
\end{equation} where the lower inequality holds trivially since $Z_m\in\mathcal{D}_{m,\sigma}$. By \eqref{free motion}, we obtain $Z_m\in U_{ijk}$.

Remark \ref{remark on t1-t2} guarantees that $Z_m(\tau_{Z_m}^{2})\notin\Sigma_{ijk}$. So  $Z_m(\tau_{Z_m}^{2})\in\Sigma_{i'j'k'}$ for some $(i',j',k')\neq (i,j,k)$. 
Moreover, particles keep performing free motion in $[\tau_{Z_m}^1,\tau_{Z_m}^2)$ except particles $i,j,k$ whose velocities instantaneously transform because of the collision at $\tau_{Z_m}^1$. 
Recall we wish to prove as well:
\begin{equation}\label{shell 2} Z_m\in U_{i'j'k'}\Leftrightarrow 2\sigma^2\leq d^2(x_{i'};x_{j'},x_{k'})\leq (\sqrt{2}\sigma+4\delta R)^2.\end{equation}
The lower inequality trivially holds because of the phase space so it suffices to prove the upper inequality. Since $(i,j,k)\neq (i',j',k')$, it suffices to distinguish the following cases:\\
(I) $i',j',k'\notin\{i,j,k\}$:  Since particles $(i',j',k')$ perform free motion up to $\tau_{Z_m}^2$, a similar argument to the one we used to obtain \eqref{free motion} yields $Z_m\in U_{i'j'k'}$. The only difference is that we apply the argument up to time $\tau_{Z_m}^2$.\\
(II) At least one of $i',j',k'$ belongs to $\{i,j,k\}$ but no more than two. The argument is similar to \text{(I)}, the only difference being that  velocities of the recolliding particles  transform at $\tau_{Z_m}^1$. Since the argument is similar for all cases, let us provide the proof in detail only for one case, for instance $(i',j',k')=(i,k,k')$,  for some $k'>k$.  The fact that $V_m\in B_R^{dm}$, conservation of energy  by the free flow and conservation of energy by the collision \eqref{conservation of energy tran} imply
$
v_i^*\left(\tau_{Z_m}^1\right),v_j^*\left(\tau_{Z_m}^1\right),v_k^*\left(\tau_{Z_m}^1\right)\in B_R^{d}.
$
For the pair $(i,k)$, we have
\begin{align*}
x_i(\tau_{Z_m}^{2})&=x_i(\tau_{Z_m}^{1})+(\tau_{Z_m}^2-\tau_{Z_m}^1)v_i^*\left(\tau_{Z_m}^1\right)=x_i+\tau_{Z_m}^1v_i+(\tau_{Z_m}^{2}-\tau_{Z_m}^1)v_i^*\left(\tau_{Z_m}^1\right),\\
x_k(\tau_{Z_m}^{2})&=x_k(\tau_{Z_m}^{1})+(\tau_{Z_m}^2-\tau_{Z_m}^1)v_k^*\left(\tau_{Z_m}^1\right)=x_k+\tau_{Z_m}^1v_k+(\tau_{Z_m}^2-\tau_{Z_m}^1)v_k^*\left(\tau_{Z_m}^1\right).
\end{align*}
Therefore, triangle inequality implies
\begin{align}
|x_i-x_k|&\leq |x_i(\tau_{Z_m}^{2})-x_k(\tau_{Z_m}^{2})|+\tau_{Z_m}^1|v_i-v_k|+(\tau_{Z_m}^2-\tau_{Z_m}^1)|v_i^*\left(\tau_{Z_m}^1\right)-v_k^*\left(\tau_{Z_m}^1\right)|\nonumber\\
&\leq |x_i(\tau_{Z_m}^{2})-x_k(\tau_{Z_m}^{2})|+2\tau_{Z_m}^1 R+2(\tau_{Z_m}^2-\tau_{Z_m}^1)R\nonumber\\
&\leq |x_i(\tau_{Z_m}^{2})-x_k(\tau_{Z_m}^{2})|+2\delta R\nonumber.
\end{align}
Similarly, for the pair $(i,k')$, we obtain
$
|x_i-x_{k'}|\leq |x_i(\tau_{Z_m}^{2})-x_{k'}(\tau_{Z_m}^{2})|+2\delta R.
$
By an argument similar to \eqref{free motion}, inequality \eqref{shell 2} follows. 
Inclusion \eqref{inclusion in shells} is proved for $I^2_{sc,ng}$. The inclusion for $I_{mc}^1$ follows similarly.

 Assume now that $Z_m$ is pre-collisional. By Remark \ref{remark on pre-post}, $Z_m^*$ is post-collisional and by  \eqref{conservation of energy tran}  $Z_{m}^*\in B_{\rho}^{dm}\times B_{R}^{dm}$. By a similar argument to  the post-collisional case, we obtain the result.
\end{proof}
\subsubsection{Measure estimates} 
Now we  estimate the measure of $I_\delta^1\cup I^1_{sc,g}\cup I^1_{mc}\cup I^2_{sc,ng}$ in order to show that outside of a small measure set we have a well defined  flow up to small time $\delta$. To estimate the measure of  $I^1_{mc}\cup I^2_{sc,ng}$, we will strongly rely on the shell-like covering made in Lemma \ref{covering}.

For this purpose, we first introduce some notation. Consider $(i,j,k)\in\mathcal{I}_m$, a permutation $\pi:\{i,j,k\}\to\{i,j,k\}$ and $(x_{\pi_j},x_{\pi_k})\in\mathbb{R}^{2d}$. We define the set
\begin{equation}\label{S perm}
S_{\pi_i}(x_{\pi_j},x_{\pi_k})=\{x_{\pi_i}\in\mathbb{R}^d: (x_i,x_j,x_k)\in U_{ijk}\}.
\end{equation}
\begin{lemma}\label{estimate of S} Let $(i,j,k)\in\mathcal{I}_m$, a permutation $\pi:\{i,j,k\}\to\{i,j,k\}$ and $(x_{\pi_j},x_{\pi_k})\in\mathbb{R}^{2d}$.   Then 
\begin{equation}\label{estimate of S perm}
|S_{\pi_i}(x_{\pi_j},x_{\pi_k})|_d\leq C_{d,R}\delta.
\end{equation}
\begin{proof} By symmetry, it suffices to prove \eqref{estimate of S perm} for the permutations $\pi=(i,j,k)$ and $\pi=(k,i,j)$. 
For convenience, let us write $\sigma_0=\sqrt{2}\sigma$, $\delta_0=4\delta R$. Scaling \eqref{scaling dynamics} implies
$0<\delta_0<<\sigma_0<<1.$

\textbf{The proof for $\pi=(k,i,j)$}: Consider $(x_i,x_j)\in\mathbb{R}^{2d}$, and let us write $\alpha=|x_i-x_j|$. Recalling \eqref{S perm}, we have
$S_k(x_i,x_j)=\left\{x_k\in\mathbb{R}^d:\sigma_0^2-\alpha^2\leq |x_i-x_k|^2\leq (\sigma_0+\delta_0)^2-\alpha^2\right\}.$
We distinguish the following cases:\\
$\bullet$ $\alpha>\sigma_0$: We have $
(\sigma_0+\delta_0)-\alpha^2<(\sigma_0+\delta_0)^2-\sigma_0^2=\delta_0(2\sigma_0+\delta_0)<\delta_0,
$
since $0<\delta_0<<\sigma_0<<1$. Thus
$S_k(x_i,x_j)\subseteq\left\{x_k\in\mathbb{R}^d:|x_i-x_k|\leq \sqrt{\delta_0}\right\},$
so $
|S_k(x_i,x_j)|_d\lesssim \delta_0^{d/2}\leq\delta_0= 4R\delta,
$
since $\delta_0<1$ and $d\geq 2$.\\
$\bullet$ $\alpha\leq\sigma_0$: By \eqref{S perm},  
$S_k(x_i,x_j)=\left\{x_k\in\mathbb{R}^d:\sqrt{\sigma_0^2-\alpha^2}\leq |x_i-x_k|\leq\sqrt{(\sigma_0+\delta_0)^2-\alpha^2}\right\}.$
 Therefore
\begin{align}
&|S_k(x_i,x_j)|_d\simeq \left(\sqrt{(\sigma_0+\delta_0)^2-\alpha^2}\right)^d-\left(\sqrt{\sigma_0^2-\alpha^2}\right)^d\\\nonumber
&=\frac{\delta_0(2\sigma_0+\delta_0)}{\sqrt{(\sigma_0+\delta_0)^2-\alpha^2}+\sqrt{\sigma_0^2-\alpha^2}}\sum_{m=0}^{d-1}\left(\sqrt{(\sigma_0+\delta_0)^2-\alpha^2}\right)^{d-1-m}\left(\sqrt{\sigma_0^2-\alpha^2}\right)^m\nonumber\\
&\leq\frac{\delta_0}{\sqrt{(\sigma_0+\delta_0)^2-\alpha^2}+\sqrt{\sigma_0^2-\alpha^2}}\left(\sqrt{(\sigma_0+\delta_0)^2-\alpha^2}+(d-1)\sqrt{\sigma_0^2-\alpha^2}\right)\label{remove exponents}\\
&\leq (d-1)\delta_0=4(d-1)R\delta\label{S_k part 2},
\end{align}
where to obtain \eqref{remove exponents} we use the fact that $0<\delta_0<<\sigma_0<<1$, and to obtain \eqref{S_k part 2}  we use the fact that $d\geq 2$.
Estimate \eqref{estimate of S perm} is proved for the case $(k,i,j)$.

\textbf{The proof for $\pi=(i,j,k)$:} Consider $(x_j,x_k)\in\mathbb{R}^{2d}$. Completing the square, one can see that 
$$S_i(x_j,x_k)=\left\{x_i\in\mathbb{R}^d:\sigma_0^2-\alpha^2\leq \left|x_i-\frac{x_j+x_k}{2}\right|^2\leq (\sigma+\delta_0)^2-\alpha^2\right\},$$
where
$\sigma_0=\sigma,\quad\delta_0=\frac{4\delta R}{\sqrt{2}},\quad\alpha=\frac{1}{2}\sqrt{2(|x_j|^2+|x_k|^2)-|x_j+x_k|^2}.$
Scaling \eqref{scaling dynamics} implies
$0<\delta_0 <<\sigma_0<<1.$
The estimate follows by an argument identical to the the previous case.
\end{proof}
\end{lemma}

\begin{lemma}\label{ellipse shell measure}
The following measure estimate holds:
\begin{equation*}|I^1_{sc,g}\cup  I^2_{sc,ng}\cup I^1_{mc}|_{2dm}\leq C_{m,d,R}\rho^{d(m-2)}\delta ^2. \end{equation*}
\end{lemma}
\begin{proof} First, we notice that $I^1_{sc,g}$ has measure zero since it is covered by  codimension-$2$ submanifolds of the phase space.
For $m=3$, the result comes trivially from Lemma \ref{covering}.
 Assume $m\geq 4$. By Lemma \ref{covering},  it suffices to uniformly estimate the measure of $U_{ijk}\cap U_{i'j'k'}$, for all  $(i,j,k)\neq (i',j',k')\in\mathcal{I}_m$. 
    Consider $(i,j,k)\neq (i',j',k')\in\mathcal{I}_m$, and recall notation from \eqref{S perm}. We will strongly rely on Lemma \ref{estimate of S}.
  We distinguish the following  cases:\\
  \\
 {(I)} $i',j',k'\notin\{i,j,k\}$:  Fubini's Theorem and \eqref{estimate of S perm} imply
\begin{equation*}
\begin{aligned}
&|U_{ijk}\cap U_{i'j'k'}|_{2dm}\lesssim R^{dm}\rho^{d(m-6)}\int_{B_{\rho}^{6d}}\mathds{1}_{S_{k}(x_i,x_j)\cap S_{k'}(x_{i'},x_{j'})}\,dx_i\,dx_{j}\,dx_k\,dx_{i'}\,dx_{j'}\,dx_{k'}\\
&\leq R^{dm}\rho^{d(m-6)}\left(\int_{B_\rho^d\times B_\rho^d}\int_{\mathbb{R}^{d}\hspace{-0.1cm}}\mathds{1}_{S_k(x_i,x_j)}\,dx_k\,dx_j\,dx_i\right)\left(\int_{B_\rho^d\times B_\rho^d}\int_{\mathbb{R}^{d}}\hspace{-0.1cm}\mathds{1}_{S_{k'}(x_{i'},x_{j'})}\,dx_k'\,dx_j'\,dx_i'\right)\hspace{-0.1cm}\\
&\leq C_{m,d,R}\rho^{d(m-2)}\delta^2.
\end{aligned}
\end{equation*}
{(II)} Exactly one of $i',j',k'$ belongs to $\{i,j,k\}$: Without loss of generality, we consider the case $i'=i, j'\neq j, k\neq k'$. Fubini's Theorem and \eqref{estimate of S perm} imply
\begin{equation*}
\begin{aligned}
&|U_{ijk}\cap U_{ij'k'}|_{2dm}\lesssim R^{dm}\rho^{d(m-5)}\int_{B_\rho^{5d}}\mathds{1}_{S_{k}(x_i,x_j)\cap S_{k'}(x_i,x_{j'})}\,dx_i\,dx_{j}\,dx_{k}\,dx_{j'}\,dx_{k'}\\
&\leq R^{dm}\rho^{d(m-5)}\int_{B_\rho^d}\left(\int_{B_\rho^d}\int_{\mathbb{R}^d}\mathds{1}_{S_k(x_i,x_j)}\,dx_k\,dx_j\right)\left(\int_{B_\rho^d}\int_{\mathbb{R}^d}\mathds{1}_{S_{k'}(x_i,x_{j'})}\,dx_{k'}\,dx_{j'}\right)\,dx_i\\
&\leq C_{m,d,R}\rho^{d(m-2)}\delta^2.
\end{aligned}
\end{equation*}
{(III)} Exactly two of $i',j',k'$ belong to $\{i,j,k\}$: Without loss of generality, we consider the case $i'=i, j'= j, k\neq k'$. Fubini's Theorem and \eqref{estimate of S perm} imply
\begin{equation*}
\begin{aligned} 
&|U_{ijk}\cap U_{ijk'}|_{2dm}\lesssim R^{dm}\rho^{d(m-4)}\int_{B_\rho^{4d}}\mathds{1}_{S_{k}(x_i,x_j)\cap S_{k'}(x_i,x_{j'})}\,dx_i\,dx_{j}\,dx_k\,dx_{k'}\\
&\leq R^{dm}\rho^{d(m-4)}\int_{B_{\rho}^d\times B_\rho^d}\left(\int_{\mathbb{R}^d}\mathds{1}_{S_k(x_i,x_j)\,dx_k}\right)\left(\int_{\mathbb{R}^d}\mathds{1}_{S_{k'}(x_i,x_j)\,dx_{k'}}\right)\,dx_j\,dx_i\\
&\leq C_{m,d,R}\rho^{d(m-2)}\delta^2.
\end{aligned}
\end{equation*}
	\end{proof}
	\begin{remark}\label{remark for negative} For negative times, analogous results of Lemma \ref{covering}, Lemma \ref{ellipse shell measure} follow similarly.
	\end{remark}
\subsubsection{The global interaction  flow}
We inductively use Lemma \ref{ellipse shell measure} to define a global  flow which preserves energy for almost all configuration. For this purpose, given $Z_m=(X_m,V_m)\in\mathbb{R}^{2dm}$, we define its kinetic energy as:
\begin{equation}\label{kinetic energy}
E_m(Z_m):=\frac{1}{2}\sum_{i=1}^m|v_i|^2.
\end{equation}
For convenience, let us define the free flow of $m$-particles.
\begin{definition} Let $m\in\mathbb{N}$. We define the free flow of $m$-particles  as the family of maps $(\Phi_m^t)_{t\in\mathbb{R}}:\mathbb{R}^{2dm}\to\mathbb{R}^{2dm}$, given by
$
\Phi_m^tZ_m=\Phi_m^t(X_m,V_m):=(X_m+tV_m,V_m).
$
\end{definition}
We establish the  existence of $\sigma$-interaction zone  flow of $m$-particles.
\begin{theorem}[Existence of the interaction flow]\label{global flow}
Let $m\in\mathbb{N}$ and $0<\sigma<<1$. There exists a full measure $G_\delta$-subset $\Gamma_{m,\sigma}\subseteq\mathcal{D}_{m,\sigma}^*$ and a measure-preserving family of diffeomorphisms $(\Psi_m^t)_{t\in\mathbb{R}}:\Gamma_{m,\sigma}\to\Gamma_{m,\sigma}$ such that 
\begin{align}
&\Psi_m^{t+s}Z_m=(\Psi_m^t\circ \Psi_m^s)(Z_m)=(\Psi_m^s\circ \Psi_m^t)(Z_m),\quad\forall Z_m\in\Gamma_{m,\sigma},\quad\forall t,s\in\mathbb{R}\label{flow property},\\
&E_m\left(\Psi_m^t Z_m\right)=E_m(Z_m),\quad\forall Z_m\in\Gamma_{m,\sigma},\quad\forall t\in\mathbb{R}\label{kinetic energy flow}.
\end{align}
Moreover for $m\geq 3$ the flow is defined a.e. on $\Gamma_{m,\sigma}\cap \partial_{sc,ng}\mathcal{D}_{m,\sigma}$ with respect to the induced measure $\,d\sigma$
and
\begin{equation}\label{bc flow}
\Psi_m^t Z_m^*=\Psi_m^t Z_m,\quad\sigma-\text{a.e. on }\Gamma_{m,\sigma}\cap\partial_{sc,ng}\mathcal{D}_{m,\sigma},\quad\forall t\in\mathbb{R}.
\end{equation} 
 This family of maps is called the $\sigma$-interaction zone flow of $m$-particles. For $m=1,2$, the  flow coincides with the free flow.

\end{theorem}
\begin{proof}

Having established the bounds of Lemma \ref{ellipse shell measure}, which are valid for both positive and negative collision times (by Remark \ref{remark for negative}), existence of the set $\Gamma$ and \eqref{flow property}-\eqref{kinetic energy flow} follow in the same spirit as in \cite{alexander}. An outline of the proof can also be found in \cite{gallagher}. For details of the proof, see \cite{thesis}.

 It remains to prove that the flow is a.e. defined on $\Gamma_{m,\sigma}\cap \partial_{sc,ng}\mathcal{D}_{m,\sigma}$ and that \eqref{bc flow} holds. We use an argument similar to \cite{simonella}. By the definition of the flow, \eqref{bc flow}  holds on  $\Gamma_{m,\sigma}\cap\partial_{sc,ng}\mathcal{D}_{m,\sigma}$. Therefore, it suffices to prove $I_{m,\sigma}\cap\partial_{sc,ng}\mathcal{D}_{m,\sigma}$ is a null subset of $\partial_{sc,ng}\mathcal{D}_{m,\sigma}$, where $I_{m,\sigma}:=\mathcal{D}_{m,\sigma}^*\setminus \Gamma_{m,\sigma}$ is the set of configurations which run into pathological trajectories in finite time. 
Let $Z_m'\in \partial_{sc,ng}\mathcal{D}_{m,\sigma}$. Then by Lemma \ref{elem dyn step}, the flow can be defined up to time $\tau^1_{Z_m'}>0$ and $\Psi^t_m Z_m'\in\mathring{\mathcal{D}}_{m,\sigma}$ for all $0<t<\tau^1_{Z_m'}$.  But since $I_{m,\sigma}$ is of measure zero and $\Gamma_{m,\sigma}$ is invariant under the flow, we have
$$0=\int_{I_{m,\sigma}\cap\mathring{\mathcal{D}}_{m,\sigma}}\,dZ_m=\int_{I_{m,\sigma}\cap\partial_{sc,ng}\mathcal{D}_{m,\sigma}}\int_0^{\tau^1_{Z_m'}}\,dt\,d\sigma(Z_m')=\int_{I_{m,\sigma}\cap\partial_{sc,ng}\mathcal{D}_{m,\sigma}}\tau^1_{Z_m'}\,d\sigma(Z_m'),$$
which implies that $\sigma(I_{m,\sigma}\cap\partial_{sc,ng}\mathcal{D}_{m,\sigma})=0$, since $\tau^1_{Z_m'}>0$.
\end{proof}

\subsection{The Liouville equation}
We introduce the flow operators used throughout the paper, and then derive the $m$-particle Liouville equation for $m\geq 3$. 
\begin{definition}
For $t\in\mathbb{R}$, we define the $\sigma$-interaction zone  flow of $m$-particles operator $T_m^t:L^\infty(\mathcal{D}_{m,\sigma})\to L^\infty(\mathcal{D}_{m,\sigma})$ as
\begin{equation}\label{liouville operator}
T_m^tg_m(Z_m)=g_{m}(\Psi_m^{-t}Z_m).
\end{equation}
\end{definition}
\begin{definition}For $t\in\mathbb{R}$ and $m\in\mathbb{N}$, we define the  free flow of $m$-particles operator $S_m^t:L^\infty(\mathbb{R}^{2dm})\to L^\infty(\mathbb{R}^{2dm})$ as:
\begin{equation}\label{free flow operator}
S_m^tg_m(Z_m)=g_m(\Phi_m^{-t}Z_m)=g_m(X_m-tV_m,V_m).
\end{equation}
\end{definition}

Assume $m\geq 3$. Given a symmetric with respect to the particles initial probability density $f_{m,0}$ supported in $\mathcal{D}_{m,\sigma}$, we define its evolution as $f_m(t,Z_m):=T_m^tf_{m,0}$. Clearly, $f_m$ is symmetric and supported in $\mathcal{D}_{m,\sigma}$.  Theorem \ref{global flow} implies that $f_m$ formally satisfies the $m$-particle Liouville equation
\begin{equation}\label{Liouville equation}
\begin{cases}
\partial_tf_m+\displaystyle\sum_{i=1}^m v_i\cdot\nabla_{x_i}f_m=0,\quad(t,Z_m)\in(0,\infty)\times\mathring{\mathcal{D}}_{m,\sigma},\\
f_m(t,Z_m^*)=f_m(t,Z_m),\quad (t,Z_m)\in[0,\infty)\times\partial_{sc}\mathcal{D}_{m,\sigma},\\
f_m(0,Z_m)=f_{m,0}(Z_m),\quad Z_m\in\mathring{\mathcal{D}}_{m,\sigma}.
\end{cases}
\end{equation}

\section{BBGKY hierarchy, Boltzmann hierarchy and the ternary Boltzmann equation}
\label{sec_derivation} 
In this section we consider $N$-particles of $\epsilon$-interaction zone, where $N\geq 3$ and $0<\epsilon<<1$. We integrate the $N$-particle Liouville's equation to formally obtain a linear hierarchy of integro-differential equations satified by the marginals of its solution (BBGKY hierarchy). We then formally derive the limiting hierarchy (Boltzmann hierarchy) occuring under the appropriate scaling and formally show it reduces to a nonlinear integro-differential equation (the new ternary Boltzmann equation) for chaotic initial data. 
 \subsection{The BBGKY hierarchy}
Consider $N$-particles of interaction zone $0<\epsilon<<1$, where $N\geq 3$. For $s\in\mathbb{N}$, we define the $s$-marginal of a symmetric probability density $f_N$, supported in $\mathcal{D}_{N,\epsilon}$, as

\begin{equation}\label{def marginals}
f_N^{(s)}(Z_s)=
\begin{cases}
\displaystyle\int_{\mathbb{R}^{2d\left(N-s\right)}}f_N(Z_N)\,dx_{s+1}...\,dx_N\,dv_{s+1}...\,dv_N,\text{ } 1\leq s< N,\\
f_N,\text{ } s=N,\\
0,\text{ } s>N,
\end{cases}
\end{equation}
where for $Z_s=(X_s,V_s)\in\mathbb{R}^{2ds}$, we write $Z_N=(X_s,x_{s+1},...,x_N,V_s,v_{s+1},...,v_N)$.
It is straightforward that, for all $1\leq s\leq N$, the marginals $f_N^{(s)}$ are symmetric probability densities, supported in $\mathcal{D}_{s,\epsilon}$.

Assume now that $f_N$ is formally the solution to the $N$-particle Liouville equation \eqref{Liouville equation} with initial data $f_{N,0}$. We seek to formally find a hierarchy of equations satisfied by the marginals of $f_N$. The answer is obvious for $s\geq N$ since by definition
$f_N^{(N)}=f_N$ and $f_N^{(s)}= 0$ for $s>N$.

Notice that $\partial \mathcal{D}_{N,\epsilon}$ is equivalent up to surface measure zero to $\Sigma^X\times\mathbb{R}^{dN}$, where 
$\Sigma^X:=\hspace{-0.2cm}\displaystyle\bigcup_{(i,j,k)\in\mathcal{I}_N}\Sigma_{ijk}^{sc,X},$
 and $\Sigma_{ijk}^{sc,X}$ are given by \eqref{simple collision surfaces}. Moreover, $\Sigma^X$ is a pairwise disjoint union.

We proceed  by integrating by parts the Liouville equation. Consider $1\leq s\leq N-1$. The boundary and initial conditions can be easily recovered integrating Liouville's equation boundary and initial conditions respectively i.e. 
\begin{equation}\label{bc of BBGKY}\begin{cases}
 f_N^{(s)}(t,Z_s^*)=f^{(s)}_N(t,Z_s),\quad (t,Z_s)\in[0,\infty)\times\partial_{sc}\mathcal{D}_{s,\epsilon},\quad s\geq 3,\\
f_N^{(s)}(0,Z_s)=f_{N,0}^{(s)}(Z_s),\quad Z_s\in\mathring{\mathcal{D}}_{s,\epsilon}.
\end{cases}
\end{equation}
Notice that for $s=1,2$ there is no boundary condition, since $\mathcal{D}_{s,\epsilon}=\mathbb{R}^{2ds}$ by definition.

Consider now a smooth test function $\phi_s$ compactly supported in $(0,\infty)\times\mathcal{D}_{s,\epsilon}$ such that whenever $(i,j,k)\in \mathcal{I}_N$ with $j\leq s$, the following holds:
\begin{equation}\label{test boundary}\phi_s(t,p_sZ_N^*)=\phi_s(t,p_sZ_N)=\phi_s(t,Z_s),\quad\forall (t,Z_N)\in (0,\infty)\times\Sigma_{ijk}^{sc},\end{equation}
where $p_s(Z_N):=Z_s$ is the natural projection in space and velocities.

 Multiplying the Liouville equation by $\phi_s$, and integrating ,   we obtain 
\begin{equation}\label{weak form initial}
\int_{(0,\infty)\times\mathcal{D}_{N,\epsilon}}\bigg(\partial_tf_N\left(t,Z_N\right)+\sum_{i=1}^Nv_i\cdot\nabla_{x_i}f_N\left(t,Z_N\right)\bigg)\phi_s(t,Z_s)\,dX_N\,dV_N\,dt=0.
\end{equation}
For the time derivative in \eqref{weak form initial}, integration by parts in time, Fubini's Theorem and then again integration by parts in time imply
\begin{equation}\label{time integration by parts}
\int_{(0,\infty)\times\mathcal{D}_{N,\epsilon}}\partial_tf_N(t,Z_N)\phi_s(t,Z_s)\,dX_N\,dV_N\,dt
=\int_{(0,\infty)\times\mathcal{D}_{s,\epsilon}}\partial_tf_N^{(s)}(t,Z_s)\phi_s(t,Z_s)\,dX_s\,dV_s\,dt.
\end{equation}
For the material derivative term in \eqref{weak form initial}, the Divergence Theorem implies 
\begin{align}
\int_{\mathcal{D}_{N,\epsilon}}\sum_{i=1}^Nv_i\cdot\nabla_{x_i}f_N\left(t,Z_N\right)\phi_s(t,Z_s)\,dX_N\,dV_N
&=\int_{\mathcal{D}_{N,\epsilon}}\diverg_{X_N}\left[f_N\left(t,Z_N\right)V_N\right]\phi_s(t,Z_s)\,dX_N\,dV_N\nonumber\\
&=A_1+A_2,\label{first int by parts diverg}
\end{align}
\begin{align*}
A_1&:=-\int_{\mathcal{D}_{N,\epsilon}}V_N\cdot\nabla_{X_N}\phi_s(t,Z_s)f_N(t,Z_N)\,dX_N\,dV_N\\
 A_2&:=\int_{\Sigma^X\times\mathbb{R}^{dN}}\hat{n}\left(X_N\right)\cdot V_Nf_N\left(t,Z_N\right)\phi_s\left(t,Z_s\right)\,dV_N\,d\sigma,
 \end{align*}
where $\hat{n}(X_N)$ is the outwards normal vector on $\Sigma^X$ at $X_N\in\Sigma^X$, $\,d\sigma$ is the surface measure on $\Sigma^X$. Moreover, by the fact that $f_N$ is supported in $\mathcal{D}_{N,\epsilon}$, the Divergence Theorem and the fact that $\phi_s$ is compactly supported, we obtain
\begin{align}
A_1&=\int_{\mathbb{R}^{2dN}}V_s\cdot\nabla_{X_s}\phi_s(t,Z_s)f_N(t,Z_N)\,dX_N\,dV_N=\int_{\mathbb{R}^{2ds}}V_s\cdot\nabla_{X_s}\phi_s(t,Z_s)f_N^{(s)}(t,Z_s)\,dX_s\,dV_s\nonumber\\
&=-\int_{\mathbb{R}^{2ds}}\diverg_{X_s}[f_N^{(s)}(t,Z_s)V_s]\phi_s(t,Z_s)\,dX_s\,dV_s=-\int_{\mathcal{D}_{s,\epsilon}}\sum_{i=1}^sv_i\nabla_{x_i}f_N^{(s)}(t,Z_s)\phi_s(t,Z_s)\,dX_s\,dV_s\label{divergence transport term},
\end{align}
Combining \eqref{weak form initial}-\eqref{first int by parts diverg}, \eqref{divergence transport term},  we obtain 
\begin{align}
&\int_{(0,\infty)\times\mathcal{D}_{s,\epsilon}}\left(\partial_tf_N^{(s)}\left(t,Z_s\right)+\sum_{i=1}^sv_i\cdot\nabla_{x_i}f_N^{(s)}\left(t,Z_s\right)\right)\phi_s\left(t,Z_s\right)\,dX_s\,dV_s\,dt=\sum_{(i,j,k)\in\mathcal{I}_N}\int_0^\infty C_{ijk}(t)\,dt,\label{weak Liouville}
\end{align}
\begin{equation}\label{c_ijk def}
C_{ijk}(t):=-\int_{\Sigma_{i,j,k}^{sc,X}\times\mathbb{R}^{dN}}\hat{n}_{ijk}\left(X_N\right)\cdot V_Nf_N\left(t,Z_N\right)\phi_s\left(t,Z_s\right)\,dV_N\,d\sigma_{ijk},
\end{equation}
and $\hat{n}_{ijk}(X_N)$ is the outwards normal vector on $\Sigma_{ijk}^{sc,X}$ at $X_N\in\Sigma_{ijk}^{sc,X}$, $\,d\sigma_{ijk}$ is the surface measure on $\Sigma_{ijk}^{sc,X}$.
We easily calculate
\begin{equation}\label{dotted normal}
-\hat{n}_{ijk}(X_N)\cdot V_N=(\sqrt{2})^{-1}\frac{\langle \frac{x_j-x_i}{\sqrt{2}\epsilon},v_j-v_i\rangle+\langle \frac{x_k-x_i}{\sqrt{2}\epsilon},v_k-v_i\rangle}{\sqrt{1+\langle \frac{x_j-x_i}{\sqrt{2}\epsilon},\frac{x_k-x_i}{\sqrt{2}\epsilon}\rangle}}.
\end{equation}
Notice that since we are integrating over $\Sigma_{ijk}^{sc,X}$, we have 
$\left(\frac{x_j-x_i}{\sqrt{2}\epsilon},\frac{x_k-x_i}{\sqrt{2}\epsilon}\right)\in\mathbb{S}_1^{2d-1}.$
Making the change of variables $(v_i,v_j,v_k)\to (v_i^*,v_j^*,v_k^*)$, under the collisional transformation induced by $\left(\frac{x_j-x_i}{\sqrt{2}\epsilon},\frac{x_k-x_i}{\sqrt{2}\epsilon}\right)$, using \eqref{dotted normal}, Proposition \ref{operator properties} parts \textit{(iv)}, \textit{(v)}
and the boundary condition of \eqref{Liouville equation}, we obtain
\begin{align}
C_{ijk}(t)
=&-(\sqrt{2})^{-1}\int_{\Sigma_{ijk}^{sc,X}\mathbb{R}^{dN}}\frac{\langle \frac{x_j-x_i}{\sqrt{2}\epsilon},v_j-v_i\rangle+\langle \frac{x_k-x_i}{\sqrt{2}\epsilon},v_k-v_i\rangle}{\sqrt{1+\langle \frac{x_j-x_i}{\sqrt{2}\epsilon},\frac{x_k-x_i}{\sqrt{2}\epsilon}\rangle}} f_N(t,Z_N^*)\phi_s(t,\pi_sZ_N^*)\,dV_N\,d\sigma_{ijk}.\label{C_ijk with -}
\end{align}
Equations \eqref{c_ijk def}-\eqref{C_ijk with -} and the test function condition \eqref{test boundary} imply 
\begin{equation}\label{remaining cases BBGKY}
C_{ijk}(t)= 0,\quad\forall (i,j,k)\notin\widetilde{\mathcal{I}}_N,\quad\forall t>0, \text{ where }\widetilde{\mathcal{I}}_N:=\left\{(i,j,k)\in\mathcal{I}_N:1\leq i\leq s<j<k\leq N\right\}.
\end{equation}
Notice we immediately observe that the $(N-1)$- marginal satisfies the $(N-1)$-Liouville equation given in \eqref{Liouville equation}.

For $1\leq s\leq N-2$ and $(i,j,k)\in \widetilde{\mathcal{I}}_N$,  the $(dN-1)$-surface measure on $\Sigma_{ijk}^{sc,X}$ can be written as
$\,d\sigma_{ijk}(X_N)=\,dS_{x_i}(x_j,x_k)\prod_{\substack {\ell=1\\ \ell\neq j,k}}^N\,dx_\ell,$
where, given $x_i\in\mathbb{R}^d$, $\,dS_{x_i}$ is the surface measure on the sphere of center $(x_i,x_i)\in\mathbb{R}^{2d}$ and radius $\sqrt{2}\epsilon$.
By this  decomposition and  the symmetry assumption on $f_N$ we obtain
$
C_{ijk}(t)=C_{i,s+1,s+2}(t),\quad\forall(i,j,k)\in\widetilde{\mathcal{I}}_N,\quad\forall t>0.
$
This observation and, \eqref{remaining cases BBGKY} yield
\begin{align}
\sum_{(i,j,k)\in\mathcal{I}_N}C_{ijk}(t)&=\sum_{i=1}^s\sum_{j=s+1}^{N-1}\sum_{k=j+1}^NC_{i,s+1,s+2}(t)\nonumber\\
&=\sum_{i=1}^s\sum_{j=s+1}^{N-1}(N-j)C_{i,s+1,s+2}(t)=\left(1+2+...+N-s-1\right)\sum_{i=1}^sC_{i,s+1,s+2}(t)\nonumber\\
&=\frac{1}{2}(N-s)(N-s-1)\sum_{i=1}^sC_{i,s+1,s+2}(t),\quad\forall t>0.\label{sum C_ijk}
\end{align}
Fix $i\in\{1,...,s\}$. Substituting
$
(\omega_1,\omega_2)=\left(\frac{x_{s+1}-x_i}{\sqrt{2}\epsilon},\frac{x_{s+2}-x_i}{\sqrt{2}\epsilon}\right),
$
and recalling the notation from \eqref{cross},
we obtain thanks to \eqref{c_ijk def}-\eqref{dotted normal},  \eqref{def marginals} and the fact that $\supp f_N^{(s+2)}\subseteq\mathcal{D}_{s+2,\epsilon}$ that
\begin{equation}\label{weak BBGKY}
\begin{aligned}
\int_0^\infty C_{i,s+1,s+2}(t)\,dt
=\int_{(0,\infty)\times\mathcal{D}_{s,\epsilon}}2^{d-1}\epsilon^{2d-1}\int_{\mathbb{S}_1^{2d-1}\times\mathbb{R}^{2d}}\frac{b\left(\omega_1,\omega_2,v_{s+1}-v_i,v_{s+2}-v_i\right)}{\sqrt{1+\langle\omega_1,\omega_2\rangle}}\times&\\
\times f_N^{(s+2)}(t,X_s,x_i+\sqrt{2}\epsilon\omega_1,x_i+\sqrt{2}\epsilon\omega_2,V_s,v_{s+1},v_{s+2})\,d\omega_1\,d\omega_2\,dv_{s+1}\,dv_{s+2}\,dX_s\,dV_s\,dt.&
\end{aligned}
\end{equation}
Splitting the cross-section to positive and negative parts, followed by an application of the relevant boundary condition to the positive part, and substituting $(\omega_1,\omega_2)\to (-\omega_1,-\omega_2)$ for the negative part,  the right hand side of \eqref{weak BBGKY} becomes:
\begin{equation}\label{decomposition in +-}
\begin{aligned}
&\int_{(0,\infty)\times\mathcal{D}_{s,\epsilon}}2^{d-1}\epsilon^{2d-1}\int_{\mathbb{S}_1^{2d-1}\times\mathbb{R}^{2d}}b_+\left(\omega_1,\omega_2,v_{s+1}-v_i,v_{s+2}-v_i\right)\\
&\times \bigg(f_N^{(s+2)}(t,Z_{s+2,\epsilon}^{i*})- f_N^{(s+2)}(t,Z_{s+2,\epsilon}^i)\bigg)\,d\omega_1\,d\omega_2\,dv_{s+1}\,dv_{s+2}\,dX_s\,dV_s\,dt,
\end{aligned}
\end{equation}
where given $i\in\{1,...,s\}$, we denote
\begin{equation*}
\begin{aligned}
Z_{s+2,\epsilon}^{i}&=(x_1,...,x_i,...,x_s,x_i-\sqrt{2}\epsilon\omega_1,x_i-\sqrt{2}\epsilon\omega_2,v_1,...v_{i-1},v_i,v_{i+1},...,v_s,v_{s+1},v_{s+2}),\\
 Z_{s+2,\epsilon}^{i*}&=(x_1,...,x_i,...,x_s,x_i+\sqrt{2}\epsilon\omega_1,x_i+\sqrt{2}\epsilon\omega_2,v_1,...v_{i-1},v_i^*,v_{i+1},...,v_s,v_{s+1}^*,v_{s+2}^*).
\end{aligned}
\end{equation*}
Finally, combining \eqref{weak Liouville}, \eqref{sum C_ijk}-\eqref{decomposition in +-}, we  formally obtain the BBGKY hierarchy for $s\in\mathbb{N}$:
\begin{equation}\label{BBGKY}\begin{cases}
\partial_tf_N^{(s)}+\displaystyle\sum_{i=1}^sv_i\cdot\nabla_{x_i}f_N^{(s)}=\mathcal{C}_{s,s+2}^Nf_N^{(s+2)},\quad (t,Z_s)\in (0,\infty)\times\mathring{\mathcal{D}}_{s,\epsilon},\\
f_N^{(s)}(t,Z_s^*)=f_N^{(s)}(t,Z_s),\quad(t,Z_s)\in [0,\infty)\times\partial_{sc}\mathcal{D}_{s,\epsilon},\text{ whenever } s\geq 3,\\
f_N^{(s)}(0,Z_s)=f_{N,0}^{(s)}(Z_s),\quad Z_s\in\mathring{\mathcal{D}}_{s,\epsilon},
\end{cases}
\end{equation}
where
\begin{equation}\label{BBGKY operator}
\mathcal{C}_{s,s+2}^N=\mathcal{C}_{s,s+2}^{N,+}-\mathcal{C}_{s,s+2}^{N,-},
\end{equation}
for $1\leq s\leq N-2$ we denote
\begin{equation}\label{BBGKY operator+}
\begin{aligned}
\mathcal{C}_{s,s+2}^{N,+}f_N^{(s+2)}(t,Z_s)
=A_{N,\epsilon,s}\sum_{i=1}^s&\int_{\mathbb{S}_1^{2d-1}\times\mathbb{R}^{2d}}\frac{b_+}{\sqrt{1+\langle\omega_1,\omega_2\rangle}} f_N^{(s+2)}\left(t,Z_{s+2,\epsilon}^{i*},\right)\,d\omega_1\,d\omega_2\,dv_{s+1}\,dv_{s+2},
\end{aligned}
\end{equation}

\begin{equation}\label{BBGKY operator-}
\begin{aligned}
\mathcal{C}_{s,s+2}^{N,-}f_N^{(s+2)}(t,Z_s)=
A_{N,\epsilon,s}\sum_{i=1}^s&\int_{\mathbb{S}_1^{2d-1}\times\mathbb{R}^{2d}}\frac{b_+}{\sqrt{1+\langle\omega_1,\omega_2\rangle}} f_N^{(s+2)}\left(t,Z_{s+2,\epsilon}^{i}\right)\,d\omega_1\,d\omega_2\,dv_{s+1}\,dv_{s+2},
\end{aligned}
\end{equation}
and we use the notation
 \begin{equation}\label{A}
 \begin{aligned}
 &A_{N,\epsilon,s}=2^{d-2}(N-s)(N-s-1)\epsilon^{2d-1},\\
 &b=b(\omega_1,\omega_2,v_{s+1}-v_i,v_{s+2}-v_i),\quad b_+=\max\{b,0\},\\
 &Z_{s+2,\epsilon}^{i}=(x_1,...,x_i,...,x_s,x_i-\sqrt{2}\epsilon\omega_1,x_i-\sqrt{2}\epsilon\omega_2,v_1,...v_{i-1},v_i,v_{i+1},...,v_s,v_{s+1},v_{s+2}),\\
 &Z_{s+2,\epsilon}^{i*}=(x_1,...,x_i,...,x_s,x_i+\sqrt{2}\epsilon\omega_1,x_i+\sqrt{2}\epsilon\omega_2,v_1,...v_{i-1},v_i^*,v_{i+1},...,v_s,v_{s+1}^*,v_{s+2}^*).
\end{aligned} 
 \end{equation}
For $s\geq N-1$ we trivially define
$\mathcal{C}_{s,s+2}^{N,+}\equiv \mathcal{C}_{s,s+2}^{N,-}\equiv  0.$

Duhamel's formula implies that the BBGKY hierarchy can be written in mild form as follows
 \begin{equation}\label{mild BBGKY sec 4}f_N^{(s)}(t,Z_s)=T_s^tf_{N,0}^{(s)}(Z_s)+\int_0^tT_s^{t-\tau}\mathcal{C}_{s,s+2}^Nf_N^{(s+2)}(\tau,Z_s)\,d\tau,\quad s\in\mathbb{N},\end{equation}
 where $T_s^t$ is the $\epsilon$-interaction zone  flow  of $s$-particles operator given in \eqref{liouville operator}. See Remark \ref{ill def C} for the validity of \eqref{mild BBGKY sec 4} in $L^\infty$.
 \subsection{The Boltzmann hierarchy}
 We will now derive the Boltzmann hierarchy as the formal limit of the BBGKY hierarchy as $N\to\infty$ and $\epsilon\to 0^+$ under the  scaling
 \begin{equation}\label{scaling}
 N\epsilon^{d-1/2}=2^{1-d/2}.
 \end{equation}
 This scaling guarantees that for a fixed $s\in\mathbb{N}$, we have
 $
 A_{N,\epsilon,s}\longrightarrow 1$, as $N\to\infty$ and $\epsilon\to 0^+$ in the scaling \eqref{scaling}.
Formally taking the limit under the scaling imposed we may define the following collisional operator:
\begin{equation}\label{Boltzmann operator}
\mathcal{C}_{s,s+2}^{\infty}=\mathcal{C}_{s,s+2}^{\infty,+}-\mathcal{C}_{s,s+2}^{\infty,-},
\end{equation}
\begin{equation}\label{Boltzmann operator +}
\begin{aligned}
\mathcal{C}_{s,s+2}^{\infty,+}f^{(s+2)}(t,Z_s)=\sum_{i=1}^s\int_{(\mathbb{S}_1^{2d-1}\times\mathbb{R}^{2d})}\frac{b_+}{\sqrt{1+\langle\omega_1,\omega_2\rangle}}f^{(s+2)}\left(t,Z_{s+2}^{i*}\right) \,d\omega_1\,d\omega_2\,dv_{s+1}\,dv_{s+2}&,
\end{aligned}
\end{equation}
\begin{equation}\label{Boltzmann operator -}
\begin{aligned}
\mathcal{C}_{s,s+2}^{\infty,-}f^{(s+2)}(t,Z_s)=\sum_{i=1}^s\int_{(\mathbb{S}_1^{2d-1}\times\mathbb{R}^{2d})}\frac{b_+}{\sqrt{1+\langle\omega_1,\omega_2\rangle}}\times f^{(s+2)}\left(t,Z_{s+2}^{i}\right)\,d\omega_1\,d\omega_2\,dv_{s+1}\,dv_{s+2}&,
\end{aligned}
\end{equation}
and 
\begin{equation}\label{boltzmann notation}
\begin{aligned}
 &b=b(\omega_1,\omega_2,v_{s+1}-v_i,v_{s+2},v_i),\quad b_+=\max\{b,0\},\\
&Z_{s+2}^i=(x_1,...,x_i,...,x_s,x_i,x_i,v_1,...v_{i-1},v_i,v_{i+1},...,v_s,v_{s+1},v_{s+2}),\\
&Z_{s+2}^{i*}=(x_1,...,x_i,...,x_s,x_i,x_i,v_1,...v_{i-1},v_i^*,v_{i+1},...,v_s,v_{s+1}^*,v_{s+2}^*).
\end{aligned}
\end{equation}

Now we are ready to introduce the Boltzmann hierarchy. More precisely, given an initial data $f_0^{(s)}$, the Boltzmann hierarchy for $s\in\mathbb{N}$ is given by:
\begin{equation}\label{Boltzmann hierarchy}\begin{cases}
\partial_tf^{(s)}+\displaystyle\sum_{i=1}^sv_i\nabla_{x_i}f^{(s)}=\mathcal{C}_{s,s+2}^\infty f^{(s+2)},\quad(t,Z_s)\in (0,\infty)\times\mathbb{R}^{2ds},\\
f^{(s)}(0,Z_s)=f_0^{(s)}(Z_s),\quad\forall Z_s\in\mathbb{R}^{2ds}.
\end{cases}
\end{equation}

Duhamel's formula implies that the Boltzmann hierarchy can be written in mild form as follows
 \begin{equation}\label{mild Boltzmann sec 4} f^{(s)}(t,Z_s)=S_s^tf_0^{(s)}(Z_s)+\int_0^tS_s^{t-\tau}\mathcal{C}_{s,s+2}^\infty f^{(s+2)}(\tau,Z_s)\,d\tau,\quad s\in\mathbb{N}.
 \end{equation}
 where $S_s^t$ denotes  free flow of $s$-particles operator given in \eqref{free flow operator}. See Remark \ref{ill def C boltz} for the validity of \eqref{mild Boltzmann sec 4} in $L^\infty$.
\subsection{ The ternary Boltzmann equation}\label{sec equation properties} A situation of particular physical interest is when particles are initially independently distributed. This translates to factorized Boltzmann hierarchy initial data i.e. 
\begin{equation}\label{tensorized initial data}
f_0^{(s)}(Z_s)=f_0^{\otimes s}(Z_s)=\prod_{i=1}^sf_0(x_i,v_i),\quad s\in\mathbb{N},
\end{equation}
where $f_0:\mathbb{R}^{2d}\to\mathbb{R}$ is a given function. One can easily verify that the anszatz 
\begin{equation}\label{propagation of chaos}
f^{(s)}(t,Z_s)=f^{\otimes s}(t,Z_s)=\prod_{i=1}^sf(t,x_i,v_i),\quad s\in\mathbb{N},
\end{equation}
solves the Boltzmann hierarchy with initial data given by \eqref{tensorized initial data}, if 
 $f:[0,\infty)\times\mathbb{R}^{2d}\to\mathbb{R}$ satisfies the ternary Boltzmann equation
\begin{equation}\label{Boltzmann equation}
\begin{cases}
\partial_tf+v\cdot\nabla_xf=Q_3(f,f,f),\quad (t,x,v)\in (0,\infty)\times\mathbb{R}^{2d},\\
f(0,x,v)=f_0(x,v),\quad (x,v)\in\mathbb{R}^{2d},
\end{cases}
\end{equation}
 where, using the notation from \eqref{intro:parameters boltzmann}, the ternary  collisional operator $Q_3$ is given by \eqref{intro:kernel}-\eqref{intro:parameters boltzmann}.
Duhamel's formula implies the ternary Boltzmann equation can be written in mild form as follows
\begin{equation}\label{mild Boltzmann equation sec 4}
\begin{aligned}
f(t,x,v)&=S_1^tf_0(x,v)+\int_0^tS_1^{t-\tau}Q_3(f,f,f)(\tau,x,v)\,d\tau.
\end{aligned}
\end{equation}
See Remark \ref{ill posed equation} for the validity of \eqref{mild Boltzmann equation sec 4} in $L^\infty$.

\section{Local well-posedness}
\label{sec_local}
 
 In this section we address local well-posedness (LWP) for the BBGKY and Boltzmann hierarchies and the ternary Boltzmann equation. As expected, these well-posedness proofs are closely related, and they rely on defining appropriate functional spaces and establishing appropriate a-priori bounds. For this reason we provide the proofs only for the BBGKY case (for more details see \cite{thesis}). The functional spaces we introduce to address well-posedness are inspired by the spaces used in \cite{lanford,gallagher}.

 \subsection{LWP for the BBGKY hierarchy} Consider $(N,\epsilon)$ in the scaling \eqref{scaling}. For $1\leq s\leq N$ and $\beta> 0$ we define  the Banach spaces
 \begin{equation*}
 X_{N,\beta,s}
 :=\left\{g_{N,s}\in L^\infty(\mathcal{D}_{s,\epsilon}):  |g_{N,s}|_{N,\beta,s}:=\esssup_{Z_s\in\mathbb{R}^{2ds}}|g_{N,s}(Z_s)|e^{\beta E_s(Z_s)}<\infty\right\},
 \end{equation*}
where $E_s(Z_s)$ is the kinetic energy of $s$-particles  given by \eqref{kinetic energy}.
For $s>N$ we trivially define  $X_{N,\beta,s}:=\left\{0\right\}.$

Consider $\mu\in\mathbb{R}$. We define the Banach space 
\begin{equation*}X_{N,\beta,\mu}:=\left\{G_N=(g_{N,s})_{s\in\mathbb{N}}:g_{N,s}\in X_{N,\beta,s},\text{ }\forall s\in\mathbb{N}\text{ and }\|G_N\|_{N,\beta,\mu}:=\sup_{s\in\mathbb{N}}e^{\mu s}|g_{N,s}|_{N,\beta,s}<\infty\right\}.\end{equation*}
Finally, given $T>0$, $\beta_0> 0$, $\mu_0\in\mathbb{R}$ and $\bm{\beta},\bm{\mu}:[0,T]\to\mathbb{R}$ decreasing functions of time with $\bm{\beta}(0)=\beta_0$, $\bm{\beta}(T)> 0$, $\bm{\mu}(0)=\mu_0$, we define the Banach space 
\begin{equation*}
\bm{X}_{N,\bm{\beta},\bm{\mu}}:=C^0\left([0,T],X_{N,\bm{\beta}(t),\bm{\mu}(t)}\right),\text{ with norm }|||\bm{G_N}|||_{N,\bm{\beta},\bm{\mu}}=\sup_{t\in[0,T]}\|\bm{G_N}(t)\|_{N,\bm{\beta}(t),\bm{\mu}(t)}.
\end{equation*}
Now, given $m\in\mathbb{N}$, we prove an important continuity estimate for the operator $\mathcal{C}_{m,m+2}^N$.
\begin{lemma}\label{a priori lemma for C BBGKY}
Let $m\in\mathbb{N}$, $\beta>0$ and $g_{N,m+2}\in X_{N,m+2,\beta}$.  Then, the following continuity estimate holds for any 
\begin{equation*}
\left|\mathcal{C}_{m,m+2}^Ng_{N,m+2}(Z_m)\right|\lesssim  \beta^{-d}\left(m\beta^{-1/2}+\sum_{i=1}^m|v_i|\right)e^{-\beta E_m(Z_m)}|g_{N,m+2}|_{N,\beta,m+2},\quad\forall Z_{m}\in \mathcal{D}_{m,\epsilon}.
\end{equation*}
\end{lemma}
\begin{proof} Let $g_{N,m+2}\in X_{N,m+2,\beta}$ and $Z_m=(X_m,V_m)\in\mathbb{N}$. If $m\geq N-1$ both sides vanish, so  we may assume  that $m\leq N-2$. Notice that conservation of energy \eqref{conservation of energy tran}  implies 
\begin{equation}\label{cons ener LWP}
E_{m+2}(Z_{m+2,\epsilon}^{i,*})=E_{m+2}(Z_{m+2,\epsilon}^i),\quad\forall i=1,...,m.
\end{equation}
 Moreover,  \eqref{estimate on denominator}, Cauchy-Schwarz inequality and triangle inequality yield
 \begin{equation}\label{estimate on b/sqrt}
 \frac{b_+(\omega_1,\omega_2,v_2-v_1,v_3-v_1)}{\sqrt{1+\langle\omega_1,\omega_2\rangle}}\leq 2\sqrt{2}\left(|v_1|+|v_2|+|v_3|\right),\quad \forall (\omega_1,\omega_2,v_1,v_2,v_3)\in\mathbb{S}_1^{2d-1}\times\mathbb{R}^{3d}.
 \end{equation}
 Therefore, by \eqref{cons ener LWP}-\eqref{estimate on b/sqrt}, the definition of the norm and scaling  \eqref{scaling}
\begin{equation*}
\begin{aligned}
\bigg|&\mathcal{C}_{m,m+2}^Ng_{N,m+2}(Z_m)\bigg|\lesssim e^{-\beta E_m(Z_m)}|g_{N,m+2}|_{N,\beta,m+2}\\
&\times\sum_{i=1}^m\int_{\mathbb{R}^{2d}}\big(|v_i|+|v_{m+1}|+|v_{m+2}|\big)e^{-\frac{\beta}{2}(|v_{m+1}|^2+|v_{m+2}|^2)}\,dv_{m+1}\,dv_{m+2}.
\end{aligned}
\end{equation*}
Using Fubini's theorem and the elementary integrals
$$\int_0^\infty e^{-\frac{\beta}{2}x^2}dx\simeq\beta^{-1/2},\int_0^\infty xe^{-\frac{\beta}{2}x^2}dx\simeq\beta^{-1},$$
we obtain the required estimate.
\end{proof}

Now we define a mild solution of the BBGKY hierarchy in the scaling \eqref{scaling} as follows:
 \begin{definition} Consider $T>0$, $\beta_0> 0$, $\mu_0\in\mathbb{R}$ and the decreasing functions $\bm{\beta},\bm{\mu}:[0,T]\to\mathbb{R}$ with $\bm{\beta}(0)=\beta_0$, $\bm{\beta}(T)> 0$, $\bm{\mu}(0)=\mu_0$.
Consider also initial data $G_{N,0}=\left(g_{N,s,0}\right)\in X_{N,\beta_0,\mu_0}$. A map $\bm{G_N}=\left(g_{N,s}\right)_{s\in\mathbb{N}}\in\bm{X}_{N,\bm{\beta},\bm{\mu}}$ is a mild solution of the BBGKY hierarchy \eqref{BBGKY} in $[0,T]$  if it satisfies 
\begin{equation}\label{mild BBGKY}
\bm{G_N}(t)=\mathcal{T}^tG_{N,0}+\int_0^t \mathcal{T}^{t-\tau}\mathcal{C}_N\bm{G_N}(\tau)\,d\tau,
\end{equation}
where
$\mathcal{C}_N G_N=\left(\mathcal{C}_{s,s+2}^Ng_{N,s+2}\right)_{s\in\mathbb{N}}$
and $\mathcal{T}^t=(T_s^t)_{s\in\mathbb{N}}$, where $T_s^t$ is given by \eqref{liouville operator}.
\end{definition}

\begin{remark}\label{ill def C}
We note that the above collision operators $\mathcal{C}_{s,s+2}^N$ are ill-defined on $L^\infty$ since they involve integration over a set of measure zero (the sphere $\mathbb{S}_1^{d-1}$). However, by filtering our BBGKY hierarchy by the flow $T_s^{-t}$, we may obtain a well defined operator on $\bm{X}_{N,\bm{\beta},\bm{\mu}}$ . This is done in detail in the erratum of Chapter 5 of \cite{gallagher} and does not affect the energy estimates or local well-posedness of the hierarchy. This filtering process can be adapted to our context. Hence, we will abuse the notation and continue to work with \eqref{mild BBGKY}. See also \cite{simonella} for a different approach which avoids this issue by working with measures on the phase space. 
\end{remark}

We will address  well-posedness of the BBGKY hierarchy by a fixed point argument. For  this purpose, we state an important estimate.
 \begin{lemma}\label{a priori lemma for T BBGKY} Let $\beta_0> 0$, $\mu_0\in\mathbb{R}$, $T>0$ and $\lambda\in (0,\beta_0/T)$. Consider the functions $\bm{\beta}_\lambda,\bm{\mu}_\lambda:[0,T]\to\mathbb{R}$ given by
 \begin{equation}\label{beta_lambda-mu_lambda}
 \bm{\beta}_\lambda(t)=\beta_0-\lambda t,\quad\bm{\mu}_\lambda(t)=\mu_0-\lambda t.
\end{equation}  
 Then  for any $\mathcal{F}(t)\subseteq [0,t]$ measurable, $s\in\mathbb{N}$ and $\bm{G_N}=\left(g_{N,s}\right)_{s\in\mathbb{N}}\in\bm{X}_{N,\bm{\beta}_\lambda,\bm{\mu}_\lambda}$ the following bound holds:
\begin{equation*}\left|\left|\left|\int_{\mathcal{F}(t)}\mathcal{T}^{t-\tau}\bm{\mathcal{C}_N}\bm{G_N}(\tau)\,d\tau\right|\right|\right|_{N,\bm{\beta}_\lambda,\bm{\mu}_\lambda}\leq C(d,\beta_0,\mu_0,T,\lambda)|||\bm{G_N}|||_{N,\bm{\beta}_\lambda,\bm{\mu}_\lambda},\end{equation*}
\begin{equation}\label{constant of well posedness}C(d,\beta_0,\mu_0,T,\lambda)\simeq\lambda^{-1}e^{-2\bm{\mu}(T)}\bm{\beta}_\lambda(T)^{-d}\left(1+\bm{\beta}_\lambda\left(T\right)^{-1/2}\right).\end{equation} 
\end{lemma}
\begin{proof}
Since energy is conserved by the flow and we have  the continuity estimate of Lemma \ref{a priori lemma for C BBGKY} for the collisional operator, the proof follows similarly to the proof of Lemma 5.3.1. in \cite{gallagher}.
\end{proof}
 Choosing $\lambda=\beta_0/2T$, and $T=T(\beta_0,\mu_0)$ small enough, Lemma \ref{a priori lemma for T BBGKY}  implies local well-posedness of the BBGKY hierarchy via a fixed point argument.
 \begin{theorem}\label{well posedness BBGKY}
 Let $\beta_0> 0$ and $\mu_0\in\mathbb{R}$. Then there is $T=T(d,\beta_0,\mu_0)>0$ such that for any initial datum $F_{N,0}=(f_{N,0}^{(s)})_{s\in\mathbb{N}}\in X_{N,\beta_0,\mu_0}$ there is unique mild solution $\bm{F_N}\in\bm{X}_{N,\bm{\beta},\bm{\mu}}$ of the BBGKY hierarchy \eqref{BBGKY} in $[0,T]$ for the functions $\bm{\beta},\bm{\mu}:[0,T]\to\mathbb{R}$ given by
\begin{equation}\label{beta mu}
\bm{\beta}(t)=\beta_0-\frac{\beta_0}{2T}t,\quad
\bm{\mu}(t)=\mu_0-\frac{\beta_0}{2T}t.
\end{equation} 
 Moreover, for any $\mathcal{F}(t)\subseteq[0,t]$ measurable, the following bounds hold:
 \begin{align}
 \label{a priori bound F_N}\left|\left|\left|\int_{\mathcal{F}(t)}\mathcal{T}^{t-\tau}C_N\bm{G_N}(\tau)\,d\tau\right|\right|\right|_{N,\bm{\beta},\bm{\mu}}&\leq\frac{1}{8}|||\bm{G_N}|||_{{N,\bm{\beta},\bm{\mu}}},\quad\forall \bm{G_N}\in\bm{X}_{N,\bm{\beta},\bm{\mu}},\\
 \label{a priori bound F_N,0}|||\bm{F_N}|||_{N,\bm{\beta},\bm{\mu}}&\leq 2\|F_{N,0}\|_{N,\beta_0,\mu_0}.
 \end{align}
 \end{theorem}
 
  \subsection{LWP for the Boltzmann hierarchy} For the Boltzmann hierarchy analogous estimates follow in a similar manner as for the BBGKY hierarchy in the appropriate functional spaces. 
  
  Given $\beta> 0$ and $s\in\mathbb{N}$ we  define the Banach space
 \begin{equation*}X_{\infty,\beta,s}:=\left\{g_s\in L^\infty(\mathbb{R}^{2ds}):|g_s|_{\infty,\beta,s}:=\esssup_{Z_s\in\mathbb{R}^{2ds}}|g_s(Z_s)|e^{\beta E_s(Z_s)}<\infty\right\}.\end{equation*}
 Consider as well $\mu\in\mathbb{R}$. We define the Banach space 
\begin{equation*}X_{\infty,\beta,\mu}:=\left\{G=(g_{s})_{s\in\mathbb{N}}:\|G\|_{\infty,\beta,\mu}:=\sup_{s\in\mathbb{N}}e^{\mu s}|g_s|_{\infty,\beta,s}<\infty\right\}.\end{equation*}
 Finally, for $T>0$, $\beta_0>0$, $\mu_0\in\mathbb{R}$ and $\bm{\beta},\bm{\mu}:[0,T]\to\mathbb{R}$ decreasing functions of time with $\bm{\beta}(T)> 0$ we define the Banach space 
 \begin{equation*}
 \bm{X}_{\infty,\bm{\beta},\bm{\mu}}=C^0\left([0,T],X_{\infty,\bm{\beta}(t),\bm{\mu}(t)}\right),\text{ with norm } |||\bm{G}|||:=\sup_{t\in[0,T]}\|\bm{G}(t)\|_{\infty,\bm{\beta}(t),\bm{\mu}(t)}.
 \end{equation*}
We define a mild solution of the Boltzmann hierarchy  as follows.
 \begin{definition} Consider $T>0$, $\beta_0> 0$, $\mu_0\in\mathbb{R}$ and the decreasing functions $\bm{\beta},\bm{\mu}:[0,T]\to\mathbb{R}$ with $\bm{\beta}(0)=\beta_0$, $\bm{\beta}(T)> 0$, $\bm{\mu}(0)=\mu_0$.
Consider also  initial data $G_{0}=\left(g_{s,0}\right)\in X_{\infty,\beta_0,\mu_0}$. A map $\bm{G}=\left(g_{s}\right)_{s\in\mathbb{N}}\in\bm{X}_{\infty,\bm{\beta},\bm{\mu}}$ is a mild solution of the Boltzmann hierarchy \eqref{Boltzmann hierarchy} in $[0,T]$, with initial data $G_0$, if it satisfies:
\begin{equation}\label{mild form boltzmann}\bm{G}(t)=\mathcal{S}^tG_{0}+\int_0^t \mathcal{S}^{t-\tau}\mathcal{C}_\infty\bm{G}(\tau)\,d\tau,\end{equation}
where
$\mathcal{C}_\infty G=\left(\mathcal{C}_{s,s+2}^\infty g_{s+2}\right)_{s\in\mathbb{N}},$
and $\mathcal{S}^tG=(S_s^t g_s)_{s\in\mathbb{N}}$, where $S_s^t$ is given by \eqref{free flow operator}.
\end{definition}

\begin{remark}\label{ill def C boltz}
As noted in Remark \ref{ill def C}, the operators $\mathcal{C}_{s,s+2}^\infty$ are ill defined on $L^\infty$ due to the integration over the lower dimension manifold $\mathbb{S}_1^{d-1}$. As in the BBGKY case, one can filter the infinite hierarchy by $S_s^{-t}$ to obtain a well defined mild formulation of the hierarchy. However, for simplicity, we will abuse notation and continue to use \eqref{mild form boltzmann}
\end{remark}

Now we state the  well-posedness result for the Boltzmann hierarchy.
 \begin{theorem}\label{lwp Boltzmann hier}
 Let $\beta_0> 0$ and $\mu_0\in\mathbb{R}$. Then there is\footnote{The time of existence is the same as in Theorem \ref{well posedness BBGKY}} $T=T(d,\beta_0,\mu_0)>0$ such that for any initial datum $F_{0}=(f_0^{(s)})_{s\in\mathbb{N}}\in X_{\infty,\beta_0,\mu_0}$ there is unique mild solution $\bm{F}\in\bm{X}_{\infty,\bm{\beta},\bm{\mu}}$ of the Boltzmann hierarchy \eqref{Boltzmann hierarchy} in $[0,T]$ for the functions $\bm{\beta},\bm{\mu}:[0,T]\to\mathbb{R}$ given by \eqref{beta mu}.

 Moreover, for any $\mathcal{F}(t)\subseteq[0,t]$ measurable, the following estimates hold:
 \begin{align}
 \label{a priori bound F Boltzmann}\left|\left|\left|\int_{\mathcal{F}(t)}\mathcal{S}^{t-\tau}C_\infty\bm{G}(\tau)\,d\tau\right|\right|\right|_{\infty,\bm{\beta},\bm{\mu}}&\leq\frac{1}{8}|||\bm{G}|||_{{\infty,\bm{\beta},\bm{\mu}}},\quad\forall \bm{G}\in\bm{X}_{\infty,\bm{\beta},\bm{\mu}},\\
 \label{a priori bound F_0 Boltzmann}|||\bm{F}|||_{\infty,\bm{\beta},\bm{\mu}}&\leq 2\|F_{0}\|_{\infty,\beta_0,\mu_0}.
 \end{align}
 \end{theorem}
\subsection{LWP for the ternary Boltzmann equation and propagation of chaos}
Here, we first present local well-posedness for the ternary Boltzmann equation.  The proofs are nonlinear analogues of the arguments used in the BBGKY case (for details see \cite{thesis}). Furthermore, we show that for chaotic initial data their tensorized product produces the unique mild solution of the Boltzmann hierarchy, hence chaos is propagated.

For $\beta>0$ let us define the Banach space
\begin{equation*}
X_{\beta,\mu}:=\left\{g\in L^\infty(\mathbb{R}^{2d}):|g|_{\beta,\mu}:=\esssup_{(x,v)\in\mathbb{R}^{2d}} |g(x,v)|e^{\mu+\frac{\beta}{2} |v|^2}<\infty\right\}.
\end{equation*}
Consider $\beta_0>0$, $\mu_0\in\mathbb{R}$, $T>0$ and $\bm{\beta},\bm{\mu}:[0,T]\to\mathbb{R}$ decreasing functions of time with $\bm{\beta}(0)=\beta_0$, $\bm{\beta}(T)>0$ and $\bm{\mu}(0)=\mu_0$.
We define the Banach space
\begin{equation*}
\bm{X}_{\bm{\beta},\bm{\mu}}:=C^0\left([0,T],X_{\bm{\beta}(t),\bm{\mu}(t)}\right),\text{ with norm }\|\bm{g}\|_{\bm{\beta},\bm{\mu}}=\esssup_{t\in[0,T]}|\bm{g}(t)|_{\bm{\beta}(t),\bm{\mu}(t)}.
\end{equation*} 
We define mild solutions to the ternary Boltzmann equation as follows:
\begin{definition}
Consider $T>0$, $\beta_0> 0$, $\mu_0\in\mathbb{R}$ and  $\bm{\beta},\bm{\mu}:[0,T]\to\mathbb{R}$  decreasing functions of time, with $\bm{\beta}(0)=\beta_0$, $\bm{\beta}(T)> 0$, $\bm{\mu}(0)=\mu_0$.
Consider also initial data $g_{0}\in X_{\beta_0,\mu_0}$. A map $\bm{g}\in\bm{X}_{\bm{\beta},\bm{\mu}}$ is a mild solution of the ternary Boltzmann equation \eqref{Boltzmann equation} in $[0,T]$, with initial data $g_0\in X_{\beta_0,\mu_0}$, if it satisfies
\begin{equation}\label{mild boltzmann equation}
\bm{g}(t)=S_1^tg_0+\int_0^tS_1^{t-\tau}Q_3(\bm{g},\bm{g},\bm{g})(\tau)\,d\tau.
\end{equation} 
where $S_1^t$ denotes the free flow of $1$-particle given in \eqref{free flow operator}.
\end{definition}

\begin{remark}\label{ill posed equation}
As in Remarks \ref{ill def C}, \ref{ill def C boltz}, the operators $Q_{3}$ can be filtered by the free flow $S_1^{-t}$ in order to define the above equation on $L^\infty$. Hence, we will abuse notation and continue to work with \eqref{mild boltzmann equation}.
\end{remark}

Let us write $B_{\bm{X}_{\bm{\beta},\bm{\mu}}}$ for the unit ball of $\bm{X}_{\bm{\beta},\bm{\mu}}$. Then the following well-posedness result holds
\begin{theorem}\label{lwp boltz eq}
 Let $\beta_0> 0$ and $\mu_0\in\mathbb{R}$. Then there is\footnote{The time of existence is the same as in Theorem \ref{well posedness BBGKY}} $T=T(d,\beta_0,\mu_0)>0$ such that for any initial data $f_0\in X_{\beta_0,\mu_0}$, with $|f_0|_{\beta_0,\mu_0}\leq \frac{1}{2}$, there is a unique mild solution $\bm{f}\in B_{\bm{X}_{\bm{\beta},\bm{\mu}}}$ to the ternary Boltzmann equation in $[0,T]$ with initial data $f_0$, where $\bm{\beta},\bm{\mu}:[0,T]\to\mathbb{R}$ are the functions given by \eqref{beta mu}.
\end{theorem}

\begin{remark}
The smallness assumption on the initial data is needed in order to produce a solution up to the time of existence of solutions to the BBGKY and Boltzmann hierarchy obtained in Theorem \ref{well posedness BBGKY}, Theorem \ref{lwp Boltzmann hier} respectively. One can produce a solution for general initial data, as was done for the Boltzmann equation in \cite{lanford}, but the time of existence would be smaller due to the nonlinearity of \eqref{Boltzmann equation}. 
\end{remark}

We can now prove that chaos is propagated by the Boltzmann hierarchy.
\begin{theorem}[Propagation of chaos]\label{theorem propagation of chaos}
Let $\beta_0>0$, $\mu_0\in\mathbb{R}$, $T>0$ the time obtained by Theorem \ref{lwp boltz eq} and $\bm{\beta},\bm{\mu}:[0,T]\to\mathbb{R}$ the functions defined by \eqref{beta mu}. Consider $f_0\in X_{\beta_0,\mu_0}$ with
$|f_0|_{\beta_0,\mu_0}\leq\frac{1}{2}$.
Assume $\bm{f}\in B_{\bm{X}_{\bm{\beta},\bm{\mu}}}$ is the corresponding mild solution of the ternary Boltzmann equation in $[0,T]$, with initial data $f_0$ given by Theorem \ref{lwp boltz eq}. Then the following hold:\\
{\it(i)} $F_0=(f_0^{\otimes s})_{s\in\mathbb{N}}\in X_{\infty,\beta_0,\mu_0}$.\\
{\it(ii)} $\bm{F}=(\bm{f}^{\otimes s})_{s\in\mathbb{N}}\in\bm{X}_{\infty,\bm{\beta},\bm{\mu}}$.\\
{\it(iii)} $\bm{F}$ is the unique mild solution of the Boltzmann hierarchy in $[0,T]$, with initial data $F_0$.
\end{theorem}
\begin{proof}
\textit{(i)} is  verified by the bound on the initial data and the definition of the norms. By the the same bound again, we may apply Theorem \ref{lwp boltz eq} to obtain the unique mild solution $\bm{f}\in B_{\bm{X}_{\bm{\beta},\bm{\mu}}}$ of the corresponding ternary Boltzmann equation. Since $\|\bm{f}\|_{\bm{\beta},\bm{\mu}}\leq 1$, the definition of the norms directly imply \textit{(ii)}. It is also staightforward to verify that $\bm{F}$ is a mild solution of the Boltzmann hierarchy in $[0,T]$, with initial data $F_0$. Uniqueness of the mild solution to the Boltzmann hierarchy, obtained by Theorem \ref{lwp Boltzmann hier}, implies that $\bm{F}$ is  the unique mild solution.
\end{proof}

\section{Convergence Statement}
\label{sec_conv statement}

In this section, we define an appropriate notion of convergence, namely convergence in observables, and we state the main result of this paper. While our convergence result is valid for a general type of Boltzmann initial data and  approximation by BBGKY hierarchy initial data  (see Definition \ref{convergence of initial data}), we also provide a rate of convergence in the case of chaotic Boltzmann initial data and initial approximation by conditioned BBGKY hierarchy initial data (introduced in Definition \ref{conditioned BBGKY data de}).  

Throughout this section, we consider $(N,\epsilon)$ in the scaling \eqref{scaling}. We will also use the phase space $\mathcal{D}_{m,\epsilon}$ of $m$-particles of $\epsilon$-interaction zone  given by \eqref{phase space} and the functional spaces of Section \ref{sec_local}.

\subsection{Approximation of Boltzmann initial data} 
This Subsection  focuses on introducing relevant types of initial data. First, we  define the general notion of BBGKY hierarchy sequences approximating   Boltzmann hierarchy  initial data. Then we show that chaotic initial data produced by tensorized probability densities are approximated by conditioned BBGKY hierarchy sequences in the scaling \eqref{scaling}.
\begin{definition}\label{convergence of initial data}Let $\beta_0>0$, $\mu_0\in\mathbb{R}$ and $G_0=(g_{s,0})_{s\in\mathbb{N}}\in X_{\infty,\beta_0,\mu_0}$. A sequence 
$
G_{N,0}=(g_{N,s,0})_{s\in\mathbb{N}}\in X_{N,\beta_0,\mu_0}$ is called a BBGKY hierarchy sequence approximating $G_0$ if the following conditions hold:
\begin{enumerate}[(i)]
\item $\displaystyle\sup_{N\in\mathbb{N}}\|G_{N,0}\|_{N,\beta_0,\mu_0}<\infty.$
\item  For any $s\in\mathbb{N}$  there holds
$\displaystyle\lim_{N\to\infty}\|g_{N,s,0}-g_{s,0}\|_{L^\infty(\mathcal{D}_{s,\epsilon})}= 0.
$
\end{enumerate}
\end{definition}
\begin{remark}
Every $G_0=(g_{s,0})_{s\in\mathbb{N}}\in X_{\infty,\beta_0,\mu_0}$ has a  BBGKY hierarchy approximating sequence. Indeed, it is straightforward to verify that the sequence $G_{N,0}=(g_{N,s,0})_{s\in\mathbb{N}}$ given by $g_{N,s,0}=\mathds{1}_{\mathcal{D}_{s,\epsilon}}g_{s,0}$ satisfies the properties stated above in the scaling \eqref{scaling}.
\end{remark}
Especially meaningful initial data, corresponding to initial independence between particles, are given below:
\begin{remark}\label{remark on conditioned}
Let $g_0\in X_{\beta_0,\mu_0+1}$ be  a positive probability density i.e. $g_0> 0$ a.e. and $\int_{\mathbb{R}^{2d}}g_0(x,v)\,dx\,dv=1$ and assume that $\|g_0\|_{{\beta_0,\mu_0+1}}\leq 1$.
Then one can easily see  that the chaotic configuration $G_0=(g_0^{\otimes s})_{s\in\mathbb{N}}\in X_{\infty,\beta_0,\mu_0+1}\subseteq X_{\infty,\beta_0,\mu_0}$. This type of initial data, corresponding to tensorized initial measures, will lead to the ternary Boltzmann equation \eqref{Boltzmann equation}. In fact, we will see that one can approximate tensorized initial data in the scaling \eqref{scaling} by conditioned BBGKY hierarchy initial data which are defined below.
\end{remark}
\begin{definition}\label{conditioned BBGKY data de}
Let $g_0\in X_{\beta_0,\mu_0+1}$ be  a positive probability density and denote $G_0=(g_0^{\otimes s})_{s\in\mathbb{N}}\in X_{\infty,\beta_0,\mu_0+1}$.
We define the conditioned BBGKY hierarchy sequence $G_{N,0}=(g_{N,0}^{(s)})_{s\in\mathbb{N}}$ of $G_0$ as:
\begin{equation}\label{conditioned BBGKY data}
g_{N,0}^{(s)}(X_s,V_s)=
\begin{cases}
\mathcal{Z}_N^{-1}\displaystyle\int_{\mathbb{R}^{2d(N-s)}}\mathds{1}_{\mathcal{D}_{N,\epsilon}}g_0^{\otimes N}(X_s,x_{s+1},...,x_N,V_{s},v_{s+1},...,v_N)\,dx_{s+1}\,dv_{s+1}...\,dx_N\,dv_{N},\quad 1\leq s<N\\
\mathcal{Z}_N^{-1}\mathds{1}_{\mathcal{D}_{N,\epsilon}}g_0^{\otimes N}(Z_N),\quad s=N,\\
0,\quad s>N.
\end{cases}
\end{equation}
where the normalization is preserved by the introduction of the partition function: $$\mathcal{Z}_m=\int_{\mathbb{R}^{2dm}}\mathds{1}_{\mathcal{D}_{m,\epsilon}}g_0^{\otimes m}(X_m,V_m)\,dX_m\,dV_m,\quad m\in\mathbb{N}.$$
Notice that since $g_0$ is a.e. positive and integrates to $1$, we have $0<Z_m<1$ for all $m\in\mathbb{N}$. 
\end{definition}

Let us now prove that the conditioned BBGKY hierarchy sequence of tensorized initial data is an approximating sequence (according to Definition \ref{convergence of initial data}). This will be a crucial tool to obtain rate of convergence to the solution of the ternary Boltzmann equation \eqref{Boltzmann equation} (see Corollary \ref{derivation corollary} for more details). We will need the following auxiliary estimate on the partition functions.
\begin{lemma}\label{estimate on conditioning} Let $\beta_0>0$, $\mu_0\in\mathbb{R}$ and $g_0\in L^\infty_xL^1_v(\mathbb{R}^{2d})$ be  a positive probability density. Then for all $(N,\epsilon)$ in the scaling \eqref{scaling} with $2C_d\epsilon^{1/2}\|g_0\|_{L^\infty_xL^1_v}<1$, where $C_d$ is a positive constant,  and all $m\in\mathbb{N}$ with $m< N$,  there holds
$$1\leq\mathcal{Z}_N^{-1}\mathcal{Z}_{N-m}\leq  (1-C_d\|g_0\|_{L^\infty_xL^1_v}\epsilon^{1/2})^{-m},$$
for some constant $C_{d}>0$.
\end{lemma}
\begin{proof}
The left hand side inequality is immediate from the definition of the phase space \eqref{phase space}. To prove the right hand side consider $k\in\mathbb{N}$ with $k\leq N$. Notice that for any $Z_{k+1}=(X_{k+1},V_{k+1})\in\mathbb{R}^{2d(k+1)}$, we have
$$\mathds{1}_{\mathcal{D}_{k+1,\epsilon}}(X_{k+1},V_{k+1})\geq\mathds{1}_{\mathcal{D}_{k,\epsilon}}(X_k,V_k)\prod_{i=1}^{k}\mathds{1}_{|x_i-x_{k+1}|>\sqrt{2}\epsilon}(x_i),$$
by the definition of the phase space \eqref{phase space}. Let us note that the above inequality applies specifically to the ternary interactions we consider. Then we can proceed in a similar manner as in the proof of Lemma 6.1.2 in \cite{gallagher}, using the ternary scaling \eqref{scaling} instead.
More specifically, the previous inequality and Fubini's Theorem imply
\begin{align*}
\mathcal{Z}_{k+1}&=\int_{\mathbb{R}^{2d(k+1)}}\mathds{1}_{\mathcal{D}_{k+1,\epsilon}}g_0^{\otimes(k+1)}(X_{k+1},V_{k+1})\,dX_{k+1}\,dV_{k+1}\\
&\geq \int_{\mathbb{R}^{2dk}}\left(\int_{\mathbb{R}^{2d}}\prod_{i=1}^k\mathds{1}_{|x_i-x|>\sqrt{2}\epsilon}(x_i)g_0(x,v)\,dx\,dv\right)\mathds{1}_{\mathcal{D}_{k,\epsilon}}(X_k,V_k)g_0^{\otimes m}(X_k,V_k)\,dX_k\,dV_k.
\end{align*}
But since $g_0$ integrates to $1$,
we have
\begin{align*}
\int_{\mathbb{R}^{2d}}\prod_{i=1}^k\mathds{1}_{|x_i-x|>\sqrt{2}\epsilon}(x_i)g_0(x,v)\,dx\,dv&\geq 1-\sum_{i=1}^k\int_{\mathbb{R}^{2d}}\mathds{1}_{|x_i-x|\leq {\sqrt{2}\epsilon}}(x_i)g_0(x,v)\,dx\,dv\geq 1-kC_d\|g_0\|_{L^\infty_xL^1_v}\epsilon^d,
\end{align*}
upon integrating on a $d$-ball of radius $\sqrt{2}\epsilon$.
 Hence 
 \begin{equation}\label{inductive step tensor}
 \mathcal{Z}_{k+1} \geq (1-kC_d\|g_0\|_{L^\infty_xL^1_v}\epsilon^d)\mathcal{Z}_k\geq (1-NC_d\|g_0\|_{L^\infty_xL^1_v}\epsilon^d)\mathcal{Z}_k\simeq (1-C_d\|g_0\|_{L^\infty_xL^1_v}\epsilon^{1/2})\mathcal{Z}_k ,
 \end{equation}
 due to scaling \eqref{scaling}. For $2C_d\|g_0\|_{L^\infty_xL^1_v}\epsilon^{1/2}<1$, we may
apply inductively \eqref{inductive step tensor}  for $k=m,...,N-1$,  and the claim follows.
\end{proof}
\begin{proposition}\label{conditioned proposition} Let $g_0\in X_{\beta_0,\mu_0+1}$ be  a positive probability density with $|g_0|_{\beta_0,\mu_0+1}\leq 1$ and $G_0=(g_0^{\otimes s})_{s\in\mathbb{N}}\in X_{\infty,\beta_0,\mu_0+1}\subseteq X_{\infty,\beta_0,\mu_0}$. Let $G_{N,0}=(g_{N,0}^{(s)})_{s\in\mathbb{N}}$ be  the conditioned BBGKY hierarchy sequence of the tensorized initial data $G_0$ given in Definition \ref{conditioned BBGKY data de}. Then $G_{N,0}$ is a BBGKY hierarchy sequence   approximating $G_0$ (in the sense of Definition \ref{convergence of initial data}) in the scaling \eqref{scaling}. In particular for all $(N,\epsilon)$ in the scaling \eqref{scaling} with $N$ large enough (or equivalently $\epsilon$ small enough), there holds the estimate
\begin{align}
\|g_{N,0}^{(s)}-g_0^{\otimes s}\|_{L^\infty(\mathcal{D}_{s,\epsilon})}&\leq C_{d,s,\beta_0,\mu_0}\epsilon^{1/2}\|G_0\|_{\infty,\beta_0,\mu_0}.\label{tensorized data approximating estimate}
\end{align}
\end{proposition}
\begin{proof}
By definition of the phase space \eqref{phase space}, for any $s\in\mathbb{N}$, with $s< N$ and $Z_N\in\mathcal{D}_{N,\epsilon}$ we can write
\begin{align*}\mathds{1}_{\mathcal{D}_{N,\epsilon}}(Z_N)=&\mathds{1}_{\mathcal{D}_{s,\epsilon}}(Z_s)\prod_{1\leq i<j\leq s<k\leq N}\mathds{1}_{|x_i-x_j|^2+|x_i-x_k|^2>2\epsilon^2}(x_i,x_j,x_k)\nonumber&\\
&\prod_{1\leq i\leq s<j<k\leq N}\mathds{1}_{|x_i-x_j|^2+|x_i-x_k|^2>2\epsilon^2}(x_i,x_j,x_k)\prod_{s+1\leq i<j<k\leq N}\mathds{1}_{|x_i-x_j|^2+|x_i-x_k|^2>2\epsilon^2}(x_i,x_j,x_k).
\end{align*}
Again this decomposition of the phase space is due to the ternary interactions we consider and is necessary to track all the cases arising from ternary interactions.
Moreover, by symmetry, for $s< N$ we can also write
$$\mathcal{Z}_{N-s}=\int_{\mathbb{R}^{2d(N-s)}}\prod_{s+1\leq \ell_1<\ell_2<\ell_3\leq N}\mathds{1}_{|x_{\ell_1}-x_{\ell_2}|^2+|x_{\ell_1}-x_{\ell_3}|^2>2\epsilon^2}(x_{\ell_1},x_{\ell_2},x_{\ell_3})\prod_{\ell=s+1}^Ng_0(x_\ell,v_\ell)\,dZ_{(s+1,N)},$$
where $\,dZ_{(s+1,N)}:=\,dx_{s+1}...\,dx_N\,dv_{s+1}...\,dv_N$.
Therefore, given $Z_s\in\mathbb{R}^{2ds}$,   an elementary calculation gives
\begin{align}\label{decomposition of the marginal}
g_{N,0}^{(s)}(Z_s)&=\mathcal{Z}^{-1}_N\mathds{1}_{\mathcal{D}_{s,\epsilon}}(Z_s)g_0^{\otimes s}(Z_s)\left(\mathcal{Z}_{N-s}-\mathcal{R}_{s+1,N}\left(Z_s\right)\right),
\end{align}
where the error term $\mathcal{R}_{s+1,N}(Z_s)>0$ is given by
\begin{align}
&\mathcal{R}_{s+1,N}(Z_s)\nonumber\\
&=\int_{\mathbb{R}^{2d(N-s)}}\left(1-\prod_{1\leq i<j\leq s<k\leq N}\mathds{1}_{|x_i-x_j|^2+|x_i-x_k|^2>2\epsilon^2}(x_k)\prod_{1\leq i\leq s<j<k\leq N}\mathds{1}_{|x_i-x_j|^2+|x_i-x_k|^2>2\epsilon^2}(x_j,x_k)\right)\nonumber\\
&\hspace{2cm}\prod_{s+1\leq \ell_1<\ell_2<\ell_3\leq N}\mathds{1}_{|x_{\ell_1}-x_{\ell_2}|^2+|x_{\ell_1}-x_{\ell_3}|^2>2\epsilon^2}(x_{\ell_1},x_{\ell_2},x_{\ell_3})\prod_{\ell=s+1}^Ng_0(x_\ell,v_\ell)\,dZ_{(s+1,N)}\nonumber\\
&\leq\int_{\mathbb{R}^{2d(N-s)}}\left(\sum_{1\leq i<j\leq s<k\leq N}\mathds{1}_{|x_i-x_j|^2+|x_i-x_k|^2\leq 2\epsilon^2}(x_k)+\sum_{1\leq i\leq s<j<k\leq N}\int_{\mathbb{R}^{2d}}\mathds{1}_{|x_i-x_j|^2+|x_i-x_k|^2\leq 2\epsilon^2}(x_j,x_k)\right)\nonumber\\
&\hspace{2cm}\prod_{s+1\leq \ell_1<\ell_2<\ell_3\leq N}\mathds{1}_{|x_{\ell_1}-x_{\ell_2}|^2+|x_{\ell_1}-x_{\ell_3}|^2>2\epsilon^2}(x_{\ell_1},x_{\ell_2},x_{\ell_3})\prod_{\ell=s+1}^Ng_0(x_\ell,v_\ell)\,dZ_{(s+1,N)}\nonumber\\
&:= I_1+I_2\label{bound on error}.
\end{align}
By \eqref{decomposition of the marginal}, and the fact that $\mathcal{Z}_{N-s}\leq 1$ since $g_0$ integrates to $1$, by definition of the norms, we have
$$\|G_{N,0}\|_{{N,\beta_0,\mu_0}}\leq\mathcal{Z}_N^{-1}\|G_0\|_{\infty,\beta_0,\mu_0}<\infty,$$
so $G_{N,0}\in X_{N,\beta_0,\mu_0}$ for all $N\in\mathbb{N}$.
Moreover, since 
\begin{equation}\label{LinfL1}
\|g_0\|_{L^\infty_xL^1_v}\leq C_d\beta^{-1/2}e^{-\mu_0}|g_0|_{\beta_0,\mu_0}<\infty,
\end{equation}
 for $2C_d\epsilon^{1/2}\|g_0\|_{L^\infty_xL^1v}<1$ (or equivalently for $N$ large enough), Lemma \ref{estimate on conditioning} gives

$$g_{N,0}^{(s)}(Z_s)\leq (1-C_{d}\|g_0\|_{L^\infty_xL^1_v}\epsilon^{1/2})^{-s}g_0^{\otimes s}(Z_s)\leq e^{2sC_{d}\|g_0\|_{L^\infty_xL^1_v}\epsilon^{1/2}}g_0^{\otimes s}(Z_s)\leq e^{s}g_0^{\otimes s}(Z_s),$$
where we used the inequality $2x-\ln(1-x)\geq 0$, $x\in[0,1/2]$. This clearly implies
$$\|G_{N,0}\|_{N,\beta_0,\mu_0}\leq \|G_0\|_{\infty,\beta_0,\mu_0+1}<\infty,$$
for $N$ large enough, thus
$\sup_{N\in\mathbb{N}}\|G_{N,0}\|_{N,\beta_0,\mu_0}<\infty.$

To prove convergence, by \eqref{decomposition of the marginal} and the definition of the norms we take
\begin{equation}\label{term to estimate 6}
\left|\mathds{1}_{\mathcal{D}_{s,\epsilon}}(g_0^{\otimes s}-g_{N,0}^{(s)})(Z_s)\right|\leq \bigg( \left|1-\mathcal{Z}_N^{-1}\mathcal{Z}_{N-s}\right|+\mathcal{Z}_{N}^{-1}\mathcal{R}_{s+1,N}(Z_s)\bigg)e^{-s\mu_0}\|G_0\|_{\infty,\beta_0,\mu_0}.
\end{equation}
Let us estimate each term on \eqref{term to estimate 6} separately. By Lemma \ref{estimate on conditioning} and the inequality $2x-\ln(1-x)\geq 0$, $x\in[0,1/2]$, for $2\epsilon^{1/2}C_d\|g_0\|_{L^\infty_xL^1_v}<1$, we have 
\begin{equation}\label{1-cond}
|1-\mathcal{Z}_{N}^{-1}\mathcal{Z}_{N-s}|\leq e^{2s\epsilon^{1/2}C_d\|g_0\|_{L^\infty_xL^1_v}}-1\leq  2es\epsilon^{1/2}C_d\|g_0\|_{L^\infty_xL^1_v},
\end{equation}
by the Mean Value Theorem.

For the term $\mathcal{Z}_N^{-1}\mathcal{R}_{s+1,N}$, we estimate each of the terms $I_1,I_2$ in \eqref{bound on error}. For the term $I_1$, fix $1\leq i<j\leq s<k\leq N$. Notice the inequality
$$\mathds{1}_{|x_i-x_j|^2+|x_i-x_k|^2\leq 2\epsilon^2}(x_k)\leq\mathds{1}_{|x_i-x_k|\leq\sqrt{2}\epsilon}(x_k).$$
Then, by symmetry, the term corresponding to $i,j,k$ is estimated by
\begin{align*}
&\int_{\mathbb{R}^{2d(N-s-1)}}\left(\int_{\mathbb{R}^{2d}}\mathds{1}_{|x_i-x_{s+1}|<\sqrt{2}\epsilon}(x_{s+1})g_0(x_{s+1},v_{s+1})\,dx_{s+1}\,dv_{s+1}\right)\\
&\hspace{2cm}\prod_{s+2\leq \ell_1<\ell_2<\ell_3\leq N}\mathds{1}_{|x_{\ell_1}-x_{\ell_2}|^2+|x_{\ell_1}-x_{\ell_3}|^2>2\epsilon^2}(x_{\ell_1},x_{\ell_2},x_{\ell_3})\prod_{\ell=s+2}^Ng_0(x_\ell,v_\ell)\,dZ_{(s+2,N)}\\
&\leq C_d\|g_0\|_{L^\infty_xL^1_v}\epsilon^d\mathcal{Z}_{N-s-1},
\end{align*}
after integrating in a $d$-ball of radius $\sqrt{2}\epsilon$ centered at $x_i$.
Adding for $1\leq i<j\leq s<k\leq N$ we obtain
\begin{equation}\label{bound on I_1 error}
I_1\le s^2 NC_d\|g_0\|_{L^\infty_xL^1_v}\epsilon^d\mathcal{Z}_{N-s-1}\simeq C_{d}s^2\epsilon^{1/2}\|g_0\|_{L^\infty_xL^1_v}\mathcal{Z}_{N-s-1},
\end{equation}
due to \eqref{scaling}. For the term $I_2$, fix $1\leq 1\leq s<j<k\leq N$. By symmetry  again the corresponding term is estimated by
\begin{align*}
&\int_{\mathbb{R}^{2d(N-s-2)}}\left(\int_{\mathbb{R}^{2d}}\mathds{1}_{|x_i-x_{s+1}|^2+|x_i-x_{s+2}|^2\leq 2\epsilon^2}(x_{s+1},x_{s+2})g_0(x_{s+1},v_{s+1})g_0(x_{s+2},v_{s+2})\,dx_{s+1}\,dx_{s+2}\,dv_{s+1}\,dv_{s+2}\right)\\
&\hspace{2cm}\prod_{s+3\leq \ell_1<\ell_2<\ell_3\leq N}\mathds{1}_{|x_{\ell_1}-x_{\ell_2}|^2+|x_{\ell_1}-x_{\ell_3}|^2>2\epsilon^2}(x_{\ell_1},x_{\ell_2},x_{\ell_3})\prod_{\ell=s+3}^Ng_0(x_\ell,v_\ell)\,dZ_{(s+3,N)}\\
&\leq C_d\|g_0\|_{L^\infty_xL^1_v}^2\epsilon^{2d}\mathcal{Z}_{N-s-2}.
\end{align*}
after integrating in a $2d$-ball of radius $\epsilon$ centered at $\binom{x_i}{x_i}$. Adding for $1\leq i\leq s<j<k\leq N$ we obtain
\begin{equation}\label{bound on I_2 error}
I_2\leq sN^2C_d\|g_0\|_{L^\infty_xL^1_v}^2\epsilon^{2d}\mathcal{Z}_{N-s-2}\simeq s C_d\|g_0\|_{L^\infty_xL^1_v}^2\epsilon\mathcal{Z}_{N-s-2},
\end{equation}
Using \eqref{bound on error}-\eqref{bound on I_2 error} and Lemma \ref{estimate on conditioning} (applied for $m=s+1$ and $m=s+2$, we obtain
\begin{align}
\mathcal{Z}_N^{-1}\mathcal{R}_{s+1,N}(Z_s)&\lesssim s^2C_d\|g_0\|_{L^\infty_xL^1_v}\epsilon^{1/2}(1-C_d\|g_0\|_{L^\infty_xL^1_v}	\epsilon^{1/2})^{-(s+1)}\nonumber\\
&\hspace{2cm}+sC_d\|g_0\|_{L^\infty_xL^1_v}^2\epsilon(1-C_d\|g_0\|_{L^\infty_xL^1_v}\epsilon^{1/2})^{-(s+2)}\nonumber\\
&\lesssim C_{d,s}\|g_0\|_{L^\infty_xL^1_v}\epsilon^{1/2},\label{bound on total error cond}
\end{align}
since $2C_d\epsilon^{1/2}\|g_0\|_{L^\infty_xL^1_v}<1$.
Combining \eqref{term to estimate 6}-\eqref{1-cond}, \eqref{bound on total error cond}, and \eqref{LinfL1}, we obtain estimate \eqref{tensorized data approximating estimate} and the required convergence follows.
\end{proof}

\subsection{Convergence in observables}
Now, we define the convergence in observables.  Given $s\in\mathbb{N}$, we use the space of test continuous and compactly supported functions in velocities $C_c(\mathbb{R}^{ds}).$
\begin{definition} Consider $T>0$, $s\in\mathbb{N}$ and $g_s\in C^0\left([0,T],L^\infty\left(\mathbb{R}^{2ds}\right)\right)$. Given a test function $\phi_s\in C_c(\mathbb{R}^{ds})$, we define the $s$-observable functional as:
$
I_{\phi_s}g_s(t)(X_s)=\displaystyle\int_{\mathbb{R}^{ds}}\phi_s(V_s)g_s(t,X_s,V_s)\,dV_s.
$
\end{definition}
 Before giving the definition of convergence in observables,
 we start with some definitions on the configurations we are using. Given $m\in\mathbb{N}$ and $\sigma>0$,  we define the set of well-separated spatial configurations 
\begin{equation}\label{separated conf space}
\Delta_m^X(\sigma)=\{\widetilde{X}_m\in\mathbb{R}^{dm}: |\widetilde{x}_i-\widetilde{x}_j|>\sigma,\quad\forall 1\leq i<j\leq m\},\quad m\geq 2,\quad \Delta_1^X(\sigma)=\mathbb{R}^{2d},
\end{equation}
and the set of well separated configurations
\begin{equation}\label{separated conf}
\Delta_m(\sigma)=\Delta_m^X(\sigma)\times\mathbb{R}^{dm}.
\end{equation}
\begin{definition}
Let $T>0$. For each $N\in\mathbb{N}$, consider $\bm{G_N}=(g_{N,s})_{s\in\mathbb{N}}\in \prod_{s=1}^\infty C^0\left([0,T],L^\infty\left(\mathbb{R}^{2ds}\right)\right)$  and $\bm{G}=(g_s)_{s\in\mathbb{N}}\in \prod_{s=1}^\infty C^0\left([0,T],L^\infty\left(\mathbb{R}^{2ds}\right)\right)$. We say that the sequence $(\bm{G_N})_{N\in\mathbb{N}}$ converges in observables to $\bm{G}$, and write $$\bm{G_N}\overset{\sim}\longrightarrow \bm{G},$$ if for any $\sigma>0$, $s\in\mathbb{N}$,  and $\phi_s\in C_c(\mathbb{R}^{ds})$, we have
\begin{equation*}
\lim_{N\to\infty}\|I_{\phi_s}g_{N,s}(t)-I_{\phi_s}g_s(t)\|_{L^\infty(\Delta_s^X(\sigma))}=0,\quad\text{uniformly in }[0,T].
\end{equation*}

\end{definition}
\subsection{Statement of the main result}
We are now in the position to state our main result. 
\begin{theorem}[Convergence]\label{convergence theorem}
Let $\beta_0> 0$, $\mu_0\in\mathbb{R}$ and consider  Boltzmann hierarchy initial data $F_0=(f_0^{(s)})_{s\in\mathbb{N}}\in X_{\infty,\beta_0,\mu_0}$. Let  $\left(F_{N,0}\right)_{N\in\mathbb{N}}$ be a  BBGKY hierarchy sequence approximating $F_0$ . Assume that:\\
\\
$\bullet$ For each $N$, $\bm{F_N}\in \bm{X}_{N,\bm{\beta},\bm{\mu}}$ is the mild solution of the BBGKY hierarchy \eqref{BBGKY}  with initial data $F_{N,0}$ in $[0,T]$.\\
$\bullet$ $\bm{F}\in\bm{X}_{\infty,\bm{\beta},\bm{\mu}}$ is the mild solution of the Boltzmann hierarchy \eqref{Boltzmann hierarchy}  with initial data $F_0$ in $[0,T]$.\\
$\bullet$ $F_0$ satisfies the following uniform continuity  condition: There exists $C>0$ such that, for any $\zeta>0$, there is $q=q(\zeta)>0$ such that for all $s\in\mathbb{N}$, and for all $Z_s,Z_s'\in\mathbb{R}^{2ds}$ with $|Z_s-Z_s'|<q$, we have
\begin{equation}\label{continuity assumption}
|f_0^{(s)}(Z_s)-f_0^{(s)}(Z_s')|<C^{s-1}\zeta.
\end{equation}
Then
$\bm{F_N}\overset{\sim}\longrightarrow\bm{F}.$
\end{theorem}
\begin{remark}To prove Theorem \ref{convergence theorem} it suffices to prove
$$\|I_s^N(t)-I_s^\infty(t)\|_{L^\infty(\Delta^X_s(\sigma))}\overset{N\to\infty}\longrightarrow 0,\text{ uniformly in $[0,T]$},
$$
 for any $s\in\mathbb{N}$, $\phi_s\in C_c(\mathbb{R}^{ds})$ and $\sigma>0$, where 
\begin{align}
I_s^N(t)(X_s)&:=I_{\phi_s}f_N^{(s)}(t)(X_s)=\int_{\mathbb{R}^{ds}}\phi_s(V_s)f_N^{(s)}(t,X_s,V_s)\,dV_s, \label{def-I-Ns} \\
I_s^\infty(t)(X_s)&:=I_{\phi_s}f^{(s)}(t)(X_s)=\int_{\mathbb{R}^{ds}}\phi_s(V_s)f^{(s)}(t,X_s,V_s)\,dV_s.  \label{def-I-s}
\end{align}
\end{remark}

The following Corollary of Theorem \ref{convergence theorem} justifies the derivation of our ternary Boltzmann equation from finitely many particle systems.

\begin{corollary}\label{derivation corollary} Let $\beta_0>0$, $\mu_0\in\mathbb{R}$ and $f_0\in X_{\beta_0,\mu_0+1}$ be a H\"older continuous $C^{0,\gamma}$, $\gamma\in(0,1]$ probability density with $|f_0|_{\beta_0,\mu_0+1}\leq 1/2$. Let us write $F_0=(f_0^{\otimes s})_{s\in\mathbb{N}}\in X_{\infty,\beta_0,\mu_0+1}$ and let $F_{N,0}=(f_{N,0}^{(s)})_{s\in\mathbb{N}}$ be the conditioned BBGKY hierarchy  sequence given in Definition \ref{conditioned BBGKY data de} approximating the tensorized data $F_0$. Then for any $\sigma>0$, $s\in\mathbb{N}$ and $\phi_s\in C_c(\mathbb{R}^{ds})$, we have the rate of convergence
\begin{equation}\label{derivation}
\|I_{\phi_s}f_N^{(s)}(t)-I_{\phi_s}f^{\otimes s}(t)\|_{L^\infty(\Delta^X_s(\sigma))}=O(\epsilon^r),\quad\text{uniformly in }[0,T],
\end{equation}
for any $0<r<\min\{1/2,\gamma\}$, where
 $\bm{F_N}=(f_N^{(s)})_{s\in\mathbb{N}}\in \bm{X_{N,\beta,\mu}}$ is the mild solution of the BBGKY hierarchy \eqref{BBGKY} in $[0,T]$ with initial data $F_{N,0}$  and $f$ is the mild solution to the ternary Boltzmann equation \eqref{Boltzmann equation} in $[0,T]$, with initial data $f_0$. 
\end{corollary}

\section{Reduction to term by term convergence}
\label{sec_term by term}

Now,  we reduce the proof of Theorem \ref{convergence theorem} to term by term convergence by truncating the observables. 
Throughout this section,  we consider $\beta_0>0$, $\mu_0\in\mathbb{R}$,  $T=T(d,\beta_0,\mu_0)>0$ be the time given by Theorems \ref{well posedness BBGKY}, \ref{lwp Boltzmann hier}, the functions $\bm{\beta},\bm{\mu}:[0,T]\to\mathbb{R}$ defined by \eqref{beta mu}, $(N,\epsilon)$ in the scaling \eqref{scaling} and initial data $F_{N,0}\in X_{N,\beta_0,\mu_0}$, $F_0\in X_{\infty,\beta_0,\mu_0}$. Let $\bm{F_N}\in\bm{X}_{N,\bm{\beta},\bm{\mu}}$, $\bm{F}\in\bm{X}_{\infty,\bm{\beta},\bm{\mu}}$ be the mild solutions of the corresponding BBGKY hierarchy and Boltzmann hierarchy  in $[0,T]$, given by Theorem \ref{well posedness BBGKY} and Theorem \ref{lwp Boltzmann hier}.
\subsection{Series expansion}
Let us fix $s\in\mathbb{N}$. Using iteratively the Duhamel's formula for the mild solution of the BBGKY hierarchy, given by \eqref{mild BBGKY}, we get the following expansion:
\begin{equation}
f_N^{(s)}(t,Z_s)=\sum_{k=0}^nf_{N}^{(s,k)}(Z_s)+R_{N}^{(n+1)}(t,Z_s)\label{functions plus remainder},
\end{equation}
where for $k\in\mathbb{N}$, we define
\begin{equation}\label{f_N k}
f_{N}^{(s,k)}(t,Z_s):=\int_0^t\int_0^{t_1}...\int_0^{t_{k-1}}T_s^{t-t_1}\mathcal{C}_{s,s+2}^NT_{s+2}^{t_1-t_2}...T_{s+2k-2}^{t_{k-1}-t_k}\mathcal{C}_{s+2k-2,s+2k}^NT_{s+2k}^{t_k}f_{N,0}^{(s+2k)}(Z_s)\,dt_k...\,dt_{1},
\end{equation}
 for $k=0$, we define
$f_{N}^{(s,0)}(t,Z_s):=T_s^tf^{(s)}_{N,0}(Z_s)$,
and  for the remainder we write
\begin{equation}\label{remainder_N}
\begin{aligned}
R_N^{(s,n+1)}(t,Z_s):=\int_0^t\int_0^{t_1}...\int_0^{t_{n}}T_s^{t-t_1}\mathcal{C}_{s,s+2}^NT_{s+2}^{t_1-t_2}...T_{s+2n-2}^{t_{n-1}-t_n}\mathcal{C}_{s+2n-2,s+2n}^NT_{s+2n}^{t_n-t_{n+1}}f_N^{\left(s+2n+2\right)}(t_{n+1},Z_s)&\\
\,dt_{n+1}...\,dt_1.&
\end{aligned}
\end{equation}

Similarly, using iteratively  Duhamel's formula for the solution of the Boltzmann hierarchy, one gets
\begin{equation}\label{function plus remainder boltzmann}
f^{(s)}(t,Z_s)=\sum_{k=0}^nf^{(s,k)}(Z_s)+R^{(n+1)}(t,Z_s)
\end{equation}
where for $k\in\mathbb{N}$, we define
\begin{equation}\label{f k}
f^{(s,k)}(t,Z_s):=\int_0^t\int_0^{t_1}...\int_0^{t_{k-1}}S_s^{t-t_1}\mathcal{C}_{s,s+2}^\infty S_{s+2}^{t_1-t_2}...
S_{s+2k-2}^{t_{k-1}-t_k}\mathcal{C}_{s+2k-2,s+2k}^\infty S_{s+2k}^{t_k}f_{0}^{(s+2k)}(Z_s)\,dt_k...\,dt_{1},
\end{equation}
for $k=0$, we define
$f^{(s,0)}(t,Z_s):=S_s^tf^{(s)}_{0}(Z_s)$, and for the remainder we write
\begin{equation}\label{remainder}
\begin{aligned}
R^{(s,n+1)}(t,Z_s):=\int_0^t\int_0^{t_1}...\int_0^{t_{n}}\hspace{-0.1cm}S_s^{t-t_1}\mathcal{C}_{s,s+2}^\infty S_{s+2}^{t_1-t_2}...S_{s+2n-2}^{t_{n-1}-t_n}\mathcal{C}_{s+2n-2,s+2n}^\infty S_{s+2n}^{t_n-t_{n+1}}f^{\left(s+2n+2\right)}(t_{n+1},Z_s)&\\
\,dt_{n+1}...\,dt_1.&
\end{aligned}
\end{equation}
\subsection{Reduction to term by term convergence}
Here we reduce the convergence proof to term by term convergence of bounded energy and separated collision times observables.

 Recalling \eqref{kinetic energy}, given $R>0$, $\ell\in\mathbb{N}$, we define the energy truncated operators
\begin{equation}\label{velocity truncation of operators}
\mathcal{C}_{\ell,\ell+2}^{N,R}g_{N,\ell+2}:=\mathcal{C}_{\ell,\ell+2}^N\left(g_{N,\ell+2}\mathds{1}_{[E_{\ell+2}\leq R^2]}\right),\quad
\mathcal{C}_{\ell,\ell+2}^{\infty,R} g_{\ell+2}:=\mathcal{C}_{\ell,\ell+2}^\infty \left(g_{\ell+2}\mathds{1}_{[E_{\ell+2}\leq R^2]}\right).
\end{equation}
Consider  $\delta>0$. Given $t\geq 0$ and $k\in\mathbb{N}$, we define   the separated collision times
\begin{equation}\label{separated collision times}
\mathcal{T}_{k,\delta}(t):=\left\{(t_1,...,t_k)\in\mathcal{T}_k(t):\quad 0\leq t_{i+1}\leq t_i-\delta,\quad\forall i\in [0,k]\right\},\quad t_{k+1}:=0,\text{ }t_0:=t.
\end{equation}
For the BBGKY hierarchy, we define for $k\in\mathbb{N}$:
\begin{equation}\label{fNRdk}
f_{N,R,\delta}^{(s,k)}(t,Z_s):=\int_{\mathcal{T}_{k,\delta}(t)}T_s^{t-t_1}\mathcal{C}_{s,s+2}^{N,R} T_{s+2}^{t_1-t_2}
...T_{s+2k-2}^{t_{k-1}-t_k}\mathcal{C}_{s+2k-2,s+2k}^{N,R} T_{s+2k}^{t_k}f_{N,0}^{(s+2k)}(Z_s)\,dt_k...\,dt_{1},
\end{equation}
and for $k=0$, we define
$
f_{N,R,\delta}^{(s,0)}(t,Z_s):=T_s^t\left(f_{N,0}\mathds{1}_{[E_s\leq R^2]}\right)(Z_s).$

For the Boltzmann hierarchy, we define for $k\in\mathbb{N}$:
\begin{equation}\label{fRdk}
f_{R,\delta}^{(s,k)}(t,Z_s):=\int_{\mathcal{T}_{k,\delta}(t)}S_s^{t-t_1}\mathcal{C}_{s,s+2}^{\infty,R} S_{s+2}^{t_1-t_2}
...S_{s+2k-2}^{t_{k-1}-t_k}\mathcal{C}_{s+2k-2,s+2k}^{\infty,R} S_{s+2k}^{t_m}f_{0}^{(s+2k)}(Z_s)\,dt_k...\,dt_{1},
\end{equation}
and for $k=0$, we define
$
f_{R,\delta}^{(s,0)}(t,Z_s):=S_s^t\left(f_{0}\mathds{1}_{[E_s\leq R^2]}\right)(Z_s).$

Given $\phi_s\in C_c(\mathbb{R}^{ds})$ and $k\in\mathbb{N}\cup\{0\}$, let us write
\begin{align}
I_{s,k,R,\delta}^N(t)(X_s):=I_{\phi_s}f_{N,R,\delta}^{(s,k)}(t)(X_s)=\int_{B_R^{ds}}\phi_s(V_s)f_{N,R,\delta}^{(s,k)}(t,X_s,V_s)\,dV_s,\label{bbgky truncated time}
\end{align}
\begin{align}
I_{s,k,R,\delta}^\infty(t)(X_s):=I_{\phi_s}f_{R,\delta}^{(s,k)}(t)(X_s)=\int_{B_R^{ds}}\phi_s(V_s)f_{R,\delta}^{(s,k)}(t,X_s,V_s)\,dV_s.\label{boltzmann truncated time}
\end{align}

Recalling the observables $I_s^N$, $I_s^\infty$ defined in \eqref{def-I-Ns}-\eqref{def-I-s}, the following estimates hold
\begin{proposition}\label{reduction}
For any $s,n\in\mathbb{N}$, $R>1$, $\delta>0$ and $t\in[0,T]$, the following estimates hold:
\begin{equation*}
\|I_s^N(t)-\sum_{k=0}^nI_{s,k,R,\delta}^N(t)\|_{L^\infty_{X_s}}\leq C_{s,\beta_0,\mu_0,T}\|\phi_s\|_{L^\infty_{V_s}}\left(2^{-n}+e^{-\frac{\beta_0}{3}R^2}+\delta C_{d,s,\beta_0,\mu_0,T}^n\right)\|F_{N,0}\|_{N,\beta_0,\mu_0},
\end{equation*}
\begin{equation*}
\|I_s^\infty(t)-\sum_{k=0}^nI_{s,k,R,\delta}^\infty(t)\|_{L^\infty_{X_s}}\leq C_{s,\beta_0,\mu_0,T}\|\phi_s\|_{L^\infty_{V_s}}\left(2^{-n}+e^{-\frac{\beta_0}{3}R^2}+\delta C_{d,s,\beta_0,\mu_0,T}^n\right)\|F_{0}\|_{\infty,\beta_0,\mu_0}.
\end{equation*}
\end{proposition}
\begin{proof}
For the proof, one  needs to successively perform the reductions described above using the a-priori bounds of Section \ref{sec_local} and connect them through the triangle inequality. For  the reduction to finitely many terms and for the energy truncation see Propositions 7.1.1., 7.2.1. in \cite{gallagher},  and for the  time separation part see \cite{thesis}.
\end{proof}
Proposition \ref{reduction} and triangle inequality imply that the convergence proof reduces to controlling the differences
$ I_{s,k,R,\delta}^N(t)-I_{s,k,R,\delta}^\infty(t)$. However  obtaining such a control requires some delicate analysis because of possible recollisions of the  backwards interaction flow.

\section{Geometric estimates} 
\label{sec_geometric}

In this section we provide the crucial geometric estimates, many of them novel, which  will be of fundamental importance in  eliminating recollisions of the backwards interaction flow in Section \ref{sec_stability} and Section \ref{sec_elimination}. 

Let us introduce some notation which we will be using from now on.  For $w\in\mathbb{R}^d$, $y\in\mathbb{R}^d\setminus\{0\}$ and $\rho>0$, we write $K_\rho^d(w,y)$ for the closed $d$-dimensional cylinder of center $w$, direction $y$ and radius $\rho$.
In case we do not need to specify the center and direction we will just be writing $K_\rho^d$ for convenience. 

\subsection{Spherical estimates}\label{subsec:spherical}
Here, we derive the spherical estimates which will enable us to control pre-collisional configurations. We will strongly rely on the following estimate, see Lemma 4 in \cite{denlinger}  for the proof.
\begin{lemma}\label{Ryan's lemma} Given $\rho,r>0$  the following estimate holds for the $d$-spherical measure of radius $r>0$:
 $$\left|\mathbb{S}_r^{d-1}\cap K_\rho^d\right|_{\mathbb{S}_r^{d-1}}\lesssim r^{d-1}\min\left\{1,\left(\displaystyle\frac{\rho}{r}\right)^{\frac{d-1}{2}}\right\}.$$
\end{lemma}
Integrating this estimate we obtain the following result, which will be used in Section \ref{sec_stability}:
\begin{proposition}\label{spherical estima} Given $0<\rho\leq 1\leq R$, the following estimate holds:
$$|B_{R}^{d}\cap K_\rho^d|_{d}\lesssim R^{d}\rho^{\frac{d-1}{2}}.$$
\end{proposition}
\begin{proof}
Using Lemma \ref{Ryan's lemma}, we obtain
\begin{equation}\label{estimate with min}
\begin{aligned}
|B_R^d\cap K_\rho^d|_d&\simeq \int_0^R |\mathbb{S}_r^{d-1}\cap K_\rho^d|_{\mathbb{S}_r^{d-1}}\,dr\lesssim \int_0^Rr^{d-1}\min\left\{1,(\frac{\rho}{r})^{\frac{d-1}{2}}\right\}\,dr\\
&\leq\int_0^\rho r^{d-1}\,dr+\rho^{\frac{d-1}{2}}\int_0^R r^{\frac{d-1}{2}}\,dr\simeq \rho^d+\rho^{\frac{d-1}{2}}R^{\frac{d+1}{2}},\quad\text{since }d\geq 2\\
&\leq R^{d}\rho^{\frac{d-1}{2}},\quad\text{since } 0<\rho\leq 1\leq R.
\end{aligned}
\end{equation}
\end{proof}
We now obtain  new geometric estimates which will be essential to derive the ellipsoidal estimates, enabling  us to control post-collisional configurations. To achieve those estimates we strongly rely on the following  representation of $\mathbb{S}_1^{2d-1}$:
\begin{equation}\label{decomposition of the sphere}
\mathbb{S}_1^{2d-1}=\left\{ (\omega_1,\omega_2)\in\mathbb{R}^d\times B_1^d: \omega_1\in\mathbb{S}_{\sqrt{1-|\omega_2|^2}}^{d-1}\right\}.
\end{equation}
\begin{lemma}\label{new lemma} For any $r,\rho>0$,
the following estimates hold for the $(2d-1)$-spherical measure
$$\left|\mathbb{S}_r^{2d-1}\cap\left(K_\rho^d\times\mathbb{R}^d\right)\right|_{\mathbb{S}_r^{2d-1}},\text{ }\left|\mathbb{S}_r^{2d-1}\cap\left(\mathbb{R}^d\times K_\rho^d\right)\right|_{\mathbb{S}_r^{2d-1}}\lesssim r^{2d-1}\min\left\{1,(\frac{\rho}{r})^{\frac{d-1}{2}}\right\}.$$
\end{lemma}
\begin{proof} By symmetry it suffices to prove the  estimate when intersecting the sphere with $K_\rho^d\times\mathbb{R}^d$. Also, after rescaling we may assume $r=1$. 
The idea is to integrate Lemma \ref{Ryan's lemma} using the representation \eqref{decomposition of the sphere}. In particular by \eqref{decomposition of the sphere} and Lemma \ref{Ryan's lemma}, we have
\begin{align}
&\left|\mathbb{S}_1^{2d-1}\cap\left(K_\rho^d\times\mathbb{R}^d\right)\right|_{\mathbb{S}_1^{2d-1}}
=\int_{B_1^{d}}\left|\mathbb{S}_{\sqrt{1-|\omega_2|^2}}^{d-1}\cap K_\rho^d\right|_{\mathbb{S}_{\sqrt{1-|\omega_2|^2}}^{d-1}}\,d\omega_2\nonumber\\
&\lesssim \int_{B_1^d}(1-|\omega_2|^2)^{\frac{d-1}{2}}\min\left\{1,\left(\frac{\rho}{\sqrt{1-|\omega_2|^2}}\right)^{\frac{d-1}{2}}\right\}\,d\omega_2\nonumber\\
&\lesssim  \int_0^1 s^{d-1}(1-s^2)^{\frac{d-1}{2}}\min\left\{1,\left(\frac{\rho}{\sqrt{1-s^2}}\right)^{\frac{d-1}{2}}\right\}\,ds.\label{last term}
\end{align}
Let us write
$
I(\rho):= \displaystyle\int_0^1 s^{d-1}(1-s^2)^{\frac{d-1}{2}}\min\left\{1,\left(\frac{\rho}{\sqrt{1-s^2}}\right)^{\frac{d-1}{2}}\right\}\,ds.
$
In the case where $\rho\geq 1$, we have
\begin{equation}\label{rho>1}
I(\rho)\lesssim\int_0^1s^{d-1}(1-s^2)^{\frac{d-1}{2}}\,ds\simeq 1.
\end{equation}

Assume now $0<\rho<1$. Then, we may decompose $I(\rho)$ as follows:
\begin{equation}\label{integral decomposition}
I(\rho)=\int_0^{\sqrt{1-\rho^2}}s^{d-1}(1-s^2)^{\frac{d-1}{2}}\left(\frac{\rho}{\sqrt{1-s^2}}\right)^{\frac{d-1}{2}}\,ds+\int_{\sqrt{1-\rho^2}}^1s^{d-1}(1-s^2)^{\frac{d-1}{2}}\,ds.
\end{equation}

Performing the change of variables $u=1-s^2$, equation \eqref{integral decomposition} can be written as:
\begin{align}
I(\rho)&=\frac{1}{2}\rho^{\frac{d-1}{2}}\int_{\rho^2}^1(1-u)^{\frac{d-2}{2}}u^{\frac{d-1}{4}}\,du+\frac{1}{2}\int_0^{\rho^2}(1-u)^{\frac{d-2}{2}}u^{\frac{d-1}{2}}\,du\nonumber\\
&\overset{(d\geq 2)}\lesssim \rho^{\frac{d-1}{2}}\int_{\rho^2}^1u^{\frac{d-1}{4}}\,du+\int_0^{\rho^2}u^{\frac{d-1}{2}}\,du\simeq\rho^{\frac{d-1}{2}}\left(1-\rho^{\frac{d+3}{2}}\right)+\rho^{d+1}\lesssim\rho^{\frac{d-1}{2}},\label{rho<1}
\end{align}
since $\rho<1$.
Combining \eqref{last term}-\eqref{rho>1} and \eqref{rho<1}, we obtain the result.
\end{proof}
In the same spirit as in Lemma \ref{new lemma}, we obtain the following estimate for the intersection of $\mathbb{S}_1^{2d-1}$ with the strip:
\begin{equation}\label{strip general}
W_{\rho,\mu,\lambda}^{2d}:=\{(\omega_1,\omega_2)\in\mathbb{R}^{2d}:|\mu \omega_1-\lambda \omega_2|\leq\rho\},\text{ where }\mu,\lambda\neq 0.
\end{equation}
\begin{lemma}\label{new lemma strip} For any $r,\rho>0$ 
the following estimate holds for the $(2d-1)$-spherical measure:

$$\left|\mathbb{S}_r^{2d-1}\cap W_{\rho,\mu,\lambda}^{2d}\right|_{\mathbb{S}_r^{2d-1}}\lesssim r^{2d-1}\min\left\{1,\left(\frac{\rho}{|\mu| r}\right)^{\frac{d-1}{2}},\left(\frac{\rho}{|\lambda| r}\right)^{\frac{d-1}{2}}\right\}.$$
\end{lemma}
\begin{proof} The proof follows the same steps as the proof of Lemma \ref{new lemma} after noticing that
\begin{align}
W_{\rho,\mu,\lambda}^{2d}&=\{(\omega_1,\omega_2)\in\mathbb{R}^{2d}: \omega_1\in B_{\rho/|\mu|}^d(\lambda\mu^{-1}\omega_2)\}\subseteq \{(\omega_1,\omega_2)\in\mathbb{R}^{2d}: \omega_1\in K_{\rho/|\mu|}^d(\lambda\mu^{-1}\omega_2)\},\nonumber\\
W_{\rho,\mu,\lambda}^{2d}&=\{(\omega_1,\omega_2)\in\mathbb{R}^{2d}: \omega_1\in B_{\rho/|\mu|}^d(\mu\lambda^{-1}\omega_2)\}\subseteq \{(\omega_1,\omega_2)\in\mathbb{R}^{2d}: \omega_1\in K_{\rho/|\lambda|}^d(\mu\lambda^{-1}\omega_2)\},\nonumber
\end{align}
where given $\omega_2\in\mathbb{R}^d$, $K_{\rho/|\mu|}^d(\lambda\mu^{-1} \omega_2),K_{\rho/|\lambda|}^d(\mu\lambda^{-1} \omega_2)$ are any cylinders of radius $\rho/|\mu|, \rho/|\lambda|$ centered at $\lambda\mu^{-1} \omega_2,\mu\lambda^{-1} \omega_2$ respectively.  
\end{proof}
\subsection{The transition map}
Now, we construct a transition map which will allow us to  control  post-collisional configurations using some appropriate ellipsoidal estimates developed in Subsection \ref{subsec:ellipsoidal}. 
We first  introduce some notation. 
Given $v_1,v_2,v_3\in\mathbb{R}^{2d}$, we define 
$$\Omega=\{\bm{\omega}=(\omega_1,\omega_2)\in\mathbb{R}^{2d}:|\omega_1|^2+|\omega_2|^2<\frac{3}{2}\text{ and }b(\omega_1,\omega_2,v_2-v_1,v_3-v_1)>0\},$$
where $b(\omega_1,\omega_2,v_2-v_1,v_3-v_1)$ is the cross-section given in \eqref{cross}, and 
\begin{equation}\label{set of post angles}
\mathcal{S}_{v_1,v_2,v_3}^+:=\mathbb{S}_1^{2d-1}\cap\Omega=\{\bm{\omega}=(\omega_1,\omega_2)\in\mathbb{S}_1^{2d-1}:b(\omega_1,\omega_2,v_2-v_1,v_3-v_1)>0\}.
\end{equation}
We  also define the smooth map 
$\Psi(\nu_1,\nu_2)=|\nu_1|^2+|\nu_2|^2+|\nu_1-\nu_2|^2$
and the $(2d-1)$-ellipsoid
\begin{equation}\label{ellipsoid}
\mathbb{E}_1^{2d-1}:=[\Psi=1]=\left\{(\nu_1,\nu_2)\in\mathbb{R}^{2d}:|\nu_1|^2+|\nu_2|^2+|\nu_1-\nu_2|^2=1\right\}.
\end{equation}
\begin{proposition}\label{transition lemma} Consider $v_1,v_2,v_3\in\mathbb{R}^d$ and $r>0$ such that 
\begin{equation}\label{reflection assumption}
|v_1-v_2|^2+|v_1-v_3|^2+|v_2-v_3|^2=r^2.
\end{equation}

We define the transition map $\mathcal{J}_{v_1,v_2,v_3}:\Omega\to\mathbb{R}^{2d}\setminus\left\{r^{-1}\begin{pmatrix}
 v_1-v_2\\
 v_1-v_3
 \end{pmatrix}\right\}$ by
\begin{equation}\label{definition of v}\bm{\nu}=\begin{pmatrix}
\nu_{1}\\
\nu_{2}\\
\end{pmatrix}=\mathcal{J}_{v_1,v_2,v_3}(\bm{\omega})
:=\frac{1}{r}
\begin{pmatrix}
v_1^*-v_2^*\\
v_1^*-v_3^*\\
\end{pmatrix},\quad\bm{\omega}=(\omega_1,\omega_2)\in\Omega.\footnote{by a small abuse of notation we extend the collisional operator $T_{\omega_1,\omega_2}$ for $(\omega_1,\omega_2)\in\Omega$, see Section \ref{sec_collision}.}
\end{equation}
\begin{enumerate}[(i)]
\item $\mathcal{J}_{v_1,v_2,v_3}$ is smooth in $\Omega$ with bounded derivative uniformly in $r$  i.e.
\begin{equation}\label{matrix derivative lemma}
\|D\mathcal{J}_{v_1,v_2,v_3}(\bm{\omega})\|_\infty\leq C_d,\quad\forall\bm{\omega}\in\Omega.
\end{equation}
\vspace{0.2cm}
\item The Jacobian of $\mathcal{J}_{v_1,v_2,v_3}$ is given by:
\begin{equation}\label{equality for jacobian}\jac(\mathcal{J}_{v_1,v_2,v_3})(\bm{\omega})\simeq r^{-2d}\frac{b^{2d}(\omega_1,\omega_2,v_2-v_1,v_3-v_1)}{(1+\langle\omega_1,\omega_2\rangle)^{2d+1}}> 0,\quad\forall\bm{\omega}=(\omega_1,\omega_2)\in\Omega.
\end{equation}
Moreover, for any $\bm{\omega}=(\omega_1,\omega_2)\in \Omega$, there holds the estimate:
\begin{equation}\label{jacobian}
 \jac(\mathcal{J}_{v_1,v_2,v_3})(\bm{\omega})\thickapprox r^{-2d}b^{2d}(\omega_1,\omega_2,v_2-v_1,v_3,v_1).
\end{equation}
\item The map
 $\mathcal{J}_{v_1,v_2,v_3}:\mathcal{S}_{v_1,v_2,v_3}^+\to\mathbb{E}_1^{2d-1}\setminus\left\{r^{-1}\begin{pmatrix}
 v_1-v_2\\
 v_1-v_3
 \end{pmatrix}\right\}$
 is bijective. Morever, there holds
 \begin{equation}\label{S=level sets}
 \mathcal{S}_{v_1,v_2,v_3}^+=[\Psi\circ\mathcal{J}_{v_1,v_2,v_3}=1].
 \end{equation}
\item For any measurable $g:\mathbb{R}^{2d}\to[0+\infty]$,  there holds the estimate
\begin{equation}\label{substitution estimate}
\int_{\mathcal{S}_{v_1,v_2,v_3}^+}(g\circ\mathcal{J}_{v_1,v_2,v_3})(\bm{\omega})|\jac\mathcal{J}_{v_1,v_2,v_3}(\bm{\omega})|\,d\bm{\omega}\lesssim\int_{\mathbb{E}_1^{2d-1}}g(\bm{\nu})\,d\bm{\nu}.
\end{equation}
\end{enumerate}
\end{proposition}
\begin{proof} For convenience, let us use the notation\footnote{by a small abuse of notation we write $\langle\cdot\text{ },\cdot\rangle$ for the inner product in $\mathbb{R}^{2d}$ as well.}:
\begin{equation}\label{vector notation}
\bm{\nu}=\begin{pmatrix}
\nu_1\\\nu_2
\end{pmatrix},\quad \bm{v}=\begin{pmatrix}
v_1-v_2\\v_1-v_3
\end{pmatrix},\quad\bm{\omega}=\begin{pmatrix}
\omega_1\\\omega_2
\end{pmatrix},\quad\pi(\bm{\omega})=\langle\omega_1,\omega_2\rangle,\quad c:-\frac{\langle\bm{\omega},\bm{v}\rangle}{1+\pi(\bm{\omega})}.
\end{equation}
 By \eqref{definition of v} and \eqref{formulas of collision}, we have
\begin{equation}\label{jac after sur}
\mathcal{J}_{v_1,v_2,v_3}(\bm{\omega})=r^{-1}\left(\bm{v}+c\bm{A}\bm{\omega}\right),\quad\text{where } \bm{A}=\begin{pmatrix}
2I_d&I_d\\I_d&2I_d
\end{pmatrix}.
\end{equation}
Notice that $\mathcal{J}_{v_1,v_2,v_3}$ maps in $\mathbb{R}^{2d}\setminus\{r^{-1}\bm{v}\}$. Indeed,  assume that $\mathcal{J}_{v_1,v_2,v_3}(\bm{\omega})=r^{-1}\bm{v}$ for some $\bm{\omega}\in\Omega$.
Since  $\bm{A}$  is invertible and $\bm{\omega}\neq 0$, \eqref{jac after sur} implies
$c=0\Rightarrow\langle\bm{\omega},\bm{v}\rangle=0,$
which is a contradiction, since $\bm{\omega}\in\Omega$. 

\textit{(i)}: Let us  calculate the derivative of $\mathcal{J}_{v_1,v_2,v_3}$. Using \eqref{jac after sur}, we obtain
\begin{equation}\label{derivative of J}
D\mathcal{J}_{v_1,v_2,v_3}(\bm{\omega})=r^{-1}\bm{A}\left(c\bm{I_{2d}}+\bm{\omega}\nabla_{\bm{\omega}}^Tc\right).
\end{equation}
Using notation from \eqref{vector notation}, we obtain
\begin{equation}\label{nabla of c}\nabla_{\bm{\omega}}c=-\frac{\bm{v}}{1+\pi(\bm{\omega})}+\frac{\langle\bm{\omega},\bm{v}\rangle\bm{\tilde{\omega}}}{\left(1+\pi(\bm{\omega})\right)^2},\end{equation}
where $\bm{\tilde{\omega}}:=\nabla_{\bm{\omega}}\pi(\bm{\omega})=\begin{pmatrix}
\omega_2\\
\omega_1
\end{pmatrix}$.
Combining \eqref{derivative of J}-\eqref{nabla of c}, we obtain
\begin{equation}\label{expanded derivative}
D\mathcal{J}_{v_1,v_2,v_3}(\bm{\omega})=r^{-1}\left(-\frac{\langle\bm{\omega},\bm{v}\rangle\bm{A}}{1+\pi(\bm{\omega})}-\frac{\bm{A}\bm{\omega}\bm{v}^T}{1+\pi(\bm{\omega})}+\frac{\langle\bm{\omega},\bm{v}\rangle\bm{A}\bm{\omega}\bm{\tilde{\omega}}^T}{\left(1+\pi(\bm{\omega})\right)^2}\right).
\end{equation}
Recall we have assumed $\bm{\omega}\in\Omega\Rightarrow |\omega_1|^2+|\omega_2|^2<\displaystyle\frac{3}{2}$, so Cauchy-Schwartz inequality implies
\begin{equation}\label{ineq on omegas in domain}
\frac{1}{4}<1+\pi(\bm{\omega})<\frac{7}{4},
\end{equation}
therefore $\mathcal{J}_{v_1,v_2,v_3}$ is differentiable in $\Omega$. It is clear from \eqref{expanded derivative}-\eqref{ineq on omegas in domain} that $\mathcal{J}_{v_1,v_2,v_3}$ is in fact smooth. Moreover using \eqref{expanded derivative}, bound \eqref{matrix derivative lemma} follows after using Cauchy-Schwartz inequality, the fact that $\bm{\omega}\in\Omega$, \eqref{reflection assumption}, \eqref{ineq on omegas in domain} and \eqref{reflection assumption}.

\textit{(ii)}: To calculate the Jacobian, we use \eqref{derivative of J} and apply Lemma \ref{linear algebra lemma} (see Appendix), to obtain
\begin{equation}\label{jac with nabla}
\jac(\mathcal{J}_{v_1,v_2,v_3})(\bm{\omega})=\det(r^{-1}\bm{A})\det(c\bm{I_{2d}}+\bm{\omega}\nabla_{\bm{\omega}}^Tc)\simeq
r^{-2d}c^{2d}\left(1+c^{-1}\langle\bm{\omega},\nabla_{\bm{\omega}}c\rangle\right).
\end{equation}
Recalling  $\bm{\tilde{\omega}}=\nabla_{\bm{\omega}}\pi(\bm{\omega})=\begin{pmatrix}
\omega_2\\
\omega_1
\end{pmatrix}$, we obtain $
c^{-1}\langle\bm{\omega},\nabla_{\bm{\omega}}c\rangle=\left(1-\displaystyle\frac{2\pi(\bm{\omega})}{1+\pi(\bm{\omega})}\right).$
Hence \eqref{jac with nabla} and \eqref{relation cross-c} imply \eqref{equality for jacobian}.
To obtain \eqref{jacobian}, we combine \eqref{equality for jacobian} and estimate \eqref{ineq on omegas in domain}.

\textit{(iii)}:   Let us first show that $\mathcal{J}_{v_1,v_2,v_3}:\mathcal{S}_{v_1,v_2,v_3}^+\to\mathbb{E}_1^{2d-1}\setminus\{r^{-1}\bm{v}\}$.
Fix $\bm{\omega}=(\omega_1,\omega_2)\in\mathcal{S}_{v_1,v_2,v_3}^+$. Using conservation of relative velocities \eqref{rel.vel} and \eqref{reflection assumption},
we get
\begin{equation*}
|\nu_{1}|^2+|\nu_{2}|^2+|\nu_{1}-\nu_{2}|^2=\frac{|v_1^*-v_2^*|^2+|v_1^*-v_3^*|^2+|v_2^*-v_3^*|^2}{r^2}=1,
\end{equation*}
thus $\mathcal{J}_{v_1,v_2,v_3}:\mathcal{S}_{v_1,v_2,v_3}^+\to\mathbb{E}_1^{2d-1}\setminus\{r^{-1}\bm{v}\}.$
To prove injectivity, let $\bm{\omega},\bm{\omega'}\in\mathcal{S}_{v_1,v_2,v_3}^+$
with $\mathcal{J}_{v_1,v_2,v_3}(\bm{\omega})=\mathcal{J}_{v_1,v_2,v_3}(\bm{\omega'}).$ 
Since $\bm{A}$ is invertible, \eqref{jac after sur} implies 
$c\omega=c'\omega'$
where $c'=c_{\omega_1',\omega_2',v_1,v_2,v_3}$.
Since $\omega,\omega'\in\Omega$, we have $c,c'\neq 0$ thus $\bm{\omega}=c^{-1}c'\bm{\omega'}$.  Since $\bm{\omega},\bm{\omega}\in\mathcal{S}_{v_1,v_1,v_3}^+$, we obtain $c=c'$, thus $\bm{\omega}=\bm{\omega'}.$

To prove surjectivity, consider $\bm{\nu}\in\mathbb{E}_1^{2d-1}\setminus\{r^{-1}\bm{v}\}$. 
 and define
 $$\bm{\omega}:=\frac{-\sign(\langle{A}^{-1}(\bm{v}-r\bm{\nu}),\bm{v}\rangle)}{\sqrt{\langle\bm{A}^{-1}(\bm{v}-r\bm{\nu}),\bm{v}\rangle-\langle\bm{A}^{-1}_1(\bm{v}-r\bm{\nu}),\bm{A}^{-1}_2(\bm{v}-r\bm{\nu})\rangle}}\bm{A}^{-1}(\bm{v}-r\bm{\nu}).$$
By \eqref{reflection assumption} and the fact that $\bm{\nu}\in\mathbb{E}_1^{2d-1}\setminus\{r^{-1}\bm{v}\}$, we have that $\bm{\omega}\in\mathcal{S}_{v_1,v_2,v_3}^{+}$ is the unique solution in $\Omega$ of $\mathcal{J}_{v_1,v_2,v_3}(\bm{\omega})=\bm{\nu}$. Relation \eqref{S=level sets} follows from the fact $\mathcal{J}_{v_1,v_2,v_3}:\Omega\to\mathbb{R}^{2d}\setminus\{r^{-1}\bm{v}\}$ and the previous consideration.\\
  \textit{(iv)}: We easily calculate
$
4\Psi(\bm{\nu})\leq |\nabla\Psi(\bm{\nu})|^2\leq 16\Psi(\bm{\nu}),$ for all $\bm{\nu}\in\mathbb{R}^{2d},
$
so $\nabla\Psi(\bm{\nu})\neq 0,$ for all $ \bm{\nu}\in [\frac{1}{2}<\Psi<\frac{3}{2}]$.
To prove the estimate we will rely on Lemma \ref{substitution lemma} (see Appendix). We have
\begin{align}
&\int_{\mathcal{S}_{v_1,v_2,v_3}^+}(g\circ\mathcal{J}_{v_1,v_2,v_3})(\bm{\omega})|\jac\mathcal{J}_{v_1,v_2,v_3}(\bm{\omega})|\frac{|\nabla\Psi(\mathcal{J}_{v_1,v_2,v_3}(\bm{\omega}))|}{|\nabla(\Psi\circ\mathcal{J}_{v_1,v_2,v_3})(\bm{\omega})|}\,d\bm{\omega}\nonumber\\
&=\int_{[\Psi\circ\mathcal{J}_{v_1,v_2,v_3}=1]}(g\circ\mathcal{J}_{v_1,v_2,v_3})(\bm{\omega})|\jac\mathcal{J}_{v_1,v_2,v_3}(\bm{\omega})|\frac{|\nabla\Psi(\mathcal{J}_{v_1,v_2,v_3}(\bm{\omega}))|}{|\nabla(\Psi\circ\mathcal{J}_{v_1,v_2,v_3})(\bm{\omega})|}\,d\bm{\omega}\label{S= level sets app}\\
&=\int_{[\Psi=1]}g(\bm{\nu})\mathcal{N}_{\mathcal{J}_{v_1,v_2,v_3}}(\bm{\nu},[\Psi\circ\mathcal{J}_{v_1,v_2,v_3}=1])\,d\bm{\nu}\label{application of sub lemma}\\
&=\int_{\mathbb{E}_1^{2d-1}}g(\bm{\nu})\mathcal{N}_{\mathcal{J}_{v_1,v_2,v_3}}(\bm{\nu},\mathcal{S}_{v_1,v_2,v_3}^+)\,d\bm{\nu},\label{apply ellipsoid}
\end{align}
where to obtain \eqref{S= level sets app} we use \eqref{S=level sets}, to obtain \eqref{application of sub lemma} we  use Lemma \ref{substitution lemma}, to obtain \eqref{apply ellipsoid} we use \eqref{ellipsoid} and \eqref{S=level sets}. Moreover, by the chain rule and \eqref{matrix derivative lemma}, we obtain
\begin{equation*}
\frac{|\nabla(\Psi\circ \mathcal{J}_{v_1,v_{2},v_{3}})(\bm{\omega})|}{|\nabla\Psi(\mathcal{J}_{v_1,v_2,v_3}(\bm{\omega}))|}=\frac{|D^T\mathcal{J}_{v_1,v_{2},v_{3}}(\bm{\omega})\nabla\Psi(\mathcal{J}_{v_1,v_{2},v_{3}}(\bm{\omega}))|}{|\nabla\Psi(\mathcal{J}_{v_1,v_2,v_3}(\bm{\omega}))|}
= C_d \|D\mathcal{J}_{v_1,v_2,v_3}(\bm{\omega})\|_\infty\leq C_d,
\end{equation*}
and  \eqref{substitution estimate} follows, since $g\geq 0$.
\end{proof}
\subsection{Ellipsoidal estimates}\label{subsec:ellipsoidal}
Now, we derive the ellipsoidal estimates which will enable us to control post-collisional configurations.

\begin{lemma}\label{matrix lemma} Let $v_1,v_2,v_3\in\mathbb{R}^d$ and $r>0$ satisfying
$|v_1-v_2|^2+|v_1-v_3|^2+|v_2-v_3|^2=r^2.$
Denoting $(\nu_1,\nu_2)=\mathcal{J}_{v_1,v_2,v_3}(\omega_1,\omega_2)$ and considering $\rho>0$, the following holds:
\begin{equation*}
\begin{aligned}
\begin{pmatrix}
v_1^*\\
v_2^*
\end{pmatrix}
\in \bm{K_\rho}&\Leftrightarrow 
\begin{pmatrix}
\nu_{1}\\
\nu_{2}
\end{pmatrix}\in \bm{S_{12}^{-1}\bar{K}_{\rho/r}},\quad
\begin{pmatrix}
v_1^*\\
v_3^*
\end{pmatrix}
\in \bm{K_\rho}&\Leftrightarrow 
\begin{pmatrix}
\nu_{1}\\
\nu_{2}
\end{pmatrix}\in \bm{S_{13}}^{-1}\bm{\bar{K}_{\rho/r}},
\end{aligned}
\end{equation*}
\begin{equation}\label{matrix}
\bm{S_{12}}=\begin{pmatrix}
I_d & I_d\\
-2I_d &I_d\\
\end{pmatrix},\quad
\bm{S_{13}}=\begin{pmatrix}
 I_d &I_d\\
I_d &-2I_d\\
\end{pmatrix},
\end{equation}
and $\bm{K_\rho}$ is either of the form  $K_\rho^d\times\mathbb{R}^d$ or $\mathbb{R}^d\times K_\rho^d$ 
while $\bm{\bar{K}_{\rho/r}}$ is either of the form  $\bar{K}_{\rho/r}^{d}\times\mathbb{R}^d$ or $\mathbb{R}^d\times \bar{K}_{\rho/r}^{d}$ respectively, and $K_{\rho}^d$, $\bar{K}_{\rho/r}^{d}$ are $d$-cylinders or radius $\rho$ and $\rho/r$ respectively.
\end{lemma}
\begin{proof}
 Using \eqref{definition of v} to eliminate $c\omega_{1},c\omega_{2}$ from \eqref{formulas of collision}, we obtain
\begin{equation*}
\begin{aligned}
v_1^*&=\frac{v_1+v_2+v_3}{3}+\frac{r}{3}(\nu_{1}+\nu_{2}),\\
v_2^*&=\frac{v_1+v_2+v_3}{3}+\frac{r}{3}(-2\nu_{1}+\nu_{2}),\\
v_3^*&=\frac{v_1+v_2+v_3}{3}+\frac{r}{3}(\nu_{1}-2\nu_{2}).
\end{aligned}
\end{equation*}
The conclusion is immediate after a translation and a dilation.
\end{proof}
Recalling $\mathbb{E}_1^{2d-1}$ from \eqref{ellipsoid}, one can see that $\bm{S_{12}}(\mathbb{E}_1^{2d-1})=\bm{S_{13}}(\mathbb{E}_1^{2d-1})$. We will denote
\begin{equation}\label{S-ellipsoid}
\mathcal{S}:=\bm{S_{12}}(\mathbb{E}_1^{2d-1})=\bm{S_{13}}(\mathbb{E}_1^{2d-1})=\left\{(y_1,y_2)\in\mathbb{R}^{2d}:|y_1|^2+|y_2|^2+\langle y_1,y_2\rangle=\frac{3}{2}\right\}.
\end{equation}

The following result  will allow us to derive the ellipsoidal estimates from the spherical estimates. 

\begin{lemma}\label{embeddings} There  exist linear bijections $T_1,T_2,P_1,P_2:\mathbb{R}^{2d}\to\mathbb{R}^{2d}$ and $c>0$, with the following properties:\\
{\it(i)} $T_1(\mathcal{S})=\mathbb{S}_1^{2d-1}$ and for any $\rho>0$, there holds $T_{1}(\bar{K}_\rho^d\times\mathbb{R}^d)\subseteq\widetilde{K}_{c\rho}^d\times\mathbb{R}^d$,\\
{\it(ii)} $T_2(\mathcal{S})=\mathbb{S}_1^{2d-1}$ and for any $\rho>0$, there holds:  $T_{2}(\mathbb{R}^d\times \bar{K}_\rho^d)\subseteq\widetilde{K}_{c\rho}^d\times\mathbb{R}^d$,\\
{\it(iii)} $P_1(\mathbb{E}_1^{2d-1})=\mathbb{S}_1^{2d-1}$ and for any $\rho>0$, there holds: $P_1(\bar{K}_{\rho}^d\times\mathbb{R}^{d})\subseteq \widetilde{K}_{c\rho}^{d}\times\mathbb{R}^{d}$,\\
{\it(iv)} $P_2(\mathbb{E}_1^{2d-1})=\mathbb{S}_1^{2d-1}$ and for any $\rho>0$, there holds: $P_2(\mathbb{R}^{d}\times \bar{K}_{\rho}^d)\subseteq \widetilde{K}_{c\rho}^{d}\times\mathbb{R}^{d}$,\\
where $\bar{K}_\rho^d$ is any $d$-cylinder of radius $\rho$ and $\widetilde{K}_{c\rho}^d$ is a $d$-cylinder of radius $c\rho$ and same direction as $\bar{K}_\rho^d$.
\end{lemma}
\begin{proof}
A direct algebraic calculation shows that the maps given by:
{\small
\begin{equation}
\begin{aligned}
T_1&=\begin{pmatrix}
-\frac{\sqrt{2}}{2}I_d&0\\
\\
\frac{\sqrt{6}}{6}I_d&\frac{\sqrt{6}}{3}I_d
\end{pmatrix},\hspace{0.1cm}
T_2&=\begin{pmatrix}
0&-\frac{\sqrt{2}}{2}I_d\\
\\
\frac{\sqrt{6}}{3}I_d&\frac{\sqrt{6}}{6}I_d
\end{pmatrix},\hspace{0.1cm}
P_1&=\begin{pmatrix}
\frac{\sqrt{6}}{2}I_d&0\\
\\
-\frac{\sqrt{2}}{2}I_d&\sqrt{2}I_d
\end{pmatrix},\hspace{0.1cm}
P_2&=\begin{pmatrix}
0&\frac{\sqrt{6}}{2}I_d\\
\\
\sqrt{2}I_d&-\frac{\sqrt{2}}{2}I_d
\end{pmatrix},\\
\end{aligned}
\end{equation}}
satisfy the properties listed above.
\end{proof}

Now we are ready to apply the results of Subsection \ref{subsec:spherical} to obtain ellipsoidal estimates. 
Recalling from \eqref{strip general} the strip 
$
W_{\rho,1,1}^{2d},
$
we obtain the following ellipsoidal estimates:
\begin{proposition}\label{ellipsoid estimates}
For any $r,\rho>0$, the following estimates hold:
\begin{enumerate}[(i)]
\item $\left|\mathcal{S}\cap\left(\bar{K}_{\rho/r}^d\times\mathbb{R}^d\right)\right|_{\mathcal{S}}\lesssim\min\left\{1,\displaystyle\left(\frac{\rho}{r}\right)^{\frac{d-1}{2}}\right\}$.
\item $\left|\mathcal{S}\cap\left(\mathbb{R}^d\times \bar{K}_{\rho/r}^d\right)\right|_{\mathcal{S}}\lesssim\min\left\{1,\displaystyle\left(\frac{\rho}{r}\right)^{\frac{d-1}{2}}\right\}$.
 \item $\left|\mathbb{E}_1^{2d-1}\cap\left(B_{\rho/r}^d\times\mathbb{R}^d\right)\right|_{\mathbb{E}_1^{2d-1}}\lesssim\min\left\{1,\displaystyle\left(\frac{\rho}{r}\right)^{\frac{d-1}{2}}\right\}$.
 \item $\left|\mathbb{E}_1^{2d-1}\cap\left(\mathbb{R}^d\times B_{\rho/r}^d\right)\right|_{\mathbb{E}_1^{2d-1}}\lesssim\min\left\{1,\displaystyle\left(\frac{\rho}{r}\right)^{\frac{d-1}{2}}\right\}$.
 \item $\left|\mathbb{E}_1^{2d-1}\cap W_{\rho/r,1,1}^{2d}\right|_{\mathbb{E}_1^{2d-1}}\lesssim\min\left\{1,\displaystyle\left(\frac{\rho}{r}\right)^{\frac{d-1}{2}}\right\}$.
 \end{enumerate}
\end{proposition}
\begin{proof}  Let us first provide the proof of \textit{(i)}. 
Lemma \ref{embeddings} asserts  $T_1:\mathcal{S}\to\mathbb{S}_1^{2d-1}$ is a linear bijection such that  $T_1(\bar{K}_{\rho/r}^d\times\mathbb{R}^d)\subseteq \widetilde{K}_{c\rho/r}\times\mathbb{R}^d$, thus substituting $\bm{\theta}=T_1\bm{\omega}$, we have
\begin{align*}
\int_{\mathcal{S}}\mathds{1}_{\bar{K}_{\rho/r}^d\times\mathbb{R}^d}(\bm{\omega})\,d\bm{\omega}&=\int_{\mathcal{S}}\mathds{1}_{T_1(\bar{K}_{\rho/r}^d\times\mathbb{R}^d)}(T_1\bm{\omega})\,d\bm{\omega}\simeq \int_{\mathbb{S}_1^{2d-1}}\mathds{1}_{T_1(\bar{K}_{\rho/r}^d\times\mathbb{R}^d)}(\bm{\theta})\,d\bm{\theta}\nonumber\\
&\lesssim  \int_{\mathbb{S}_1^{2d-1}}\mathds{1}_{\widetilde{K}_{c\rho/r}^d\times\mathbb{R}^d}(\bm{\theta})\,d\bm{\theta}\lesssim \min\left\{1,(\frac{\rho}{r})^{\frac{d-1}{2}}\right\},
\end{align*}
by  Lemma \ref{new lemma}.  The proof for \textit{(ii)} is identical using bijection $T_2$ instead.
For estimates \textit{(iii)} and \textit{(iv)}  we use in a similar way  bijections $P_1$, $P_2$ and the fact that ball $B_{\rho/r}^d$ embeds in a cylinder of the form  $\bar{K}_{\rho/r}^d$.
For estimate \textit{(v)}, recalling notation from \eqref{strip general}, notice that 
$P_1(W_{\eta/r,1,1}^{2d})=W_{\eta/r,\mu,\lambda}^{2d},$ 
for $\mu=(3\sqrt{2}+\sqrt{6})/6$ and $\lambda=-\sqrt{6}/3$. Then the claim comes with a similar argument using Lemma \ref{new lemma strip} instead of Lemma \ref{new lemma}.
\end{proof}

\section{Good configurations and stability}
\label{sec_stability}

In this section we define good configurations and study their stability properties under the adjunction of a collisional pair of particles. Heuristically speaking, given $m\in\mathbb{N}$, a configuration $Z_m\in\mathbb{R}^{2dm}$ is called good configuration if the backwards interaction flow coincides with the backwards free flow. The aim of this section is to investigate conditions under which a given good configuration $Z_m$ remains a good configuration after adding a pair of  particles. This is possible on the complement of a small measure set of particles which is constructed in Proposition \ref{bad set}.
Proposition \ref{measure of B} uses the geometric tools developed in Section \ref{sec_geometric} to derive a measure estimate for this pathological set. 

This section is the heart of our contribution, since we will strongly rely on Proposition \ref{bad set} and Proposition  \ref{measure of B} when we use them inductively to control the differences  of the  BBGKY hierarchy truncated observable, given in \eqref{bbgky truncated time}, and the Boltzmann hierarchy truncated observable, given in \eqref{boltzmann truncated time}.

We recall the cylinder notation introduced at the beginning of Section \ref{sec_geometric}.

\subsection{Adjunction of new particles}  
We start with some definitions on the configurations we are using. Given $m\in\mathbb{N}$ and $\sigma>0$,  recall from \eqref{separated conf space}-\eqref{separated conf} the set of well-separated spatial configurations 
\begin{equation*}
\Delta_m^X(\sigma)=\{\widetilde{X}_m\in\mathbb{R}^{dm}: |\widetilde{x}_i-\widetilde{x}_j|>\sigma,\quad\forall 1\leq i<j\leq m\},\quad m\geq 2,\quad \Delta_1^X(\sigma)=\mathbb{R}^{2d},
\end{equation*}
and the set of well separated configurations
\begin{equation*}
\Delta_m(\sigma)=\Delta_m^X(\sigma)\times\mathbb{R}^{dm}.
\end{equation*}
Given $\sigma>0$, $t_0>0$, we define the set of good configurations as:
\begin{equation}\label{good conf def}
G_m(\sigma,t_0)=\left\{Z_m=(X_m,V_m)\in\mathbb{R}^{2dm}:Z_m(t)\in\Delta_m(\sigma),\quad\forall t\geq t_0\right\},
\end{equation}
where $Z_m(t)=(X_m-tV_m,V_m)$ denotes the backwards in time free flow of $Z_m=(X_m,V_m)$.
From  now on, we consider parameters 
$R>>1$ and $0< \delta,\eta,\epsilon_0,\alpha<<1$ satisfying:
\begin{equation}\label{choice of parameters}
 \alpha<<\epsilon_0<<\eta\delta,\quad R\alpha<<\eta\epsilon_0.
\end{equation}
For convenience we choose the parameters in \eqref{choice of parameters} in the very end of the paper, see \eqref{final parameters}.

The following result, see Lemma 12.2.1 in \cite{gallagher} for the proof, is useful for the adjunction of particles to a given configuration. 
 \begin{lemma}\label{adjuction of 1}
  Consider parameters $\alpha,\epsilon_0,R,\eta,\delta$  as in \eqref{choice of parameters} and $\epsilon<<\alpha$. Let $\bar{y}_1,\bar{y}_2\in\mathbb{R}^d$, with $|\bar{y}_1-\bar{y}_2|>\epsilon_0$ and $v_1\in B_R^d$. Then there is a $d$-cylinder $K_\eta^d\subseteq\mathbb{R}^d$, such that for any $Z_2=(y_1,y_2,v_1,v_2)$ with $y_1\in B_\alpha^d(\bar{y}_1)$, $y_2\in B_\alpha^d(\bar{y}_2)$ and  
 $v_2\in B_R^d\setminus K_\eta^d$, we have $Z_2\in G_2(\sqrt{2}\epsilon,0)\cap G_2(\epsilon_0,\delta)$. 
\end{lemma}

\subsection{Stability of good configurations under adjunction of collisional pair}
\label{paragraph stab}
We prove a statement and a measure estimate regarding the stability of good configurations under the adjunction of a collisional pair of particles to any of the initial configurations. 

Recalling the  cross-section $b$  given in \eqref{cross}, given $v\in\mathbb{R}^d$, we denote
\begin{equation}\label{pre-post notation}
\left(\mathbb{S}_1^{2d-1}\times B_R^{2d}\right)^+(v)=\big\{(\omega_1,\omega_2,v_{1},v_{2})\in\mathbb{S}_{1}^{2d-1}\times B_R^{2d}:b(\omega_1,\omega_2,v_{1}-v,v_{2}-v)>0\big\}.
\end{equation}

We prove the following Proposition, which will be the inductive step of the convergence proof. We then provide the corresponding measure estimate.

Recall that given $m\in\mathbb{N}$ and $Z_m\in\mathbb{R}^{2dm}$ we denote as
$Z_m(t)=\left(X_m\left(t\right),V_m\left(t\right)\right)=(X_m-tV_m,V_m)$
the backwards evolution in time of $Z_m$.
In particular, $Z_m(0)=Z_m$.
Recall also  the notation from \eqref{phase space interior} 
$$\mathring{\mathcal{D}}_{m+2,\epsilon}=\left\{Z_{m+2}=(X_{m+2},V_{m+2})\in\mathbb{R}^{2d(m+2)}:d^2(x_i;x_j,x_k)>2\epsilon^2,\quad\forall i<j<k\in\left\{1,...,m+2\right\}\right\}.$$
\begin{proposition}\label{bad set} Consider parameters $\alpha,\epsilon_0,R,\eta,\delta$ as in \eqref{choice of parameters} and $\epsilon<<\alpha$.  Let $m\in\mathbb{N}$, $\bar{Z}_m=(\bar{X}_m,\bar{V}_m)\in G_m(\epsilon_0,0)$, $\ell\in\{1,...,m\}$ and $X_m\in B_{\alpha/2}^{dm}(\bar{X}_m)$.  Then there is a subset $\mathcal{B}_{\ell}(\bar{Z}_m)\subseteq (\mathbb{S}_1^{2d-1}\times B_R^{2d})^+(\bar{v}_\ell)$ such that:
\begin{enumerate}[(i)]
\item For any $(\omega_1,\omega_2,v_{m+1},v_{m+2})\in (\mathbb{S}_1^{2d-1}\times B_R^{2d})^+(\bar{v}_{\ell})\setminus\mathcal{B}_{\ell}(\bar{Z}_m)$, one has:
\begin{align}
Z_{m+2}(t)&\in\mathring{\mathcal{D}}_{m+2,\epsilon},\quad\forall t\geq 0,\label{pre-0}\\
Z_{m+2}&\in G_{m+2}(\epsilon_0/2,\delta),\label{pre-delta}\\
\bar{Z}_{m+2}&\in G_{m+2}(\epsilon_0,\delta),\label{pre-delta-}
\end{align}
where 
\begin{equation}\label{pre-notation}
\begin{aligned}
Z_{m+2}&=(x_1,...,x_{\ell},...,x_m,x_{m+1},x_{m+2},\bar{v}_1,...,\bar{v}_{\ell},...,\bar{v}_m,v_{m+1},v_{m+2}),\\
x_{m+1}&=x_{\ell}-\sqrt{2}\epsilon\omega_1,\quad x_{m+2}=x_{\ell}-\sqrt{2}\epsilon\omega_2,\\
\bar{Z}_{m+2}&=(\bar{x}_1,...,\bar{x}_{\ell},...,\bar{x}_m,\bar{x}_{\ell},\bar{x}_{\ell},\bar{v}_1,...,\bar{v}_{\ell},...,\bar{v}_m,v_{m+1},v_{m+2}).
\end{aligned}
\end{equation}

\item For any $(\omega_1,\omega_2,v_{m+1},v_{m+2})\in (\mathbb{S}_1^{2d-1}\times B_R^{2d})^+(\bar{v}_{\ell})\setminus\mathcal{B}_{\ell}(\bar{Z}_m)$, one has:
\begin{align}
Z_{m+2}^*(t)&\in\mathring{\mathcal{D}}_{m+2,\epsilon},\quad\forall t\geq 0,\label{post-0}\\
Z_{m+2}^*&\in G_{m+2}(\epsilon_0/2,\delta),\label{post-delta}\\
\bar{Z}_{m+2}^*&\in G_{m+2}(\epsilon_0,\delta),\label{post-delta-}
\end{align}
where 
\begin{equation}\label{post-notation}
\begin{aligned}
Z_{m+2}^*&=(x_1,...,x_{\ell},...,x_m,x_{m+1},x_{m+2},\bar{v}_1,...,\bar{v}_{\ell}^*,...,\bar{v}_m,v_{m+1}^*,v_{m+2}^*),\\
x_{m+1}&=x_{\ell}+\sqrt{2}\epsilon\omega_1,\quad x_{m+2}=x_{\ell}+\sqrt{2}\epsilon\omega_2,\\
\bar{Z}_{m+2}^*&=(\bar{x}_1,...,\bar{x}_{\ell},...,\bar{x}_m,\bar{x}_{\ell},\bar{x}_{\ell},\bar{v}_1,...,\bar{v}_{\ell}^*,...,\bar{v}_m,v_{m+1}^*,v_{m+2}^*),\\
(\bar{v}_{\ell}^*,v_{m+1}^*,v_{m+2}^*)&=T_{\omega_1,\omega_2}(\bar{v}_{\ell},v_{m+1},v_{m+2}).
\end{aligned}
\end{equation}
\end{enumerate}
\end{proposition}
\begin{proof} By symmetry, we may assume without loss of generality that $\ell=m$. For convenience, let us define the set of indices:
$$
\mathcal{F}_{m+2}=\left\{(i,j)\in\left\{1,...,m+2\right\} \times \left\{1,...,m+2\right\} : \; i<\min\left\{j,m\right\}\right\}.
$$

{\bf{Proof of \it{(i)} }} Here we use the notation from \eqref{pre-notation}. We start by formulating the following claim, which will imply \eqref{pre-0}.


\begin{lemma} \label{lemma-rec-mainprop-claim1}
Under the hypothesis of Proposition \ref{bad set}, there is a set $\mathcal{B}_{m}^{0,-}(\bar{Z}_m)\subseteq \mathbb{S}_1^{2d-1}\times B_R^{2d}$ such that for any $(\omega_1,\omega_2,v_{m+1},v_{m+2})\in \left(\mathbb{S}_1^{2d-1}\times B_R^{2d}\right)^+(\bar{v}_{m})\setminus\mathcal{B}_{m}^{0,-}(\bar{Z}_m)$, there holds:
\begin{align}
|x_i(t)-x_j(t)|>\sqrt{2}\epsilon,\quad\forall t\geq 0,\quad\forall (i,j)\in\mathcal{F}_{m+2},&\label{pre-0-1}\\
d^2\left(x_{m}\left(t\right);x_{m+1}\left(t\right),x_{m+2}\left(t\right)\right)>2\epsilon^2,\quad\forall t\geq 0.&\label{pre-0-2}
\end{align}
\end{lemma}

We observe that \eqref{pre-0-1}-\eqref{pre-0-2} imply  \eqref{pre-0}.

\steps{Proof of Lemma \ref{lemma-rec-mainprop-claim1} }\\

{\it{Step 1 - the proof of  \eqref{pre-0-1}}}:
Fix $(i,j)\in\mathcal{F}_{m+2}$. We distinguish the following cases:\\
$\bullet$ $j\leq m$: Since $\bar{Z}_m\in G_m(\epsilon_0,0)$ and $i<j\leq m$, we have
$|\bar{x}_i-\bar{x}_j-t(\bar{v}_i-\bar{v}_j)|>\epsilon_0$ for all $t\geq 0$.
Hence, triangle inequality implies
\begin{equation}\label{triang j<=s}
\begin{aligned}
|x_i(t)-x_j(t)|=|x_i-x_j-t(\bar{v}_i-\bar{v}_j)|
\geq |\bar{x}_i-\bar{x}_j-t(\bar{v}_i-\bar{v}_j)|-\alpha
\geq \epsilon_0-\alpha 
\geq\frac{\epsilon_0}{2} 
> \sqrt{2}\epsilon,
\end{aligned}
\end{equation} 
since $\epsilon<<\alpha<<\epsilon_0$. Therefore, \eqref{pre-0-1} holds for any $(\omega_1,\omega_2,v_{m+1},v_{m+2})\in \mathbb{S}_1^{2d-1}\times B_R^{2d}$.\\
$\bullet$ $j=m+1$: Since $(i,j) \in\mathcal{F}_{m+2}$ we have $i\leq m-1$. 
Then for $\bar{Z}_m\in G_m(\epsilon_0,0)$ and $X_m\in B_{\alpha/2}^{dm}(\bar{X}_m)$, we conclude
\begin{align*}
|\bar{x}_i-\bar{x}_{m}|>\epsilon_0,\quad |x_{m+1}-\bar{x}_{m}| \leq |x_{m}-\bar{x}_{m}|+|x_{m+1}-x_{m}|\leq\frac{\alpha}{2}+\sqrt{2}\epsilon|\omega_1|\leq\frac{\alpha}{2}+\sqrt{2}\epsilon<\alpha.
\end{align*}
Applying part \textit{(i)} of Lemma \ref{adjuction of 1} with $\bar{y}_1=\bar{x}_i$, $\bar{y}_2=\bar{x}_{m}$, $y_1=x_i$, $y_2=x_{m+1}$,
 we can find a cylinder $K_\eta^{d,i}$ such that for any $v_{m+1}\in B_R^d\setminus K_\eta^{d,i}$  we have:
$|x_i(t)-x_{m+1}(t)|>\sqrt{2}\epsilon,$ for all $t\geq 0$.
Hence \eqref{pre-0-1} holds for any $(\omega_1,\omega_2,v_{m+1},v_{m+2})\in (\mathbb{S}_1^{2d-1}\times B_R^{2d})\setminus U_{m+1}^i$, where
\begin{equation}\label{V_s+1-i}
 U_{m+1}^i= \mathbb{S}_1^{2d-1}\times K_\eta^{d,i}\times\mathbb{R}^d.
 \end{equation}
$\bullet$ $j=m+2:$ Since $(i,j)\in\mathcal{F}_{m+2}$, we obtain $i<m$. Hence, a similar argument to the previous case yields that \eqref{pre-0-1} holds for any $(\omega_1,\omega_2,v_{m+1},v_{m+2})\in (\mathbb{S}_1^{2d-1}\times B_R^{2d})\setminus U_{m+2}^i$, where
\begin{equation}\label{V_s+2-i}
 U_{m+2}^i= \mathbb{S}_1^{2d-1}\times\mathbb{R}^d\times K_\eta^{d,i}.
 \end{equation}
We conclude that \eqref{pre-0-1} holds for any 
$
(\omega_1,\omega_2,v_{m+1},v_{m+2})\in (\mathbb{S}_1^{2d-1}\times B_R^{2d})\setminus\displaystyle\bigcup_{i=1}^{m-1}(U_{m+1}^i\cup U_{m+2}^{i}).
$

{\it{Step 2 - the proof of  \eqref{pre-0-2}}:} Let us recall notation from \eqref{pre-post notation}. Fixing $t\geq 0$ and considering $(\omega_1,\omega_2,v_{m+1},v_{m+2})\in \left(\mathbb{S}_1^{2d-1}\times B_R^{2d}\right)^+(\bar{v}_{m})$, we have
\begin{align}
d^2\left(x_{m}\left(t\right);x_{m+1}\left(t\right),x_{m+2}\left(t\right)\right)&=|\sqrt{2}\epsilon\omega_1+t(v_{m+1}-\bar{v}_{m})|^2+|\sqrt{2}\epsilon\omega_2+t(v_{m+2}-\bar{v}_{m})|^2\nonumber\\
&\geq 2\epsilon^2(|\omega_1|^2+|\omega_2|^2)+2\sqrt{2}\epsilon t b(\omega_1,\omega_2,v_{m+1}-\bar{v}_{m},v_{m+2}-\bar{v}_{m})\nonumber\\
&> 2\epsilon^2\label{pre negative}, 
\end{align}
where to obtain \eqref{pre negative} we use the fact that $(\omega_1,\omega_2,v_{m+1},v_{m+2})\in(\mathbb{S}_1^{2d-1}\times B_R^{2d})^+(\bar{v}_m)$.

Defining  the set 
$\mathcal{B}_{m}^{0,-}(\bar{Z}_m)=\bigcup_{i=1}^{m-1}(U_{m+1}^i\cup U_{m+2}^{i}),$
Lemma \ref{lemma-rec-mainprop-claim1} is proved, and  \eqref{pre-0} follows.

Let us now find a set $\mathcal{B}_{m}^{\delta,-}(\bar{Z}_m)\subseteq\mathbb{S}_1^{2d-1}\times B_R^{2d}$ such that \eqref{pre-delta}  holds in the complement. We distinguish the following cases\\
$\bullet$ $(i,j)\in\mathcal{F}_{m+2}$, $j\leq m$: We use the same argument as in \eqref{triang j<=s} to obtain the lower bound $\epsilon_0/2$. \\
$\bullet$ $(i,j)\in\mathcal{F}_{m+2}$, $j\in\{m+1,m+2\}$: \eqref{pre-delta} holds for $(\omega_1,\omega_2,v_{m+1},v_{m+2})\in \left(\mathbb{S}_1^{2d-1}\times B_R^{2d}\right)\setminus\mathcal{B}_{m}^{0,-}(\bar{Z}_m)$, using part \textit{(ii)} of Lemma \ref{adjuction of 1} and similar arguments to the corresponding cases in the proof of Lemma  \ref{lemma-rec-mainprop-claim1}. Let us note that the lower bound is in fact $\epsilon_0$.\\
$\bullet$ $(i,j)=(m,m+1)$: Triangle inequality implies that for $t\geq\delta$ and $(\omega_1,\omega_2,v_{m+1},v_{m+2})\in\mathbb{S}_1^{2d-1}\times B_R^{2d}$, such that $|v_{m+1}-\bar{v}_{m}|>\eta$, 
we have
 \begin{align}
|x_{m}(t)-x_{m+1}(t)|&=|\sqrt{2}\epsilon\omega_1-t(\bar{v}_{m}-v_{m+1})|\geq |\bar{v}_{m}-v_{m+1}|t-\sqrt{2}\epsilon|\omega_1|  \nonumber \\
& \geq |\bar{v}_{m}-v_{m+1}|\delta-\sqrt{2}\epsilon> \eta\delta-\sqrt{2}\epsilon>\epsilon_0\label{eta bound and parameters},
\end{align}
where to obtain \eqref{eta bound and parameters} we use the fact that $\epsilon<<\epsilon_0<<\eta\delta$. Let us note that the lower bound is in fact $\epsilon_0$.
Therefore, \eqref{pre-delta} holds for $(\omega_1,\omega_2,v_{m+1},v_{m+2})\in(\mathbb{S}_1^{2d-1}\times B_R^{2d})\setminus V_{m,m+1}$, where
\begin{equation}\label{V_m_s s+1}
V_{m,m+1}=\mathbb{S}_1^{2d-1}\times B_\eta^d\left(\bar{v}_{m}\right)\times \mathbb{R}^d.
\end{equation}
$\bullet$ $(i,j)=(m,m+2)$: Same arguments as in the case $(i,j)=(m,m+1)$ yield that \eqref{pre-delta} holds for $(\omega_1,\omega_2,v_{m+1},v_{m+2})\in(\mathbb{S}_1^{2d-1}\times B_R^{2d})\setminus V_{m,m+2}$, where 
\begin{equation}\label{V_m_s s+2}
V_{m,m+2}=\mathbb{S}_1^{2d-1}\times \mathbb{R}^d\times B_\eta^d\left(\bar{v}_{m}\right).
\end{equation}
The lower bound is in fact $\epsilon_0$.\\
$\bullet$ $(i,j) = (m+1,m+2)$. Triangle inequality implies that for $t\geq\delta$ and $(\omega_1,\omega_2,v_{m+1},v_{m+2})\in\mathbb{S}_1^{2d-1}\times B_R^{2d}$, such that $|v_{m+1}-v_{m+2}|>\eta$, 
we have
\begin{align}
|x_{m+1}(t)-x_{m+2}(t)|&=|\sqrt{2}\epsilon(\omega_2-\omega_1)-t(v_{m+1}-v_{m+2})|\nonumber\\&\geq |v_{m+1}-v_{m+2}|t-\sqrt{2}\epsilon|\omega_1 - \omega_2|\nonumber\\
&\geq|v_{m+1}-v_{m+2}|\delta-\sqrt{2}\epsilon(|\omega_1|+|\omega_2|)\nonumber\\
&\geq |v_{m+1}-v_{m+2}|\delta-2\sqrt{2}\epsilon\nonumber\\
&>\eta\delta-2\sqrt{2}\epsilon>\epsilon_0,\label{triangle i=s+1, j=s+2}
\end{align}
where to obtain \eqref{triangle i=s+1, j=s+2}  we use the fact that $\epsilon<<\epsilon_0<<\eta\delta$.
Recalling from \eqref{strip general} the $2d$-strip
\begin{equation}\label{strip 1}
W_{\eta,1,1}^{2d}=\{(w_1,w_2)\in\mathbb{R}^{2d}:|w_1-w_2|\leq\eta\},
\end{equation}
we obtain that \eqref{pre-delta} holds for $(\omega_1,\omega_2,v_{m+1},v_{m+2})\in(\mathbb{S}_1^{2d-1}\times B_R^{2d})\setminus U_{m+1,m+2}$, where
\begin{equation}\label{V_s+1 s+2}
U_{m+1,m+2}=\mathbb{S}_1^{2d-1}\times W_{\eta,1,1}^{2d}.
\end{equation}
Notice that the lower bound is in fact $\epsilon_0$ again.

Defining
\begin{equation}\label{B_delta_-}
\mathcal{B}_{m}^{\delta,-}(\bar{Z}_m)=\mathcal{B}_{m}^{0,-}(\bar{Z}_m)\cup V_{m,m+1}\cup V_{m,m+2}\cup U_{m+1,m+2},
\end{equation}
we conclude that \eqref{pre-delta} holds  for 
$(\omega_1,\omega_2,v_{m+1},v_{m+2})\in(\mathbb{S}_1^{2d-1}\times B_R^{2d})\setminus \mathcal{B}_{m}^{\delta,-}(\bar{Z}_m).$

Let us note that the only case which prevents  $Z_{m+2}\in G_{m+2}(\epsilon_0,\delta)$ is the case $1\leq i<j\leq m$, where we obtain a lower bound of $\epsilon_0/2$. In all other cases we can obtain lower bound $\epsilon_0$. 

A similar argument shows that, for $(\omega_1,\omega_2,v_{m+1},v_{m+2})\in(\mathbb{S}_1^{2d-1}\times B_R^{2d})\setminus \mathcal{B}_{m}^{\delta,-}(\bar{Z}_m)$, \eqref{pre-delta-} holds for all $1\leq i<j\leq m+2$ except the case $1\leq i<j\leq m$. However in this case, for any $1\leq i<j\leq m$, we have
$
|\bar{x}_i(t)-\bar{x}_j(t)|=|\bar{x}_i-\bar{x}_j-t(\bar{v}_i-\bar{v}_j)|>\epsilon_0,
$
since $\bar{Z}_m\in G_m(\epsilon_0,0)$. This observation shows that  \eqref{pre-delta-} holds for 
$(\omega_1,\omega_2,v_{m+1},v_{m+2})\in(\mathbb{S}_1^{2d-1}\times B_R^{2d})\setminus \mathcal{B}_{m}^{\delta,-}(\bar{Z}_m),$
as well.

We conclude that the set
\begin{equation}\label{B- explicit}
\begin{aligned}
\mathcal{B}_{m}^-(\bar{Z}_m)=&(\mathbb{S}_1^{2d-1}\times B_R^{2d})^+(\bar{v}_m)\cap\left[ V_{m,m+1}\cup V_{m,m+2}\cup U_{m+1,m+2}\cup\bigcup_{i=1}^{m-1}(U_{m+1}^i\cup U_{m+2}^{i})\right],
\end{aligned}
\end{equation}
is the set we need for the pre-collisional case.

{\bf{Proof of \it{(ii)} }}  Here we use the notation from \eqref{post-notation}. 
The proof follows the steps of the pre-collisional case, but we replace the velocities $(\bar{v}_m,v_{m+1},v_{m+2})$ by the transformed velocities $(\bar{v}_{m}^*,v_{m+1}^*,v_{m+2}^*)$ and then pull-back. For details see \cite{thesis}. It is worth mentioning that the $m$-particle  needs special treatment since its velocity is transformed to $\bar{v}_{m}^*$. Following similar arguments to the precollisional case, we conclude that the appropriate set for the postcollisional case is given by
\begin{equation}\label{B^+ explicit}
\begin{aligned}
\mathcal{B}_{m}^+(\bar{Z}_m)=&(\mathbb{S}_1^{2d-1}\times B_R^{2d})^+(\bar{v}_m)\cap\left[V_{m,m+1}^*\cup V_{m,m+2}^*\cup U_{m+1,m+2}^*\cup \bigcup_{i=1}^{m-1}(V_{m}^{i,*}\cup U_{m+1}^{i,*}\cup U_{m+2}^{i,*})\right],
\end{aligned}
\end{equation}
where
\begin{equation}\label{post sets}
\begin{aligned}
V_{m}^{i,*}&=\left\{(\omega_1,\omega_2,v_{m+1},v_{m+2})\in\mathbb{S}_1^{2d-1}\times\mathbb{R}^{2d} :\bar{v}_{m}^*\in K_\eta^{d,i}\right\},\\
U_{m+1}^{i,*}&= \left\{(\omega_1,\omega_2,v_{m+1},v_{m+2})\in\mathbb{S}_1^{2d-1}\times\mathbb{R}^{2d}: v_{m+1}^*\in K_\eta^{d,i}\right\},\\
U_{m+2}^{i,*}&= \left\{(\omega_1,\omega_2,v_{m+1},v_{m+2})\in\mathbb{S}_1^{2d-1}\times\mathbb{R}^{2d}: v_{m+2}^*\in K_\eta^{d,i}\right\},\\
V_{m,m+1}^*&=\left\{(\omega_1,\omega_2,v_{m+1},v_{m+2})\in\mathbb{S}_1^{2d-1}\times\mathbb{R}^{2d}: (v_{m}^*,v_{m+1}^*)\in W_{\eta,1,1}^{2d}\right\},\\
V_{m,m+2}^*&=\left\{(\omega_1,\omega_2,v_{m+1},v_{m+2})\in\mathbb{S}_1^{2d-1}\times\mathbb{R}^{2d}: (v_{m}^*,v_{m+2}^*)\in W_{\eta,1,1}^{2d}\right\},\\
U_{m+1,m+2}^*
&=\left\{(\omega_1,\omega_2,v_{m+1},v_{m+2})\in\mathbb{S}_1^{2d-1}\times\mathbb{R}^{2d}: (v_{m+1}^*,v_{m+2}^*)\in W_{\eta,1,1}^{2d}\right\}.
\end{aligned}
\end{equation}

Therefore, the set we need is 
\begin{equation}\label{B}
\mathcal{B}_{m}(\bar{Z}_m)=\mathcal{B}_{m}^{-}(\bar{Z}_m)\cup\mathcal{B}_{m}^{+}(\bar{Z}_m).
\end{equation} 
\end{proof}

We now use the results of Section \ref{sec_geometric} to estimate the measure of this set, up to the parameters chosen.
\begin{proposition}\label{measure of B} Consider parameters  $\alpha,\epsilon_0,R,\eta,\delta$  as in \eqref{choice of parameters} and $\epsilon<<\alpha$.  Let $m\in\mathbb{N}$, $\bar{Z}_m\in G_m(\epsilon_0,0)$, $\ell\in\{1,...,m\}$ and $\mathcal{B}_{\ell}(\bar{Z}_m)$ the set given in the statement of Proposition \ref{bad set}. Denoting by $|\cdot|$  the product measure on $\mathbb{S}_1^{2d-1}\times B_R^{2d}$, the following estimate holds:
\begin{equation*}
\left|\mathcal{B}_{\ell}(\bar{Z}_m)\right|\lesssim mR^{2d}\eta^{\frac{d-1}{4d+2}}.
\end{equation*}
\end{proposition}
\begin{proof} Without loss of generality, we may assume that $\ell=m$.

{\bf{Estimate of $\mathcal{B}_{m}^-(\bar{Z}_m)$}.} We recall \eqref{B- explicit}.\\
$\bullet$ Estimate of the terms corresponding to $V_{m,m+1}$, $V_{m,m+2}$, $U_{m+1,m+2}$:
Recalling \eqref{V_m_s s+1} , we have
$ 
V_{m,m+1}=\mathbb{S}_1^{2d-1}\times B_\eta^d\left(\bar{v}_{m}\right)\times \mathbb{R}^d.
$ 
We have
$(\mathbb{S}_1^{2d-1}\times B_R^{2d})^+(\bar{v}_m)\cap V_{m,m+1}\subseteq \mathbb{S}_1^{2d-1}\times \left(B_R^d\cap B_\eta^d\left(\bar{v}_m\right)\right)\times B_R^d,$
so
\begin{equation}\label{estimate V_s,s+1}
\begin{aligned}
|(\mathbb{S}_1^{2d-1}\times B_R^{2d})^+(\bar{v}_m)\cap V_{m,m+1}| &\leq |\mathbb{S}_1^{2d-1}|_{\mathbb{S}_1^{2d-1}} |B_R^d\cap B_\eta^d(\bar{v}_m)|_d |B_R^d|_d\lesssim R^d\eta^d.
\end{aligned}
\end{equation}
 
In a similar way, we obtain 
\begin{equation}\label{estimate V_s,s+2}
|(\mathbb{S}_1^{2d-1}\times B_R^{2d})^+(\bar{v}_m)\cap V_{m,m+2}|\lesssim R^d\eta^d.
\end{equation}  
Recalling \eqref{V_s+1 s+2}, we have 
$U_{m+1,m+2}=\mathbb{S}_1^{2d-1}\times W_{\eta,1,1}^{2d},$
thus
$(\mathbb{S}_1^{2d-1}\times B_R^{2d})^+(\bar{v}_m)\cap U_{m+1,m+2}\subseteq \mathbb{S}_1^{2d-1}\times\left[\left(B_R^d\times B_R^d\right)\cap W_{\eta,1,1}^{2d}\right],$
hence
\begin{equation}\label{estimate V_s+1,s+2}
\begin{aligned}
|(\mathbb{S}_1^{2d-1}\times B_R^{2d})^+(\bar{v}_m)\cap U_{m+1,m+2}|&\leq |\mathbb{S}_1^{2d-1}|_{\mathbb{S}_1^{2d-1}}|(B_R^{d}\times B_R^d)\cap W_{\eta,1,1}^{2d}|_{2d}\\
&\lesssim\int_{B_R^d}\int_{B_R^d}\mathds{1}_{B_\eta^d(v_{m+1})}(v_{m+2})\,dv_{m+2}\,dv_{m+1}\\
&\lesssim R^d\eta^d.
\end{aligned}
\end{equation}
$\bullet$ Estimate of the terms corresponding to $U_{m+1}^i$, $U_{m+2}^i$ , $i\in\{1,...,m-1\}$: Fix $i\in\{1,...,m-1\}$.
Recalling the set $U_{m+1}^i=\mathbb{S}_1^{2d-1}\times K_\eta^{d,i}\times\mathbb{R}^d,$ from \eqref{V_s+1-i},
we have
$(\mathbb{S}_1^{2d-1}\times B_R^{2d})^+(\bar{v}_m)\cap U_{m+1}^i\subseteq\mathbb{S}_1^{2d-1}\times \left[B_R^{2d}\cap \left(K_\eta^{d,i}\times\mathbb{R}^d\right)\right].$
Since $\eta<<1<<R$,  Proposition \ref{spherical estima}  implies that 
\begin{equation}\label{estimate V_s+1^i}
|(\mathbb{S}_1^{2d-1}\times B_R^{2d})^+\cap U_{m+1}^i|\leq|\mathbb{S}_1^{2d-1}|_{\mathbb{S}_1^{2d-1}}|B_R^{2d}\cap \left(K_\eta^{d,i}\times\mathbb{R}^d\right)|_{2d}\lesssim |\left(B_R^{d}\cap K_\eta^{d,i}\right)\times B_R^{d}|_{2d} \lesssim R^{2d}\eta^{\frac{d-1}{2}}.
\end{equation}
In a similar way, we obtain
\begin{equation}\label{estimate V_s+2^i}
|(\mathbb{S}_1^{2d-1}\times B_R^{2d})^+\cap U_{m+2}^i|\lesssim R^{2d}\eta^{\frac{d-1}{2}}.
\end{equation}
Therefore, recalling \eqref{B- explicit}, using estimates \eqref{estimate V_s,s+1}-\eqref{estimate V_s+2^i} and the facts that $s\geq 1$, $\eta<<1<<R$, sub-additivity implies
\begin{equation}\label{measure B^-}
|\mathcal{B}_{m}^-(\bar{Z}_m)|\lesssim mR^{2d}\eta^{\frac{d-1}{2}}< mR^{2d}\eta^{\frac{d-1}{4d+2}},\quad\text{since }\eta<<1.
\end{equation}

{\bf{Estimate of $\mathcal{B}_m^+(\bar{Z}_m)$:}} We recall \eqref{B^+ explicit}.
To estimate the measure of $\mathcal{B}_{m}^+(\bar{Z}_m)$, we will strongly rely on the properties of the transition map defined in Proposition \ref{transition lemma}. 

Let us define $\Phi_{\bar{v}_m}:\mathbb{R}^{2d}\to\mathbb{R}$  by
$
\Phi_{\bar{v}_m}(v_{m+1},v_{m+2})=|v_{m+1}-\bar{v}_m|^2+|v_{m+2}-\bar{v}_m|^2+|v_{m+1}-v_{m+2}|^2.
$
We can easily see that given $r>0$ and $(v_{m+1},v_{m+2})\in\Phi^{-1}_{\bar{v}_m}(\{r^2\})$, we have
\begin{equation}\label{nabla of level sets}
2r\leq |\nabla\Phi_{\bar{v}_m}(v_{m+1},v_{m+2})|\leq 4r.
\end{equation}
Let also define the set 
$G_R^{2d}(\bar{v}_m):=[0\leq \Phi_{\bar{v}_m}\leq 16R^2].$
Notice that by triangle inequality and the fact that $\bar{v}_m\in B_R^{d}$, we have 
\begin{equation}\label{bound on ball}
B_R^{2d}\subseteq G_R^{2d}(\bar{v}_m).
\end{equation}
Recall from \eqref{set of post angles} the set $\mathcal{S}_{\bar{v}_{m},v_{m+1},v_{m+2}}^+$.
Then, Fubini's Theorem and the co-area formula yield 
\begin{align}
&|\mathcal{B}_{m}^+(\bar{Z}_m)|=
\int_{(\mathbb{S}_1^{2d-1}\times B_R^{2d})^+(\bar{v}_m)}\mathds{1}_{\mathcal{B}_{m}^+(\bar{Z}_m)}\,d\omega_1\,d\omega_2\,dv_{m+1}\,dv_{m+2}\nonumber\\
&\leq\int_{G_R^{2d}(\bar{v}_m)}\int_{\mathcal{S}_{\bar{v}_m,v_{m+1},v_{m+2}}^+}\mathds{1}_{\mathcal{B}_{m}^+(\bar{Z}_m)}\,d\omega_1\,d\omega_2\,dv_{m+1}\,dv_{m+2}\label{bound of ball use}\\
&=\int_0^{16R^2}\int_{\Phi^{-1}_{\bar{v}_m}(\{s\})}|\nabla\Phi_{\bar{v}_m}(v_{m+1},v_{m+2})|^{-1}\int_{\mathcal{S}_{\bar{v}_m,v_{m+1},v_{m+2}}^+}\mathds{1}_{\mathcal{B}_{m}^+(\bar{Z}_m)}\,d\omega_1\,d\omega_2\,dv_{m+1}\,dv_{m+2}\,ds\nonumber\\
&=\int_0^{4R}2r\int_{\Phi^{-1}_{\bar{v}_m}(\{r^2\})}|\nabla\Phi_{\bar{v}_m}(v_{m+1},v_{m+2})|^{-1}\int_{\mathcal{S}_{\bar{v}_m,v_{m+1},v_{m+2}}^+}\mathds{1}_{\mathcal{B}_{m}^+(\bar{Z}_m)}\,d\omega_1\,d\omega_2\,dv_{m+1}\,dv_{m+2}\,dr\nonumber\\
&\lesssim\int_0^{4R}\int_{\Phi^{-1}_{\bar{v}_m}(\{r^2\})}\int_{\mathcal{S}_{\bar{v}_m,v_{m+1},v_{m+2}}^+}\mathds{1}_{\mathcal{B}_{m}^+(\bar{Z}_m)}(\omega_1,\omega_2)\,d\omega_1\,d\omega_2\,dv_{m+1}\,dv_{m+2}\,dr,\label{integral expression for post}
\end{align}
where to obtain \eqref{bound of ball use}, we use \eqref{bound on ball}, and to obtain \eqref{integral expression for post} we use the lower bound of \eqref{nabla of level sets}.

We estimate the integral
$\displaystyle\int_{\mathcal{S}_{\bar{v}_{m},v_{m+1},v_{m+2}}^+}\hspace{-0.7cm}\mathds{1}_{\mathcal{B}_{m}^+(\bar{Z}_m)}(\omega_1,\omega_2)\,d\omega_1\,d\omega_2,$
for fixed $0<r\leq 4R$ and $(v_{m+1},v_{m+2})\in \Phi^{-1}_{\bar{v}_m}(\{r^2\})$. Let us introduce a  parameter $0<\beta<1$, which will be chosen later in terms of $\eta$.
Writing
\begin{equation}\label{notation for M}
\bm{\omega}=
(\omega_1,\omega_2),\quad
\bm{v}=
(v_{m+1}-\bar{v}_{m},
v_{m+2}-\bar{v}_{m}),
\end{equation}
we have
$
b(\bm{\omega},\bm{v})=\langle\bm{\omega},\bm{v}\rangle.
$
Inspired in part by \cite{denlinger} (Proposition $1$), we decompose
$$
\mathcal{S}^+_{\bar{v}_{m},v_{m+1},v_{m+2}}=\mathcal{S}_{\bar{v}_{m},v_{m+1},v_{m+2}}^{1,+}\cup \mathcal{S}_{\bar{v}_{m},v_{m+1},v_{m+2}}^{2,+},
$$
where
\begin{align}
\mathcal{S}_{\bar{v}_{m},v_{m+1},v_{m+2}}^{1,+}&=\left\{\bm{\omega}\in\mathcal{S}_{\bar{v}_{m},v_{m+1},v_{m+2}}^+:\langle\bm{\omega},\bm{v}\rangle>\beta |\bm{v}|\right\},\label{S1}\\
\mathcal{S}_{\bar{v}_{m},v_{m+1},v_{m+2}}^{2,+}&=\left\{\bm{\omega}\in\mathcal{S}_{\bar{v}_{m},v_{m+1},v_{m+2}}^+:0<\langle\bm{\omega},\bm{v}\rangle\leq\beta |\bm{v}|\right\}.\label{S2}
\end{align}
Notice that $\mathcal{S}_{\bar{v}_{m},v_{m+1},v_{m+2}}^{2,+}$ is the union of  two unit $(2d-1)$-spherical caps of angle $\pi/2-\arccos\beta$. Thus integrating in spherical coordinates, we have
\begin{equation}\label{estimate on S2}\int_{\mathcal{S}_{\bar{v}_{m},v_{m+1},v_{m+2}}^{2,+}}\mathds{1}_{\mathcal{B}_{m}^+(\bar{Z}_m)}(\omega_1,\omega_2)\,d\omega_1\,d\omega_2\lesssim \frac{\pi}{2}-\arccos\beta=\arcsin\beta.
\end{equation}
Let us estimate the terms corresponding to  $\mathcal{S}_{\bar{v}_{m},v_{m+1},v_{m+2}}^{1,+}$.  Our purpose is to change variables under the transition map $\mathcal{J}_{\bar{v}_m,v_{m+1},v_{m+2}}$, and use part \textit{(iv)} of Proposition \ref{transition lemma}.

Notice that for $\bm{\omega}\in\mathcal{S}_{\bar{v}_{m},v_{m+1},v_{m+2}}^{1,+}$, the lower estimate of \eqref{jacobian} and \eqref{S1} imply
\begin{equation}\label{estimate on inverse jacobian}
\jac^{-1}(\mathcal{J}_{\bar{v}_{m},v_{m+1},v_{m+2}})(\bm{\omega})\lesssim r^{2d}b^{-2d}(\bm{\omega},\bm{v})\leq r^{2d}\beta^{-2d} |\bm{v}|^{-2d}\lesssim\beta^{-2d},
\end{equation}
since by triangle inequality and Young's inequality, we have $$r^2=|\bar{v}_{m}-v_{m+1}|^2+|\bar{v}_{m}-v_{m+2}|^2+|v_{m+1}-v_{m+2}|^2\leq 3(|\bar{v}_{m}-v_{m+1}|^2+|\bar{v}_{m}-v_{m+2}|^2)= 3 |\bm{v}|^2.$$
$\bullet$ Estimate of $V_{m,m+1}^*$, $V_{m,m+2}^*$, $U_{m+1,m+2}^*$ terms:
By recalling  \eqref{post sets} 
$$V_{m,m+1}^*=\left\{(\omega_1,\omega_2,v_{m+1},v_{m+2})\in\mathbb{S}_1^{2d-1}\times B_R^{2d}:\bar{v}_{m}^*-v_{m+1}^*\in B_\eta^d\right\},$$
and \eqref{definition of v}, given $\bm{\omega}=(\omega_1,\omega_2)\in \mathcal{S}_{\bar{v}_{m},v_{m+1},v_{m+2}}^{1,+}$, we have 
 \begin{equation}\label{V_s,s+1 iff}\bar{v}_{m}^*-v_{m+1}^*\in B_\eta^d\Leftrightarrow \bm{\nu}=(\nu_1,\nu_2)\in B_{\eta/r}^d\times\mathbb{R}^d.\end{equation}
 Therefore,  we obtain
\begin{align}
&\int_{\mathcal{S}_{\bar{v}_m,v_{m+1},v_{m+2}}^{1,+}}\mathds{1}_{V_{m,m+1}^*}(\bm{\omega})\,d\bm{\omega}=\int_{\mathcal{S}_{\bar{v}_m,v_{m+1},v_{m+2}}^{1,+}}(\mathds{1}_{B_{\eta/r}^d\times\mathbb{R}^d}\circ\mathcal{J}_{\bar{v}_m,v_{m+1},v_{m+2}})(\bm{\omega})\,d\bm{\omega}\nonumber\\
&\lesssim\beta^{-2d}\int_{\mathcal{S}_{\bar{v}_m,v_{m+1},v_{m+2}}^{1,+}}(\mathds{1}_{B_{\eta/r}^d\times\mathbb{R}^d}\circ\mathcal{J}_{\bar{v}_m,v_{m+1},v_{m+2}})(\bm{\omega})\jac\mathcal{J}_{\bm{v}_m,v_{m+1},v_{m+2}}(\bm{\omega})\,d\bm{\omega}\label{use of bounds}\\
&\lesssim\beta^{-2d}\int_{\mathbb{E}_1^{2d-1}}\mathds{1}_{B_{\eta/r}^d\times\mathbb{R}^{d}}(\bm{\nu})\,d\bm{\nu}\lesssim \beta^{-2d}\min\left\{1,\left(\frac{\eta}{r}\right)^{\frac{d-1}{2}}\right\},\label{ball estimate}
\end{align}
where to obtain \eqref{use of bounds} we use \eqref{estimate on inverse jacobian},  to obtain \eqref{ball estimate} we use part $\textit{(iv)}$ of Proposition \ref{transition lemma} and  part \textit{(iii)} of Proposition \ref{ellipsoid estimates}. 
Thus
\begin{equation}\label{estimate on S1-Vs_s+1}
\int_{\mathcal{S}_{\bar{v}_{m},v_{m+1},v_{m+2}}^{1,+}}\mathds{1}_{V_{m,m+1}^*}(\omega_1,\omega_2)\,d\omega_1\,\,d\omega_2\lesssim \beta^{-2d}\min\left\{1,\left(\frac{\eta}{r}\right)^{\frac{d-1}{2}}\right\}.
\end{equation}
In a similar manner, recalling from \eqref{post sets} the sets $V_{m,m+2}^*, U_{m+1,m+2}^*$ respectively,
and parts \textit{(iv)}, \textit{(v)} of Proposition \ref{ellipsoid estimates} respectively, we  obtain the corresponding estimates.\\
$\bullet$ Estimate of $V_m^{i,*}$, $U_{m+1}^{i,*}$, $U_{m+2}^{i,*}$, $i\in\{1,...,m-1\}$ terms: Consider $i\in\{1,...,m-1\}$. By recalling \eqref{post sets}, the set $V_{m}^{i,*}$ can be equivalently written as
\begin{align*}
V_{m}^{i,*}=\{(\omega_1,\omega_2,v_{m+1},v_{m+2})\in\mathbb{S}_1^{2d-1}\times B_R^{2d}: (\bar{v}_m^*,v_{m+1}^*)\in K_\eta^{d,i}\times\mathbb{R}^d\}.
\end{align*}
 Recalling also the operator $\bm{S_{12}}$ defined in \eqref{matrix}, Lemma \ref{matrix lemma} implies
\begin{equation}\label{V_s^* iff}
(\bar{v}_m^*,v_{m+1}^*)\in K_\eta^{d,i}\times\mathbb{R}^d\Leftrightarrow \left(\bm{S_{12}}\circ\mathcal{J}_{\bar{v}_m,v_{m+1},v_{m+2}}\right)(\omega_1,\omega_2)\in \bar{K}_{\eta/r}^{d,i}\times\mathbb{R}^d,
\end{equation}
where $K_{\eta}^{d,i}$ is a $d$-cylinder of radius $\eta$  and $\bar{K}_{\eta/r}^{d,i}$ is a $d$-cylinder of radius $\eta/r$.
Recalling $\mathcal{S}=\bm{S_{12}}(\mathbb{E}_1^{2d-1})$ from \eqref{S-ellipsoid}, 
and using the same reasoning to change variables under $\mathcal{J}_{\bar{v}_m,v_{m+1},v_{m+2}}$ as in the estimate for $V_{m,m+1}^*$, we have
\begin{align}
&\int_{\mathcal{S}_{\bar{v}_{m},v_{m+1},v_{m+2}}^{1,+}}\mathds{1}_{V_{m}^{i,*}}(\omega_1,\omega_2)\,d\omega_1\,d\omega_2=\int_{\mathcal{S}_{\bar{v}_{m},v_{m+1},v_{m+2}}^{1,+}}\mathds{1}_{(\bar{v}_m^*,v_{m+1}^*)\in K_{\eta}^{d,i}\times\mathbb{R}^d}(\omega_1,\omega_2)\,d\omega_1\,d\omega_2\nonumber\\
&=\int_{\mathcal{S}_{\bar{v}_{m},v_{m+1},v_{m+2}}^{1,+}}(\mathds{1}_{\bar{K}_{\eta/r}^{d,i}\times\mathbb{R}^d}\circ\bm{S_{12}}\circ \mathcal{J}_{\bar{v}_m,v_{m+1},v_{m+2}})(\omega_1,\omega_2)\,d\omega_1\,d\omega_2\label{V_s^* iff inside}\\
&\lesssim\beta^{-2d}\int_{\mathbb{E}_1^{2d-1}}(\mathds{1}_{\bar{K}_{\eta/r}^{d,i}\times\mathbb{R}^{d}}\circ\bm{S_{12}})(\nu_1,\nu_2)\,d\nu_1\,d\nu_2\label{jac and change V_s^*}\\
&\lesssim\beta^{-2d}\int_{\mathcal{S}}\mathds{1}_{\bar{K}_{\eta/r}^{d,i}\times\mathbb{R}^{d}}(\theta_1,\theta_2)\,d\theta_1\,d\theta_2\label{estimate on S1-Vs^*}\\
&\lesssim \beta^{-2d}\min\left\{1,\left(\frac{\eta}{r}\right)^{\frac{d-1}{2}}\right\},\label{use of Lemma estimate review}
\end{align}
where to obtain \eqref{V_s^* iff inside} we use \eqref{V_s^* iff}, to obtain \eqref{jac and change V_s^*} we use  estimate \eqref{estimate on inverse jacobian} and part \textit{(iv)} of Proposition \ref{transition lemma}, to obtain \eqref{estimate on S1-Vs^*} we make the linear transformation $(\theta_1,\theta_2)=\bm{S_{12}}(\nu_1,\nu_2)$ and use the fact that $\mathcal{S}=\bm{S_{12}}(\mathbb{E}_1^{2d-1})$, and  to obtain \eqref{use of Lemma estimate review} we use part \textit{(i)} of Proposition \ref{ellipsoid estimates}. 

Recalling $U_{m+1}^{i,*}$, $U_{m+2}^{i,*}$ from \eqref{post sets}, and using respectively the map $\bm{S_{12}}$ from Lemma \ref{matrix lemma} and estimate \textit{(ii)} from Proposition \ref{ellipsoid estimates}, the map $\bm{S_{13}}$ from Lemma \ref{matrix lemma} and estimate \textit{(ii)} from Proposition \ref{ellipsoid estimates}, we obtain the corresponding estimates in a similar way.

We conclude that
\begin{equation}\label{on S_1}
\int_{\mathcal{S}_{\bar{v}_{m},v_{m+1},v_{m+2}}^{1,+}}\mathds{1}_{\mathcal{B}_{m}^+(\bar{Z}_m)}(\omega_1,\omega_2)\,d\omega_1\,d\omega_2\lesssim m\beta^{-2d} \min\left\{1,\left(\frac{\eta}{r}\right)^{\frac{d-1}{2}}\right\}
\end{equation}

Therefore, recalling  $\mathcal{S}_{\bar{v}_{m},v_{m+1},v_{m+2}}^+=\mathcal{S}_{\bar{v}_{m},v_{m+1},v_{m+2}}^{1+}\cup\mathcal{S}_{\bar{v}_{m},v_{m+1},v_{m+2}}^{2+}$, and using estimates \eqref{estimate on S2}, \eqref{on S_1}, we obtain the estimate:
\begin{equation}\label{estimate on post integral}
\int_{\mathcal{S}_{\bar{v}_{m},v_{m+1},v_{m+2}}^+}\mathds{1}_{\mathcal{B}_{m}^+(\bar{Z}_m)}(\omega_1,\omega_2)\,d\omega_1\,d\omega_2\lesssim \arcsin\beta+m\beta^{-2d}\min\left\{1,\left(\frac{\eta}{r}\right)^{\frac{d-1}{2}}\right\}.
\end{equation}
Hence, \eqref{integral expression for post} yields
\begin{equation}\label{post collisional estimate with beta}
\begin{aligned}
|\mathcal{B}_{m}^{+}(\bar{Z}_m)|&\lesssim \int_0^{4R}\int_{\Phi_{\bar{v}_m}^{-1}(\{r^2\})}\left(\arcsin\beta+m\beta^{-2d}\min\left\{1,\left(\frac{\eta}{r}\right)^{\frac{d-1}{2}}\right\}\right)\,dv_{m+1}\,dv_{m+2}\,dr\\
&\lesssim \int_0^{4R}r^{2d-1}\left(\arcsin\beta+m\beta^{-2d}\min\left\{1,\left(\frac{\eta}{r}\right)^{\frac{d-1}{2}}\right\}\right)\,dr\\
&\lesssim mR^{2d}\left(\arcsin\beta+\beta^{-2d}\eta^{\frac{d-1}{2}}\right)\\
&\lesssim mR^{2d}\left(\beta+\beta^{-2d}\eta^{\frac{d-1}{2}}\right),
\end{aligned} 
\end{equation}
after using an estimate similar to \eqref{estimate with min} and the fact that $m\geq 1$, $\beta<<1$. Choosing $\beta=\eta^{\frac{d-1}{4d+2}}<<1$,  since $d\geq 2$, we obtain
\begin{equation}\label{measure B+}
|\mathcal{B}_{m}^{+}(\bar{Z}_m)|\lesssim m R^{2d}\eta^{\frac{d-1}{4d+2}}.
\end{equation}

Combining \eqref{B}, \eqref{measure B^-}, \eqref{measure B+}, we obtain the required estimate.
\end{proof}

\section{Elimination of recollisions}\label{sec_elimination}
In this section we reduce the convergence proof to comparing truncated elementary observables. We first restrict to good configurations and then inductively  reduce the convergence proof to truncated elementary observables, which will be comparable in the scaled limit.
\subsection{Restriction to good configurations}
 Throughout this subsection,  we consider $\beta_0>0$, $\mu_0\in\mathbb{R}$,  $T>0$ given in Theorems \ref{well posedness BBGKY}, \ref{lwp Boltzmann hier}, the functions $\bm{\beta},\bm{\mu}:[0,T]\to\mathbb{R}$ defined by \eqref{beta mu}, $(N,\epsilon)$ in the scaling \eqref{scaling} and initial data $F_{N,0}\in X_{N,\beta_0,\mu_0}$, $F_0\in X_{\infty,\beta_0,\mu_0}$. Let $\bm{F_N}\in\bm{X}_{N,\bm{\beta},\bm{\mu}}$, $\bm{F}\in\bm{X}_{\infty,\bm{\beta},\bm{\mu}}$ be the mild solutions of the corresponding BBGKY  and Boltzmann hierarchies in $[0,T]$, given by Theorem \ref{well posedness BBGKY} and Theorem \ref{lwp Boltzmann hier} respectively.

For the convenience of a reader we recall the notation from Section \ref{sec_stability}. Specifically, given $m\in\mathbb{N}$, $\sigma>0$ and $t_0>0$, we denote
\begin{align*}
\Delta_m^X(\sigma)&=\{\widetilde{X}_m\in\mathbb{R}^{dm}: |\widetilde{x}_i-\widetilde{x}_j|>\sigma,\quad\forall 1\leq i<j\leq m\},\quad m\geq 2,\quad \Delta_1(\sigma)=\mathbb{R}^{d},\\
\Delta_m(\sigma)&=\Delta_m^X(\sigma)\times\mathbb{R}^{dm},\\
G_m(\sigma,t_0)&=\left\{Z_m=(X_m,V_m)\in\mathbb{R}^{2dm}: Z_m(t)\in\Delta_m(\sigma),\quad\forall t\geq t_0\right\},
\end{align*}
where  $Z_m(t)$ denotes the backwards free flow, given by:
$
Z_m(t)=(X_m-tV_m,V_m),
$
for $t\geq 0$. 
Given $\epsilon,\epsilon_0>0$ with $\epsilon<<\epsilon_0$ and $\delta>0$, we define the new set
\begin{equation}\label{both epsilon-epsilon_0}
G_m(\epsilon,\epsilon_0,\delta):=G_m(\epsilon,0)\cap G_m(\epsilon_0,\delta).
\end{equation}
Inductively using Lemma \ref{adjuction of 1} and Proposition \ref{spherical estima}, we obtain the following result. For more details on the proof see \cite{thesis}.
\begin{proposition}\label{initially good configurations}
Let $s\in\mathbb{N}$, $\alpha,\epsilon_0,R,\eta,\delta$ be parameters as in \eqref{choice of parameters} and $\epsilon<<\alpha$. Then for any $X_s\in\Delta_s^X(\epsilon_0)$, there is a subset of velocities $\mathcal{M}_s(X_s)\subseteq B_R^{ds}$ of measure
\begin{equation}\label{measure of initialization}
\left|\mathcal{M}_s\left(X_s\right)\right|_{ds}\leq C_{d,s} R^{ds}\eta^{\frac{d-1}{2}},
\end{equation}
such that
$
Z_s=(X_s,V_s)\in G_s(\epsilon,\epsilon_0,\delta),
$
for all $V_s\in B_R^{ds}\setminus \mathcal{M}_s(X_s).$
\end{proposition}

Consider $s,n\in\mathbb{N}$, parameters $\alpha,\epsilon_0,R,\eta,\delta$  as in \eqref{choice of parameters}, $(N,\epsilon)$  in the scaling \eqref{scaling} with $\epsilon<<\alpha$, $0\leq k\leq n$ and $t\in[0,T]$. Let us recall the observables $I_{s,k,R,\delta}^N(t)$, $I_{s,k,R,\delta}^\infty(t)$ defined in \eqref{bbgky truncated time}-\eqref{boltzmann truncated time}. We restrict the domain of integration to velocities giving good configurations. For convenience, given $X_s\in\Delta_s^X(\epsilon_0)$, we write $\mathcal{M}_s^c(X_s)=B_R^{ds}\setminus\mathcal{M}_s(X_s).$
We define
\begin{align}
\widetilde{I}_{s,k,R,\delta}^N(t)(X_s)&=\int_{\mathcal{M}_s^c(X_s)}\phi_s(V_s)f_{N,R,\delta}^{(s,k)}(t,X_s,V_s)\,dV_s\label{good observables BBGKY },\\
\widetilde{I}_{s,k,R,\delta}^\infty(t)(X_s)&=\int_{\mathcal{M}_s^c(X_s)}\phi_s(V_s)f_{R,\delta}^{(s,k)}(t,X_s,V_s)\,dV_s\label{good observables Boltzmann}.
\end{align}

We now apply Proposition \ref{initially good configurations} and the a-priori estimates of Section \ref{sec_local} to restrict to initially good configurations.
\begin{proposition}\label{restriction to initially good conf} Let $s,n\in\mathbb{N}$,   $\alpha,\epsilon_0,R,\eta,\delta$ be parameters as in \eqref{choice of parameters},  $(N,\epsilon)$  in the scaling \eqref{scaling} with $\epsilon<<\alpha$, and $t\in[0,T]$. Then, the following estimates hold:
\begin{equation*}
\sum_{k=0}^n\|I_{s,k,R,\delta}^N(t)-\widetilde{I}_{s,k,R,\delta}^N(t)\|_{L^\infty\left(\Delta_s^X\left(\epsilon_0\right)\right)}
\leq C_{d,s,\mu_0,T}R^{ds}\eta^{\frac{d-1}{2}}\|F_{N,0}\|_{N,\beta_0,\mu_0},
\end{equation*}
\begin{equation*}
\sum_{k=0}^n\|I_{s,k,R,\delta}^\infty (t)-\widetilde{I}_{s,k,R,\delta}^\infty (t)\|_{L^\infty\left(\Delta_s^X\left(\epsilon_0\right)\right)}\leq C_{d,s,\mu_0,T}R^{ds}\eta^{\frac{d-1}{2}}\|F_{0}\|_{\infty,\beta_0,\mu_0}.
\end{equation*}
\end{proposition}
\begin{remark}\label{no need for k=0}
Under the assumptions of Proposition \ref{restriction to initially good conf}, given $X_s\in\Delta_s^X(\epsilon_0)$, the definition of $\mathcal{M}_s(X_s)$ implies that $\widetilde{I}_{s,0,R,\delta}^N(t)(X_s)=\widetilde{I}_{s,0,R,\delta}^\infty(t)(X_s)$ for all $ t\in[0,T].$
Therefore, Proposition \ref{restriction to initially good conf} reduces the convergence to controlling the differences $\widetilde{I}_{s,k,R,\delta}^N(t)-\widetilde{I}_{s,k,R,\delta}^\infty(t),$ for $k=1,...,n$, in the scaled limit.
\end{remark}
\subsection{Reduction to elementary observables}\label{reduction to elementary observables}
Here, given $s,n\in\mathbb{N}$, parameters $\alpha,\epsilon_0,R,\eta,\delta$  as in \eqref{choice of parameters}  $1\leq k\leq n$,  $(N,\epsilon)$ in the scaling \eqref{scaling} with $\epsilon<<\alpha$, and $t\in[0,T]$, inspired by notation used in \cite{lanford,gallagher}, we expand $\widetilde{I}_{s,k,R,\delta}^N(t)$ and $\widetilde{I}_{s,k,R,\delta}^\infty(t)$, defined in \eqref{good observables BBGKY }-\eqref{good observables Boltzmann}, in terms of elementary observables.

 For this purpose, given $\ell,N\in\mathbb{N}$ with $\ell<N$, $R>1$,  we decompose the truncated BBGKY hierarchy collisional operator (given in \eqref{BBGKY operator}-\eqref{A}) in the following way:
\begin{equation*}
\mathcal{C}_{\ell,\ell+2}^{N,R}=\sum_{i=1}^\ell\mathcal{C}_{\ell,\ell+2}^{N,R,+,i}-\sum_{i=1}^\ell\mathcal{C}_{\ell,\ell+2}^{N,R,-,i},
\end{equation*}
\begin{equation*}
\begin{aligned}
\mathcal{C}_{\ell,\ell+2}^{N,R,+,i}g_{\ell+2}(Z_\ell):=A_{N,\epsilon,\ell}\int_{\mathbb{S}_1^{2d-1}\times B_R^{2d}}&b_+(\omega_{\ell+1},\omega_{\ell+2},v_{\ell+1}-v_i,v_{\ell+2}-v_i)\\
&\times g_{\ell+2}\mathds{1}_{[E_{\ell+2}\leq R^2]}(Z_{\ell+2,\epsilon}^{i*})\,d\omega_{\ell+1}\,d\omega_{\ell+2}\,dv_{\ell+1}\,dv_{\ell+2},
\end{aligned}
\end{equation*} 
\begin{equation*}
\begin{aligned}
\mathcal{C}_{\ell,\ell+2}^{N,R,-,i}g_{\ell+2}(Z_\ell):=A_{N,\epsilon,\ell}\int_{\mathbb{S}_1^{2d-1}\times B_R^{2d}}&b_+(\omega_{\ell+1},\omega_{\ell+2},v_{\ell+1}-v_i,v_{\ell+2}-v_i)\\
&\times g_{\ell+2}\mathds{1}_{[E_{\ell+2}\leq R^2]}(Z_{\ell+2,\epsilon}^i)\,d\omega_{\ell+1}\,d\omega_{\ell+2}\,dv_{\ell+1}\,dv_{\ell+2}.
\end{aligned}
\end{equation*}
For $s\in\mathbb{N}$ and $k\in\mathbb{N}$, let us denote $\mathcal{U}_{s,k}=\mathcal{A}_{s,k}\times\mathcal{B}_{s,k}$, where
\begin{align}
\mathcal{A}_{s,k}:&=\left\{J=(j_1,...,j_k)\in\mathbb{N}^k:j_i\in\left\{-1,1\right\},\quad\forall i\in\left\{1,...,k\right\}\right\}\label{J_k},\\
\mathcal{B}_{s,k}:&=\left\{M=(m_1,...,m_k)\in\mathbb{N}^k:m_i\in\left\{1,...,s+2i-2\right\},\quad\forall i\in\left\{1,...,k\right\}\right\}\label{M_k}.
\end{align}
Under this notation, given $s,n\in\mathbb{N}$, parameters $\alpha,\epsilon_0,R,\eta,\delta$ as in \eqref{choice of parameters}, $1\leq k\leq n$,  $(N,\epsilon)$ in the scaling \eqref{scaling} with $\epsilon<<\alpha$, and $t\in[0,T]$, the BBGKY hierarchy observable functional $\widetilde{I}_{s,k,R,\delta}^N(t)$ (given in \eqref{good observables BBGKY }) can be expressed as a superposition of elementary observables
\begin{equation}\label{superposition BBGKY}
\widetilde{I}_{s,k,R,\delta}^N(t)(X_s)=\sum_{(J,M)\in\mathcal{U}_{s,k}}\left(\prod_{i=1}^kj_i\right)\widetilde{I}_{s,k,R,\delta}^N(t,J,M)(X_s),
\end{equation}
\begin{equation}\label{elementary observable BBGKY}
\begin{aligned}
\widetilde{I}_{s,k,R,\delta}^N(t,J,M)(X_s):=\int_{\mathcal{M}_s^c(X_s)}\phi_s(V_s)\int_{\mathcal{T}_{k,\delta}(t)}T_s^{t-t_1}\mathcal{C}_{s,s+2}^{N,R,j_1,m_1} T_{s+2}^{t_1-t_2}...\mathcal{C}_{s+2k-2,s+2k}^{N,R,j_k,m_k} T_{s+2k}^{t_m}f_{0}^{(s+2k)}(Z_s)\,dt_k...\,dt_{1}dV_s.
\end{aligned}
\end{equation}

Similarly, given $\ell,N\in\mathbb{N}$ with $\ell<N$, $R>1$, we decompose the truncated Boltzmann hierarchy collisional operator (given in  \eqref{Boltzmann operator}-\eqref{boltzmann notation}) as:
\begin{equation*}
\mathcal{C}_{\ell,\ell+2}^{\infty,R}=\sum_{i=1}^\ell\mathcal{C}_{\ell,\ell+2}^{\infty,R,+,i}-\sum_{i=1}^\ell\mathcal{C}_{\ell,\ell+2}^{\infty,R,-,i},
\end{equation*}
\begin{equation*}
\mathcal{C}_{\ell,\ell+2}^{\infty,R,+,i}g_{\ell+2}(Z_\ell):=\int_{\mathbb{S}_1^{2d-1}\times B_R^{2d}}b_+(\omega_{\ell+1},\omega_{\ell+2},v_{\ell+1}-v_i,v_{\ell+2}-v_i) g_{\ell+2}\mathds{1}_{[E_{\ell+2}\leq R^2]}(Z_{\ell+2}^{i*})\,d\omega_{\ell+1}\,d\omega_{\ell+2}\,dv_{\ell+1}\,dv_{\ell+2},
\end{equation*}
\begin{equation*}
\mathcal{C}_{\ell,\ell+2}^{\infty,R,-,i}g_{\ell+2}(Z_\ell):=\int_{\mathbb{S}_1^{2d-1}\times B_R^{2d}}b_+(\omega_{\ell+1},\omega_{\ell+2},v_{\ell+1}-v_i,v_{\ell+2}-v_i) g_{\ell+2}\mathds{1}_{[E_{\ell+2}\leq R^2]}(Z_{\ell+2}^i)\,d\omega_{\ell+1}\,d\omega_{\ell+2}\,dv_{\ell+1}\,dv_{\ell+2}.
\end{equation*}

Under this notation, given $s,n\in\mathbb{N}$,  $t\in[0,T]$, parameters $\alpha,\epsilon_0,R,\eta,\delta$ as in \eqref{choice of parameters}, $1\leq k\leq n$, and $t\in[0,T]$, the Boltzmann hierarchy observable functional $\widetilde{I}_{s,k,R,\delta}^\infty(t)$ (given in \eqref{good observables Boltzmann}) can be expressed as a superposition of elementary observables
\begin{equation}\label{superposition Boltzmann}
\widetilde{I}_{s,k,R,\delta}^\infty(t)(X_s)=\sum_{(J,M)\in\mathcal{U}_{s,k}}\left(\prod_{i=1}^kj_i\right)\widetilde{I}_{s,k,R,\delta}^\infty(t,J,M)(X_s),
\end{equation}
\begin{equation}\label{elementary observable Boltzmann}
\begin{aligned}
\widetilde{I}_{s,k,R,\delta}^\infty(t,J,M)(X_s)&:=\int_{\mathcal{M}_s^c(X_s)}\phi_s(V_s)\int_{\mathcal{T}_{k,\delta}(t)}S_s^{t-t_1}\mathcal{C}_{s,s+2}^{\infty,R,j_1,m_1} S_{s+2}^{t_1-t_2}...\mathcal{C}_{s+2k-2,s+2k}^{\infty,R,j_k,m_k} S_{s+2k}^{t_m}f_{0}^{(s+2k)}(Z_s)\,dt_k...\,dt_{1}dV_s.
\end{aligned}
\end{equation}

\subsection{Boltzmann pseudo-trajectories}\label{Boltzmann pseudotrajectory sub}
In this subsection, we introduce an explicit discrete backwards in time construction of so called Boltzmann  pseudo-trajectory, which lets us keep track of the collisions.  Similar constructions, although continuous in time, can be found in \cite{lanford}, \cite{gallagher}, \cite{denlinger}. Let $s\in\mathbb{N}$, $Z_s=(X_s,V_s)\in\mathbb{R}^{2ds}$, $k\in\mathbb{N}$ and $t\in[0,T]$. Given $\delta>0$, let us recall from \eqref{separated collision times} the set
$\mathcal{T}_{k,\delta}(t).$

Consider $(t_1,...,t_k)\in\mathcal{T}_{k,\delta}(t)$, $J=(j_1,...,j_k)$, $M=(m_1,...,m_k)$, $(J,M)\in\mathcal{U}_{s,k}$, and for each $i=1,...,k$, we consider  $(\omega_{s+2i-1},\omega_{s+2i},v_{s+2i-1},v_{s+2i})\in\mathbb{S}_{1}^{2d-1}\times\mathbb{R}^{2d}.$
We inductively define the Boltzmann  pseudo-trajectory of $Z_s$.
Roughly speaking, the Boltzmann pseudo-trajectory is formulated as follows:

Assume we are given a configuration $Z_s=(X_s,V_s)\in\mathbb{R}^{2ds}$ 
at time $t_0=t$. $Z_s$ evolves under backwards free flow until the time $t_1$ when a pair of particles $(\omega_{s+1},\omega_{s+2},v_{s+1},v_{s+2})$ is added to the $m_1$-particle, the adjunction being pre-collisional if $j_1=-1$ and post-collisional if $j_1=1$. We then form an $(s+2)$-configuration and continue this process inductively until time $t_{k+1}=0$.
More precisely, given
$Z_s=(X_s,V_s)\in\mathbb{R}^{2ds}$:\\
{\bf{Time $t_0=t$}:} We initially define 
$$
Z_s^\infty(t_{0}^-)=\left(x_1^\infty(t_0^-),...,x_s^\infty(t_0^-),v_1^\infty(t_0^-),...,v_s^\infty(t_0^-)\right):=Z_s.
$$\\
{\bf{Time $t_i$}, $i\in\{1,...,k\}$:} Consider $i\in\left\{1,...,k\right\}$, and assume we know 
$$Z_{s+2i-2}^\infty (t_{i-1}^-)=\left(x_1^\infty(t_{i-1}^-),...,x_{s+2i-2}^\infty(t_{i-1}^-),v_1^\infty(t_{i-1}^-),...,v_{s+2i-2}^\infty(t_{i-1}^-)\right).$$
We define $Z_{s+2i-2}^\infty (t_{i}^+)=\left(x_1^\infty(t_{i}^+),...,x_{s+2i-2}^\infty(t_{i}^+),v_1^\infty(t_{i}^+),...,v_{s+2i-2}^\infty(t_{i}^+)\right)$ as:
\begin{equation*}
Z_{s+2i-2}^\infty(t_i^+):=\left(X_{s+2i-2}^\infty\left(t_{i-1}^-\right)-\left(t_{i-1}-t_i\right)V_{s+2i-2}^\infty\left(t_{i-1}^-\right),
V_{s+2i-2}^\infty\left(t_{i-1}^-\right)\right).
\end{equation*}
We also define $Z_{s+2i}^\infty(t_i^-)=\left(x_1^\infty(t_{i}^-),...,x_{s+2i}^\infty(t_{i}^-),v_1^\infty(t_{i}^-),...,v_{s+2i}^\infty(t_{i}^-)\right)$ as:
\begin{equation*}
\left(x_j^\infty(t_i^-),v_j^\infty(t_i^-)\right):=\left(x_j^\infty(t_i^+),v_j^\infty(t_i^+)\right)\quad\forall j\in\{1,...,s+2i-2\}\setminus\left\{m_i\right\},
\end{equation*}
and if $j_i=-1$:
\begin{equation*}
\begin{aligned}
\left(x_{m_i}^\infty(t_i^-),v_{m_i}^\infty(t_i^-)\right):&=\left(x_{m_{i}}^\infty(t_{i}^+),v_{m_{i}}^\infty(t_{i}^+)\right),\\
\left(x_{s+2i-1}^\infty(t_i^-),v_{s+2i-1}^\infty(t_i^-)\right):&=\left(x_{m_{i}}^\infty(t_{i}^+),v_{s+2i-1}\right),\\
\left(x_{s+2i}^\infty(t_i^-),v_{s+2i}^\infty(t_i^-)\right):&=\left(x_{m_{i}}^\infty(t_{i}^+),v_{s+2i}\right),
\end{aligned}
\end{equation*}
while if $j_i=1$:
\begin{equation*}
\begin{aligned}
\left(x_{m_i}^\infty(t_i^-),v_{m_i}^\infty(t_i^-)\right):&= \left(x_{m_{i}}^\infty(t_{i}^+),v_{m_{i}}^{\infty*}(t_{i}^+)\right),\\
\left(x_{s+2i-1}^\infty(t_i^-),v_{s+2i-1}^\infty(t_i^-)\right):&=\left(x_{m_{i}}^\infty(t_{i}^+),v_{s+2i-1}^*\right),\\
\left(x_{s+2i}^\infty(t_i^-),v_{s+2i}^\infty(t_i^-)\right):&= \left(x_{m_{i}}^\infty(t_{i}^+),v_{s+2i}^*\right),\\
(v_{m_{i}}^{\infty*}(t_{i}^-),v_{s+2i-1}^*,v_{s+2i}^*)&=T_{\omega_{s+2i-1},\omega_{s+2i}}\left(v_{m_{i}}^\infty(t_{i}^+),v_{s+2i-1},v_{s+2i}\right).
\end{aligned}
\end{equation*}
{\bf{Time $t_{k+1}=0$}:}
We finally obtain 
$$Z_{s+2k}^\infty(0^+)=Z_{s+2k}^\infty(t_{k+1}^+)=\left(X_{s+2k}^\infty\left(t_{k}^-\right)-t_kV_{s+2k}^\infty\left(t_k^-\right),V_{s+2k}^\infty\left(t_k^-\right)\right).$$

The sequence $Z_{s+2i}^\infty(t_i^+)$, $i=0,...,k+1$ is called Boltzmann pseudo-trajectory of $Z_s$.

 The construction process is illustrated in Figure \ref{pseudo} (to be read from right to left):
\begin{figure}[htp]
\centering
\begin{tikzpicture}[node distance=2.6cm,auto,>=latex']\label{boltzmann pseudo diagram}
\node[int](0-){\small$ Z_s^\infty(t_0^-)$};
\node[int,pin={[init]above:\small$\begin{matrix}(\omega_{s+1},\omega_{s+2},v_{s+1},v_{s+2}),\\(j_1,m_1)\end{matrix}$}](1+)[left of=0-,node distance=2.4cm]{\small$Z_s^\infty(t_1^+)$};
\node[int](1-)[left of=1+,node distance=1.5cm]{$Z_{s+2}^\infty(t_1^-)$};
\node[](intermediate1)[left of=1-,node distance=2cm]{...};
\node[int,pin={[init]above:\small$\begin{matrix}(\omega_{s+2i-1},\omega_{s+2i},v_{s+2i-1},v_{s+2i}),\\(j_i,m_i)\end{matrix}$}](i+)[left of=intermediate1,node distance=2.6cm]{\small$Z_{s+2i-2}^\infty(t_i^+)$};
\node[int](i-)[left of=i+,node distance=1.8cm]{\small$Z_{s+2i}^\infty(t_i^-)$};
\node[](intermediate2)[left of=i-,node distance=2.2cm]{...};
\node[int](end)[left of=intermediate2,node distance=2.6cm]{\small$Z_{s+2k}^\infty(t_{k+1}^+)$};

\path[<-] (1+) edge node {\tiny$t_{0}-t_1$} (0-);
\path[<-] (intermediate1) edge node {\tiny$t_{1}-t_2$} (1-);
\path[<-] (i+) edge node {\tiny$t_{i-1}-t_i$} (intermediate1);
\path[<-] (intermediate2) edge node {\tiny$t_{i}-t_{i+1}$} (i-);
\path[<-] (end) edge node {\tiny$t_{k}-t_{k+1}$} (intermediate2);
\end{tikzpicture}
\caption{ }
\label{pseudo}
\end{figure}

 \subsection{Reduction to truncated elementary observables}\label{par_reduction to truncated}
 
We  now use the Boltzmann  pseudo-trajectory to define the truncated observables for the BBGKY hierarchy and Boltzmann hierarchy. The  proof will then be reduced to the convergence of the corresponding truncated elementary observables. Given $\ell\in\mathbb{N}$, parameters $\alpha,\epsilon_0,R,\eta,\delta$ as in \eqref{choice of parameters} and $\epsilon<<\alpha$, recall the set  $G_\ell(\epsilon,\epsilon_0,\delta)$ from \eqref{both epsilon-epsilon_0}.

Let $s\in \mathbb{N}$, $X_s\in\Delta_s^X(\epsilon_0)$, $1\leq k\leq n$, $(J,M)\in\mathcal{U}_{s,k}$ and $t\in[0,T]$ and $(t_1,...,t_k)\in\mathcal{T}_{k,\delta}(t)$, where we recall from \eqref{separated collision times} the set
$\mathcal{T}_{k,\delta}(t)$.
 By Proposition \ref{initially good configurations}, for any $V_s\in\mathcal{M}_s^c(X_s)$, we have 
$Z_s=(X_s,V_s)\in G_s(\epsilon,\epsilon_0,\delta).$
Since $t_0-t_1\geq\delta$, we obtain $Z_s^\infty(t_1^+)\in G_s(\epsilon_0,0).$
Recalling notation from \eqref{pre-post notation}, Proposition \ref{bad set} (see \eqref{pre-delta-} for the pre-collisional case or \eqref{post-delta-} for the post-collisional case) yields there is a set $\mathcal{B}_{m_1}\left(Z_s^{\infty}\left(t_1^+\right)\right)\subseteq (\mathbb{S}_1^{2d-1}\times B_R^{2d})^+\left(v_{m_1}^\infty\left(t_1^+\right)\right)$  such that 
$$Z_{s+2}^\infty(t_2^+)\in G_{s+2}(\epsilon_0,0),\quad\forall (\omega_{s+1},\omega_{s+2},v_{s+1},v_{s+2})\in\mathcal{B}_{m_1}^c\left(Z_s^{\infty}\left(t_1^+\right)\right).$$
Clearly this process can be iterated. In particular, given $i\in\left\{2,...,k\right\}$, we have 
  $Z_{s+2i-2}^\infty(t_{i}^+)\in G_{s+2i-2}(\epsilon_0,0),$
  so there exists a set $\mathcal{B}_{m_i}\left(Z_{s+2i-2}^\infty\left(t_{i}^+\right)\right)\subseteq (\mathbb{S}_1^{2d-1}\times B_R^{2d})^+\left(v_{m_i}^\infty\left(t_1^+\right)\right)$ such that: 
  \begin{equation}
  Z_{s+2i}^\infty(t_{i+1}^+)\in G_{s+2i}(\epsilon_0,0),\quad\forall (\omega_{s+2i-1},\omega_{s+2i},v_{s+2i-1},v_{s+2i})\in \mathcal{B}_{m_i}^c\left(Z_{s+2i-2}^\infty\left(t_{i}^+\right)\right).
  \end{equation}
  We finally  obtain $Z_{s+2k}^\infty(0^+)\in G_{s+2k}(\epsilon_0,0)$.

Let us now define the truncated elementary observables. Heuristically we will truncate the domains of adjusted particles in the definition of the observables $\widetilde{I}_{s,k,R,\delta}^N$, $\widetilde{I}_{s,k,R,\delta}^\infty$ (see \eqref{good observables BBGKY }-\eqref{good observables Boltzmann}).

More precisely, let $s,n\in\mathbb{N}$, $\alpha,\epsilon_0,R,\eta,\delta$ be parameters as in \eqref{choice of parameters},  $(N,\epsilon)$ in the scaling \eqref{scaling} with $\epsilon<<\alpha$, $1\leq k\leq n$, $(J,M)\in\mathcal{U}_{s,k}$ and $t\in [0,T]$. For $X_s\in\Delta_s^X(\epsilon_0)$, Proposition \ref{initially good configurations} implies there is a set of velocities $\mathcal{M}_s(X_s)\subseteq B_R^{2d}$ such that 
$Z_s=(X_s,V_s)\in G_s(\epsilon,\epsilon_0,\delta)$, for all $V_s\in\mathcal{M}_s^c(X_s).$
 Following the reasoning above, we define the BBGKY hierarchy truncated observables as:
\begin{equation}\label{truncated BBGKY}
J_{s,k,R,\delta}^N(t,J,M)(X_s):=\int_{\mathcal{M}_s^c(X_s)}\phi_s(V_s)\int_{\mathcal{T}_{k,\delta}(t)}T_s^{t-t_1}\widetilde{\mathcal{C}}_{s,s+2}^{N,R,j_1,m_1} T_{s+2}^{t_1-t_2}...\widetilde{\mathcal{C}}_{s+2k-2,s+2k}^{N,R,j_k,m_k} T_{s+2k}^{t_m}f_{0}^{(s+2k)}(Z_s)\,dt_k,...\,dt_{1}dV_s,
\end{equation}
where 
$
\widetilde{\mathcal{C}}_{s+2i-2,s+2i}^{N,R,j_i,m_i}g_{N,s+2i}:=\mathcal{C}_{s+2i-2,s+2i}^{N,R,j_i,m_i}\left[g_{N,s+2i}\mathds{1}_{(\omega_{s+2i-1},\omega_{s+2i},v_{s+2i-1},v_{s+2i})\in
\mathcal{B}^c_{m_i}\left(Z_{s+2i-2}^\infty\left(t_i^+\right)\right)}\right].
$

In the same spirit, for $X_s\in\Delta_s^X(\epsilon_0)$, we define the Boltzmann hierarchy truncated elementary observables as:
\begin{equation}\label{truncated Boltzmann}
J_{s,k,R,\delta}^\infty(t,J,M)(X_s):=\int_{\mathcal{M}_s^c(X_s)}\phi_s(V_s)\int_{\mathcal{T}_{k,\delta}(t)}S_s^{t-t_1}\widetilde{\mathcal{C}}_{s,s+2}^{\infty,R,j_1,m_1} S_{s+2}^{t_1-t_2}...\widetilde{\mathcal{C}}_{s+2k-2,s+2k}^{\infty,R,j_k,m_k} S_{s+2k}^{t_m}f_{0}^{(s+2k)}(Z_s)\,dt_k,...\,dt_{1}dV_s,
\end{equation}
where 
$\widetilde{\mathcal{C}}_{s+2i-2,s+2i}^{\infty,R,j_i,m_i}g_{s+2i}:=\mathcal{C}_{s+2i-2,s+2i}^{\infty,R,j_i,m_i}\left[g_{s+2i}\mathds{1}_{(\omega_{s+2i-1},\omega_{s+2i},v_{s+2i-1},v_{s+2i})\in\mathcal{B}^c_{m_i}\left(Z_{s+2i-2}^\infty\left(t_i^+\right)\right)}\right].
$
 
 Recalling the observables $\widetilde{I}_{s,k,R,\delta}^N$, $\widetilde{I}_{s,k,R,\delta}^\infty$ from \eqref{elementary observable BBGKY}, \eqref{elementary observable Boltzmann} and using Proposition \ref{measure of B} (since we integrate at least in one of the bad sets), we obtain:
 \begin{proposition}\label{truncated element estimate} Let $s,n\in\mathbb{N}$,  $\alpha,\epsilon_0,R,\eta,\delta$ be parameters as in \eqref{choice of parameters},  $(N,\epsilon)$   in the scaling \eqref{scaling} with $\epsilon<<\alpha$ and $t\in[0,T]$.  Then the following estimates hold uniformly in $N$:
 \begin{equation*}
 \sum_{k=1}^n\sum_{(J,M)\in\mathcal{U}_{s,k}}\hspace{-0.4cm}\|\widetilde{I}_{s,k,R,\delta}^N(t,J,M)-J_{s,k,R,\delta}^N(t,J,M)\|_{L^\infty\left(\Delta_s^X\left(\epsilon_0\right)\right)}\leq C_{d,s,\mu_0,T}^n\|\phi_s\|_{L^\infty_{V_s}} R^{d(s+3n)}\eta^{\frac{d-1}{4d+2}}\|F_{N,0}\|_{N,\beta_0,\mu_0},
 \end{equation*}
 \begin{equation*}
 \sum_{k=1}^n\sum_{(J,M)\in\mathcal{U}_{s,k}}\hspace{-0.4cm}\|\widetilde{I}_{s,k,R,\delta}^\infty(t,J,M)-J_{s,k,R,\delta}^\infty(t,J,M)\|_{L^\infty\left(\Delta_s^X\left(\epsilon_0\right)\right)}\leq 
 C_{d,s,\mu_0,T}^n\|\phi_s\|_{L^\infty_{V_s}} R^{d(s+3n)}\eta^{\frac{d-1}{4d+2}}\|F_{0}\|_{\infty,\beta_0,\mu_0}.
 \end{equation*}
 \end{proposition}
 \begin{proof}
 As usual, it suffices to prove the estimate for the BBGKY hierarchy case and the Boltzmann hierarchy case follows similarly. Fix $k\in\left\{1,...,n\right\}$ and $(J,M)\in\mathcal{U}_{s,k}$. We first estimate the difference:
 \begin{equation}\label{estimated difference}
 \widetilde{I}_{s,k,R,\delta}^N(t,J,M)(X_s)-J_{s,k,R,\delta}^N(t,J,M)(X_s).
 \end{equation}
 Triangle and Cauchy-Scwhartz inequalities yield
 \begin{equation}\label{triangle on cross}\big|b(\omega_1,\omega_2,v_1-v,v_2-v)\big|\leq 4R,\quad\forall(\omega_1,\omega_2)\in\mathbb{S}_1^{2d-1}, \quad\forall v,v_1,v_2\in B_R^{d},
 \end{equation}
 so
 \begin{equation}\label{estimate on the rest of the terms}
 \int_{\mathbb{S}_1^{2d-1}\times B_R^{2d}}|b(\omega_1,\omega_2,v_1-v,v_2-v_2)|\,d\omega_1\,d\omega_2\,dv_1\,dv_2\leq C_d R^{2d+1}\leq C_dR^{3d},\quad\forall v\in B_R^d.
 \end{equation}
 
 But in order to estimate the difference \eqref{estimated difference}, we integrate at least once over $\mathcal{B}_{m_i}\left(Z_{s+2i-2}^\infty\left(t_{i}^+\right)\right)$ for some $i\in\left\{1,...,k\right\}$. Proposition \ref{measure of B} and the expression \eqref{triangle on cross} yield the estimate:
 \begin{equation}\label{exclusion bad set 1}
 \begin{aligned}
 \int_{\mathcal{B}_{m_i}\left(Z_{s+2i-2}^\infty\left(t_{i}^+\right)\right)}|b(\omega_1,\omega_2,v_1-v,v_2-v)|\,d\omega_1\,d\omega_2\,dv_1\,dv_2&\leq C_d(s+2i-2)R^{2d+1}\eta^{\frac{d-1}{4d+2}}\\
 &\leq C_d(s+2k)R^{3d}\eta^{\frac{d-1}{4d+2}} ,\quad\forall v\in B_R^d.
 \end{aligned}
 \end{equation}
  Moreover, we have the elementary inequalities:
 \begin{align}
  \|f_{N,0}^{(s+2k)}\|_{L^\infty}&\leq e^{-(s+2k)\mu_0}\|F_{N,0}\|_{N,\beta_0,\mu_0},\label{exclusion bad set 2 norm}\\
  \int_{T_{k,\delta}(t)}\,dt_1...\,dt_k&\leq\int_0^t\int_0^{t_1}...\int_0^{t_{k-1}}\,dt_1...\,dt_k\leq\frac{T^k}{k!}.\label{exclusion bad set 2 time}
 \end{align}
 Therefore, \eqref{estimate on the rest of the terms}-\eqref{exclusion bad set 2 time} imply
 \begin{equation*}
 \begin{aligned}
 \big|&\widetilde{I}_{s,k,R,\delta}^N(t,J,M)(X_s)-J_{s,k,R,\delta}^N(t,J,M)(X_s)\big|\leq \\
 &\leq \|\phi_s\|_{L^\infty_{V_s}}e^{-(s+2k)\mu_0}\|F_{N,0}\|_{N,\beta_0,\mu_0}C_d R^{ds}C_d^{k-1}R^{3d(k-1)}(s+2k) C_dR^{3d}\eta^{\frac{d-1}{4d+2}}\frac{T^k}{k!}\\
 &\leq C_{d,s,\mu_0,T}^k\|\phi_s\|_{L^\infty_{V_s}}\frac{(s+2k)}{k!}R^{d(s+3k)}\eta^{\frac{d-1}{4d+2}} \|F_{N,0}\|_{N,\beta_0,\mu_0}.
 \end{aligned}
 \end{equation*}
 Adding for all $(J,M)\in\mathcal{U}_{s,k}$, we get  $2^ks(s+2)...(s+2k-2)$ contributions, thus
 \begin{equation}\label{use of stirling}
 \begin{aligned}
 &\sum_{(J,M)\in\mathcal{U}_{s,k}}\|\widetilde{I}_{s,k,R,\delta}^N(t,J,M)-J_{s,k,R,\delta}^N(t,J,M)\|_{L^\infty\left(\Delta_s^X\left(\epsilon_0\right)\right)}\\
 &\leq C_{d,s,\mu_0,T}^k\|\phi_s\|_{L^\infty_{V_s}}R^{d(s+3k)}\frac{(s+2k)^{k+1}}{k!}\eta^{\frac{d-1}{4d+2}}\|F_{N,0}\|_{N,\beta_0,\mu_0}\\
 &\leq C_{d,s,\mu_0,T}^k\|\phi_s\|_{L^\infty_{V_s}}R^{d(s+3k)}\eta^{\frac{d-1}{4d+2}}\|F_{N,0}\|_{N,\beta_0,\mu_0},
 \end{aligned}
 \end{equation}
since 
$\displaystyle\frac{(s+2k)^{k+1}}{k!}=\frac{(s+2k)(s+2k)^k}{k!}\leq C_s^k,$
 Summing over $k=1,...,n$, we obtain the required estimate.
 \end{proof}

\section{Convergence proof}\label{sec:convergence proof}
In Subsection \ref{par_reduction to truncated}, given $s,n\in\mathbb{N}$,  parameters  $\alpha,\epsilon_0,R,\eta,\delta$ as in \eqref{choice of parameters},  $(N,\epsilon)$   in the scaling \eqref{scaling} with $\epsilon<<\alpha$ and $t\in[0,T]$,  we have reduced the convergence proof to controlling the differences $J_{s,k,R,\delta}^N(t,J,M)-J_{s,k,R,\delta}^\infty (t,J,M)$ for given $1\leq k\leq n$ and  $(J,M)\in\mathcal{U}_{s,k}$, where $J_{s,k,R,\delta}^N(t,J,M)$, $J_{s,k,R,\delta}^\infty (t,J,M)$ are given by \eqref{truncated BBGKY}-\eqref{truncated Boltzmann}, respectively. 
Throughout this section $s\in\mathbb{N}$ will be fixed. We also consider $\beta_0>0$, $\mu_0\in\mathbb{R}$, $T>0$ and $F_0\in X_{\infty,\beta_0,\mu_0}$ as in the statement of Theorem \ref{convergence theorem}.

\subsection{BBGKY  pseudo-trajectories and proximity to the Boltzmann  pseudo-trajectories}
Consider $s\in\mathbb{N}$, $(N,\epsilon)$ in the scaling \eqref{scaling}, $k\in\mathbb{N}$ and $t\in[0,T]$. Given $\delta>0$ recall from \eqref{separated collision times} the set
$\mathcal{T}_{k,\delta}(t)$.
Let $Z_s=(X_s,V_s)\in\mathbb{R}^{2ds}$, $(t_1,...,t_k)\in\mathcal{T}_k(t)$, $J=(j_1,...,j_k)$, $M=(m_1,...,m_k)$, $(J,M)\in\mathcal{U}_{s,k}$, and for each $i=1,...,k$, we consider  $(\omega_{s+2i-1},\omega_{s+2i},v_{s+2i-1},v_{s+2i})\in\mathbb{S}_{1}^{2d-1}\times B_R^{2d}.$

In the same spirit as in Subsection \ref{Boltzmann pseudotrajectory sub} where we introduced the Boltzmann pseudo-trajectory, we define the BBGKY  pseudo-trajectory, the main difference being that we take into account the interaction zone of the adjusted particles in each step. 
More precisely, given
$Z_s=(X_s,V_s)\in\mathbb{R}^{2ds}$:\\
{\bf{Time $t_0=t$}:} We initially define 
$$
Z_s^N(t_{0}^-)=\left(x_1^N(t_0^-),...,x_s^N(t_0^-),v_1^N(t_0^-),...,v_s^N(t_0^-)\right):=Z_s.
$$\\
{\bf{Time $t_i$}, $i\in\{1,...,k\}$:} Consider $i\in\left\{2,...,k\right\}$, and assume we know 
$$Z_{s+2i-2}^N (t_{i-1}^-)=\left(x_1^N(t_{i-1}^-),...,x_{s+2i-2}^N(t_{i-1}^-),v_1^N(t_{i-1}^-),...,v_{s+2i-2}^N(t_{i-1}^-)\right).$$
We define $Z_{s+2i-2}^N (t_{i}^+)=\left(x_1^N(t_{i}^+),...,x_{s+2i-2}^N(t_{i}^+),v_1^N(t_{i}^+),...,v_{s+2i-2}^N(t_{i}^+)\right)$ as:
\begin{equation*}
Z_{s+2i-2}^N(t_i^+):=\left(X_{s+2i-2}^N\left(t_{i-1}^-\right)-\left(t_{i-1}-t_i\right)V_{s+2i-2}^N\left(t_{i-1}^-\right),
V_{s+2i-2}^N\left(t_{i-1}^-\right)\right).
\end{equation*}
We also define $Z_{s+2i}^N(t_i^-)=\left(x_1^N(t_{i}^-),...,x_{s+2i}^N(t_{i}^-),v_1^N(t_{i}^-),...,v_{s+2i}^N(t_{i}^-)\right)$ as:
\begin{equation*}
\left(x_j^N(t_i^-),v_j^N(t_i^-)\right):=\left(x_j^N(t_i^+),v_j^N(t_i^+)\right),\quad\forall j\in\{1,...,s+2i-2\}\setminus\left\{m_i\right\},
\end{equation*}
and if $j_i=-1$:
\begin{equation*}
\begin{aligned}
\left(x_{m_i}^N(t_i^-),v_{m_i}^N(t_i^-)\right):&=\left(x_{m_{i}}^N(t_{i}^+),v_{m_{i}}^N(t_{i}^+)\right),\\
\left(x_{s+2i-1}^N(t_i^-),v_{s+2i-1}^N(t_i^-)\right):&=\left(x_{m_{i}}^N(t_{i}^+)-\sqrt{2}\epsilon\omega_{s+2i-1},v_{s+2i-1}\right),\\
\left(x_{s+2i}^N(t_i^-),v_{s+2i}^N(t_i^-)\right):&=\left(x_{m_{i}}^N(t_{i}^+)-\sqrt{2}\epsilon\omega_{s+2i},v_{s+2i}\right),
\end{aligned}
\end{equation*}
while if $j_i=1$:
\begin{equation*}
\begin{aligned}
\left(x_{m_i}^N(t_i^-),v_{m_i}^N(t_i^-)\right):&=\left(x_{m_{i}}^N(t_{i}^+),v_{m_{i}}^{N*}(t_{i}^+)\right),\\
\left(x_{s+2i-1}^N(t_i^-),v_{s+2i-1}^N(t_i^-)\right):&=\left(x_{m_{i}}^N(t_{i}^+)+\sqrt{2}\epsilon\omega_{s+2i-1},v_{s+2i-1}^*\right),\\
\left(x_{s+2i}^N(t_i^-),v_{s+2i}^N(t_i^-)\right):&=\left(x_{m_{i}}^N(t_{i}^+)+\sqrt{2}\epsilon\omega_{s+2i},v_{s+2i}^*\right),\\
(v_{m_{i}}^{N*}(t_{i}^-),v_{s+2i-1}^*,v_{s+2i}^*)&=T_{\omega_{s+2i-1},\omega_{s+2i}}\left(v_{m_{i}}^N(t_{i}^+),v_{s+2i-1},v_{s+2i}\right).
\end{aligned}
\end{equation*}
{\bf{Time $t_{k+1}=0$}:}
We finally obtain 
$$Z_{s+2k}^N(0^+)=Z_{s+2k}^N(t_{k+1}^+)=\left(X_{s+2k}^N\left(t_{k}^-\right)-t_kV_{s+2k}^N\left(t_k^-\right),V_{s+2k}^N\left(t_k^-\right)\right).$$
The sequence $Z_{s+2i}^N(t_i^+)$, $i=0,...,k+1$ is called BBGKY pseudo-trajectory of $Z_s$. The construction  can be illustrated by an analogous diagram to Figure \ref{pseudo}.

We now state a  proximity result for the corresponding BBGKY  and Boltzmann pseudo-trajectories. The proof of this result follows inductively from the definition of the pseudo-trajectories, for more details see \cite{thesis}.
 \begin{lemma}\label{proximity}
  Let $s,n\in\mathbb{N}$, $(N,\epsilon)$ in the scaling \eqref{scaling}, $1\leq k\leq n$, $(J,M)\in
\mathcal{U}_{s,k}$, $t\in[0,T]$ and $(t_1,...,t_k)\in\mathcal{T}_{k}(t)$. Fix $Z_s=(X_s,V_s)\in\mathbb{R}^{2ds}$. For each $i=1,...,k$, consider $(\omega_{s+2i-1},\omega_{s+2i},v_{s+2i-1},v_{s+2i})\in\mathbb{S}_{1}^{2d-1}\times\mathbb{R}^{2d}$. Then for all $i=1,...,k+1$ and $\ell=1,...,s+2i-2$, we have
 \begin{equation}\label{proximity claim}
 |x_{\ell}^N(t_{i}^+)-x_{\ell}^\infty(t_{i}^+)|\leq \sqrt{2}\epsilon (i-1),\quad v_\ell^N(t_{i}^+)=v_\ell^\infty(t_{i}^+).
 \end{equation}
 In particular, if $s<n$, there holds:
 \begin{equation}\label{total proximity}
 \left|X_{s+2i-2}^N(t_i^+)-X_{s+2i-2}^\infty(t_i^+)\right|\leq \sqrt{6}n^{3/2}\epsilon,\quad\forall i=1,...,k+1.
 \end{equation}
 \end{lemma}

 \subsection{Reformulation in terms of pseudo-trajectories}
 We will now re-write the Boltzmann hierarchy truncated elementary observables, defined in \eqref{truncated Boltzmann}, and the BBGKY hierarchy truncated elementary observables, defined in \eqref{truncated BBGKY}, in terms of pseudo-trajectories. 
 
Let $s,n\in\mathbb{N}$ with $s<n$, parameters  $\alpha,\epsilon_0,R,\eta,\delta$ as in \eqref{choice of parameters}.  For the Boltzmann hierarchy case, there is always free flow between the collision times. Therefore, for $X_s\in \Delta_s^X(\epsilon_0)$, $1\leq k\leq n$, $(J,M)\in\mathcal{U}_{s,k}$ and $t\in[0,T]$, the Boltzmann hierarchy truncated elementary observable can be written 
\begin{equation}\label{truncated elementary boltzmann}
\begin{aligned}
&J_{s,k,R,\delta}^\infty(t,J,M)(X_s)=\int_{\mathcal{M}_s^c(X_s)}\phi_s(V_s)\int_{\mathcal{T}_{k,\delta}(t)}\int_{\mathcal{B}_{m_1}^c\left(Z_{s}^\infty\left(t_1^+\right)\right)}...\int_{\mathcal{B}_{m_k}^c\left(Z_{s+2k-2}^\infty\left(t_{k}^+\right)\right)}\\
&\prod_{i=1}^{k}b_{+}\left(\omega_{s+2i-1},\omega_{s+2i},v_{s+2i-1}-v_{m_{i}}^\infty\left(t_i^+\right),
v_{s+2i}-v_{m_{i}}^\infty\left(t_i^+\right)\right)
f_{0}^{(s+2k)}\left(Z_{s+2k}^\infty
\left(0^+\right)\right)\\
&\hspace{0.5cm}\times\prod_{i=1}^{k}\left(\,d\omega_{s+2i-1}
\,d\omega_{s+2i}\,dv_{s+2i-1}\,dv_{s+2i}\right)\,dt_k...\,dt_1\,dV_s.
\end{aligned}
\end{equation}

It is not immediate to obtain a comparable expansion at the BBGKY level because of the recollisions. However, thanks to Proposition \ref{bad set} and Lemma \ref{proximity}, this is possible  for $N$ large enough.

More precisely, fix $X_s\in \Delta_s^X(\epsilon_0)$, $1\leq k\leq n$, $(J,M)\in\mathcal{U}_{s,k}$, $t\in[0,T]$ and $(t_1,...,t_k)\in\mathcal{T}_{k,\delta}(t)$. Consider $(N,\epsilon)$ in the scaling \eqref{scaling} with $N$ large enough such that $n^{3/2}\epsilon<<\alpha$. By Proposition \ref{initially good configurations}, given $V_s\in\mathcal{M}_s^c(X_s)$, we have $Z_s=(X_s,V_s)\in G_s(\epsilon,\epsilon_0,\delta)$.  By the definition of the set $G_s(\epsilon,\epsilon_0,\delta)$, see \eqref{both epsilon-epsilon_0}, we have
$
Z_s\in G_s(\epsilon,\epsilon_0,\delta)\Rightarrow Z_s(\tau)\in\mathring{\mathcal{D}}_{s,\epsilon},
$
for all $\tau\geq 0$,
thus
\begin{equation}\label{equality of the flows k=0}
 \Psi_{s}^{\tau-t_0}Z_{s}^N\left(t_0^-\right)=\Phi_{s}^{\tau-t_0}Z_{s}^N\left(t_0^-\right),\quad\forall \tau\in [t_{1},t_{0}],
\end{equation}
 where $\Psi_{s}$, given in \eqref{liouville operator},  denotes the  $\epsilon$-interaction zone  flow of $s$-particles  and  $\Phi_{s}$, given in \eqref{free flow operator}, denotes  the free flow of $s$-particles. We also have
$
Z_s=(X_s,V_s)\in G_s(\epsilon,\epsilon_0,\delta)\Rightarrow Z_{s}^\infty(t_1^+)\in G_s(\epsilon_0,0).
$
Moreover, for all $i\in\{1,...,k\}$, we have seen that for all $(\omega_{s+2i-1},\omega_{s+2i},v_{s+2i-1},v_{s+2i})\in \mathcal{B}_{m_i}^c(Z_{s+2i-2}^\infty(t_i^+))$
\begin{equation}\label{final condition on Boltz pseudo k=1,,,}
Z_{s+2i}^\infty(t_{i+1}^+)\in G_{s+2i}(\epsilon_0,0).
\end{equation}
Since  $s<n$ and $n^{3/2}\epsilon<<\alpha$,  \eqref{total proximity}  from Lemma \ref{proximity} implies
$$
\left|X_{s+2i-2}^N(t_{i}^+)-X_{s+2i-2}^\infty(t_{i}^+)\right|\leq\displaystyle\frac{\alpha}{2},\quad\forall i=1,...,k.
$$
 Then, Proposition \ref{bad set} yields that for any $i=1,...,k$, we have
\begin{equation}\label{equality of the flows}
\Psi_{s+2i}^{\tau-t_i}Z_{s+2i}^N\left(t_i^-\right)=\Phi_{s+2i}^{\tau-t_i}Z_{s+2i}^N\left(t_i^-\right),\quad\forall \tau\in [t_{i+1},t_{i}].
\end{equation}
 Moreover, Lemma \ref{proximity} also implies  that
$v_{m_i}^N(t_i^+)=v_{m_i}^\infty(t_i^+)$, for all $i=1,...,k$.
 Therefore, for $N$ large enough such that $n^{3/2}\epsilon<<\alpha$, \eqref{equality of the flows k=0}, \eqref{equality of the flows} yield the expansion
\begin{equation}\label{truncated elementary bbgky}
\begin{aligned}
&J_{s,k,R,\delta}^N(t,J,M)(X_s)=\bm{A_{N,\epsilon}^{s,k}}\int_{\mathcal{M}_s^c(X_s)}\phi_s(V_s)\int_{\mathcal{T}_{k,\delta}(t)}\int_{\mathcal{B}_{m_1}^c\left(Z_{s}^\infty\left(t_1^+\right)\right)}...\int_{\mathcal{B}_{m_k}^c\left(Z_{s+2k-2}^\infty\left(t_k^+\right)\right)}\\
&\prod_{i=1}^{k}b_{+}\left(\omega_{s+2i-1},\omega_{s+2i},v_{s+2i-1}-v_{m_{i}}^\infty\left(t_i^+\right),
v_{s+2i}-v_{m_{i}}^\infty\left(t_i^+\right)\right)
f_{N,0}^{(s+2k)}\left(Z_{s+2k}^N
\left(0^+\right)\right)\\
&\hspace{0.5cm}\times\prod_{i=1}^{k}\left(\,d\omega_{s+2i-1}
\,d\omega_{s+2i}\,dv_{s+2i-1}\,dv_{s+2i}\right)\,dt_k...\,dt_1\,dV_s,
\end{aligned}
\end{equation}
where, recalling \eqref{A}, we denote
\begin{equation}\label{k-A}
\bm{A_{N,\epsilon}^{s,k}}:=\prod_{i=1}^kA_{N,\epsilon,s+2i-2}=2^{k(d-2)}\epsilon^{k(2d-1)}\prod_{i=1}^k(N-s-2i+2)(N-s-2i+1).
\end{equation}
\begin{remark}\label{inductive scaled limit}Notice that for fixed $s,k\in\mathbb{N}$, $(N,\epsilon)$ in the scaling \eqref{scaling}, there holds the estimate
\begin{align}\label{estimate on scaling}
0<1-\bm{A_{N,\epsilon}^{s,k}}\leq 2^{\frac{d+1}{2}}\epsilon^{d-1/2}k(s+2k-1).
\end{align}
In particular $\bm{A_{N,\epsilon}^{s,k}}\nearrow 1$, as $N\to\infty$ and $\epsilon\to 0$ in the scaling \eqref{scaling}.

\end{remark}
Let us approximate the BBGKY hierarchy initial data by Boltzmann hierarchy initial data defining some auxiliary functionals. Let $s\in\mathbb{N}$ and $X_s\in\Delta_s^X(\epsilon_0)$. For $ 1\leq k\leq n$, $(J,M)\in\mathcal{U}_{s,k}$ and $t\in[0,T]$, we define the auxiliary functional $\widehat{J}_{s,k,R,\delta}^N(t,J,M)$ which differs from $J_{s,k,R,\delta}^N(t,J,M)$ by the absence of the scaling factor $\bm{A_{N,\epsilon}^{s,k}}$ and the use of Boltzmann hierarchy initial data: 
\begin{equation}\label{auxiliary functionals}
\begin{aligned}
&\widehat{J}_{s,k,R,\delta}^N(t,J,M)(X_s):=\int_{\mathcal{M}_s^c(X_s)}\phi_s(V_s)\int_{\mathcal{T}_{k,\delta}(t)}\int_{\mathcal{B}_{m_1}^c\left(Z_{s}^\infty\left(t_1^+\right)\right)}...\int_{\mathcal{B}_{m_k}^c\left(Z_{s+2k-2}^\infty\left(t_k^+\right)\right)}\\
&\prod_{i=1}^{k}b_{+}\left(\omega_{s+2i-1},\omega_{s+2i},v_{s+2i-1}-v_{m_{i}}^\infty\left(t_i^+\right),
v_{s+2i}-v_{m_{i}}^\infty\left(t_i^+\right)\right)
f_{0}^{(s+2k)}\left(Z_{s+2k}^N
\left(0^+\right)\right)\\
&\hspace{0.5cm}\times\prod_{i=1}^{k}\left(\,d\omega_{s+2i-1}
\,d\omega_{s+2i}\,dv_{s+2i-1}\,dv_{s+2i}\right)\,dt_k...\,dt_1\,dV_s.
\end{aligned}
\end{equation}

Due to the scaling \eqref{scaling} and convergence of the initial data, we conclude that the auxiliary functionals approximate the BBGKY hierarchy truncated elementary observables $J_{s,k,R,\delta}^N$, defined in \eqref{truncated elementary bbgky}.
 \begin{proposition}\label{aux prop 1} Let $s,n\in\mathbb{N}$, with $s<n$,  $\alpha,\epsilon_0,R,\eta,\delta$ be parameters as in \eqref{choice of parameters},  and $t\in[0,T]$. Then for any
  $\zeta>0$, there is $N^*=N^*(\zeta)\in\mathbb{N}$, such that for all $(N,\epsilon)$ in the scaling \eqref{scaling} with $N>N^*$, there holds:
 \begin{equation}\label{estimate 1}
\sum_{k=1}^n\sum_{(J,M)\in\mathcal{U}_{s,k}}\|J_{s,k,R,\delta}^N(t,J,M)-\widehat{J}_{s,k,R,\delta}^N(t,J,M)\|_{L^\infty\left(\Delta_s^X\left(\epsilon_0\right)\right)}\leq
 C_{d,s,\mu_0,T}^n\|\phi_s\|_{L^\infty_{V_s}}R^{d(s+3n)}\zeta ^2.
 \end{equation}
In the case of tensorized initial data and approximation by conditioned BBGKY initial data (see Proposition \ref{conditioned proposition}), the estimate can be improved to
 \begin{equation}\label{estimate 1 on rate}
 \sum_{k=1}^n\sum_{(J,M)\in\mathcal{U}_{s,k}}\|J_{s,k,R,\delta}^N(t,J,M)-\widehat{J}_{s,k,R,\delta}^N(t,J,M)\|_{L^\infty\left(\Delta_s^X\left(\epsilon_0\right)\right)}\leq
 C_{d,s,\beta_0,\mu_0,T}^n\|\phi_s\|_{L^\infty_{V_s}}R^{d(s+3n)}\epsilon^{1/2},
 \end{equation}
 for all $(N,\epsilon)$ in the scaling \eqref{scaling} with $N$ large enough.

 \end{proposition}
 \begin{proof}
 Fix $1\leq k\leq n$ and $(J,M)\in\mathcal{U}_{s,k}$. Consider $(N,\epsilon)$ in the scaling \eqref{scaling} with $N$ large enough such that $n^{3/2}\epsilon<<\alpha$. Triangle inequality and the  fact that $\Delta_s^X(\epsilon_0)\subseteq\Delta_s^X(\epsilon_0/2)$ yield
 \begin{align}
 &\|J_{s,k,R,\delta}^N(t,J,M)-\widehat{J}_{s,k,R,\delta}^N(t,J,M)\|_{L^\infty\left(\Delta_s^X\left(\epsilon_0\right)\right)}\nonumber\\
 &\leq \|J_{s,k,R,\delta}^N(t,J,M)-\bm{A_{N,\epsilon}^{s,k}}\widehat{J}_{s,k,R,\delta}^N(t,J,M)\|_{L^\infty\left(\Delta_s^X\left(\epsilon_0/2\right)\right)}+(1-\bm{A_{N,\epsilon}^{s,k}})\|\widehat{J}_{s,k,R,\delta}^N(t,J,M)\|_{L^\infty\left(\Delta_s^X\left(\epsilon_0\right)\right)}.\label{auxiliary estimate 2}
 \end{align}

 We estimate each of the terms in \eqref{auxiliary estimate 2}. For the first term, let us fix $(t_1,...,t_k)\in\mathcal{T}_{k,\delta}(t)$. Applying \eqref{final condition on Boltz pseudo k=1,,,} for  $i=k-1$, we obtain 
 $Z_{s+2k-2}^\infty(t_k^+)\in G_{s+2k-2}(\epsilon_0,0).$
Since $s<n$ and $n^{3/2}\epsilon<<\alpha$, \eqref{total proximity}, applied for $i=k$, implies
 $|X_{s+2k-2}^N(t_k^+)-X_{s+2k-2}^\infty(t_k^+)|\leq\frac{\alpha}{2}.$
 Therefore, Proposition \ref{bad set} (precisely expression \eqref{pre-delta} for the pre-collisional case, \eqref{post-delta} for the post-collisional case) implies
$
 Z_{s+2k}^N(0^+)\in G_{s+2k}(\epsilon_0/2,0)\subseteq\Delta_{s+2k}(\epsilon_0/2).
 $
 Thus  \eqref{estimate on the rest of the terms}, \eqref{exclusion bad set 2 norm}-\eqref{exclusion bad set 2 time},  \eqref{truncated elementary bbgky}-\eqref{auxiliary functionals}  imply 
 \begin{equation}\label{final estimate 2}
 \begin{aligned}
 \|J_{s,k,R,\delta}^N(t,J,M)-\bm{A_{N,\epsilon}^{s,k}}\widehat{J}_{s,k,R,\delta}^N(t,J,M)\|_{L^\infty\left(\Delta_s^X\left(\epsilon_0/2\right)\right)}&\leq  \frac{C_{d,s,T}^k}{k!}\|\phi_s\|_{L^\infty_{V_s}}R^{d(s+3k)}\|f_{N,0}^{(s+2k)}-f_0^{(s+2k)}\|_{L^\infty(\Delta_{s+2k}(\epsilon_0/2))}\\
 &\leq \frac{C_{d,s,T}^k}{k!}\|\phi_s\|_{L^\infty_{V_s}}R^{d(s+3k)}\|f_{N,0}^{(s+2k)}-f_0^{(s+2k)}\|_{L^\infty(\mathcal{D}_{s+2k,\epsilon})},
 \end{aligned}
 \end{equation}
 as long as $\epsilon<\epsilon_0/2\sqrt{2}$ (i.e. $N$ large enough) so that $\Delta_{s+2k}(\epsilon_0/2)\subseteq\mathcal{D}_{s+2k,\epsilon}$.
 For the second term, using  \eqref{estimate on the rest of the terms}
 we obtain
 \begin{equation}\label{final estimate 3}
 \begin{aligned}
 \|\widehat{J}_{s,k,R,\delta}^N(t,J,M)\|_{L^\infty\left(\Delta_s^X\left(\epsilon_0\right)\right)}\leq \frac{C_{d,s,\mu_0,T}^k}{k!}\|\phi_s\|_{L^\infty_{V_s}}R^{d(s+3k)}\|F_0\|_{\infty,\beta_0,\mu_0}.
 \end{aligned}
 \end{equation}
 
 Adding over all $(J,M)\in\mathcal{U}_{s,k}$, $k=1,...,n$, using \eqref{auxiliary estimate 2}-\eqref{final estimate 3}, \eqref{estimate on scaling} and an argument similar to \eqref{use of stirling} to control the summation over $k=1,...,n$ , for $N$ large enough, we obtain the estimate
 \begin{equation*}
 \begin{aligned}
\sum_{k=1}^n\sum_{(J,M)\in\mathcal{U}_{s,k}}\|&J_{s,k,R,\delta}^N(t,J,M)-\widehat{J}_{s,k,R,\delta}^N(t,J,M)\|_{L^\infty\left(\Delta_s^X\left(\epsilon_0\right)\right)}\leq
C_{d,s,\mu_0,T}^n\|\phi_s\|_{L^\infty_{V_s}}R^{d(s+3n)}\\
&\hspace{-1cm}\times\left(\sup_{k\in\{1,...,n\}}\|(f_{N,0}^{(s+2k)}-f_0^{(s+2k)})\|_{L^\infty(\mathcal{D}_{s+2k,\epsilon})}+\|F_0\|_{\infty,\beta_0,\mu_0}\epsilon^{d-1/2}\right).
 \end{aligned}
 \end{equation*}
Since $n$ is fixed, the result follows from convergence in the level of initial data and the scaling estimate \eqref{estimate on scaling}. 
 
  In the case of tensorized initial data and approximation by conditioned BBGKY initial data, the estimate can be improved to \eqref{estimate 1 on rate} using \eqref{tensorized data approximating estimate}.
  
 \end{proof}
 Due to the proximity Lemma \ref{proximity} and the uniform continuity assumption \eqref{continuity assumption} on the Boltzmann hierarchy initial data, we also obtain the following
 \begin{proposition}\label{aux prop 2} Let $s,n\in\mathbb{N}$ with $s<n$,  $\alpha,\epsilon_0,R,\eta,\delta$ be parameters as in \eqref{choice of parameters} and $t\in[0,T]$. Then for any  $\zeta>0$, there is $N^*=N^*(\zeta)\in\mathbb{N}$, such that for all $(N,\epsilon)$ in the scaling \eqref{scaling} with
 $N>N^*$, there holds
 \begin{equation}\label{estimate 2}
\sum_{k=1}^n\sum_{(J,M)\in\mathcal{U}_{s,k}}\|\widehat{J}_{s,k,R,\delta}^N(t,J,M)-J_{s,k,R,\delta}^\infty(t,J,M)\|_{L^\infty\left(\Delta_s^X\left(\epsilon_0\right)\right)}\leq C_{d,s,\mu_0,T}^n\|\phi_s\|_{L^\infty_{V_s}}R^{d(s+3n)}\zeta^2.
 \end{equation}
  
  In the case of H\"older continuous $C^{0,\gamma}$, $\gamma\in(0,1]$ tensorized initial data  (see Remark \ref{remark on conditioned}), the estimate can be improved to
   \begin{equation}\label{estimate 2 rate}
\sum_{k=1}^n\sum_{(J,M)\in\mathcal{U}_{s,k}}\|\widehat{J}_{s,k,R,\delta}^N(t,J,M)-J_{s,k,R,\delta}^\infty(t,J,M)\|_{L^\infty\left(\Delta_s^X\left(\epsilon_0\right)\right)}\leq C_{d,s,\mu_0,T}^n\|\phi_s\|_{L^\infty_{V_s}}R^{d(s+3n)}\epsilon^\gamma,
 \end{equation}
 for all $(N,\epsilon)$ in the scaling \eqref{scaling}.

 \end{proposition}

 \begin{proof}
Let $\zeta>0$. Fix $1\leq k\leq n$ and $(J,M)\in\mathcal{U}_{s,k}$. Since $s<n$, Lemma \ref{proximity} yields
\begin{equation}\label{continuity for integral}
|Z_{s+2k}^N(0^+)-Z_{s+2k}^\infty(0^+)|\leq \sqrt{6}n^{3/2}\epsilon,\quad\forall Z_s\in\mathbb{R}^{2ds}.
\end{equation}
Thus the continuity assumption \eqref{continuity assumption} on $F_0$, \eqref{continuity for integral} and the scaling \eqref{scaling} imply that there exists $N^*=N^*(\zeta)\in\mathbb{N}$, such that for all $N>N^*$, we have
\begin{equation}\label{continuity satisfied}
|f_0^{(s+2k)}(Z_{s+2k}^N(0^+))-f_0^{(s+2k)}(Z_{s+2k}^\infty(0^+))|\leq C^{s+2k-1}\zeta^2,\quad\forall Z_s\in\mathbb{R}^{2ds}.
\end{equation}
In the same spirit as in the proof of Proposition \ref{aux prop 1}, using \eqref{continuity satisfied}, \eqref{estimate on the rest of the terms}, \eqref{exclusion bad set 2 time}, and summing over $(J,M)\in\mathcal{U}_{s,k}$, $k=1,...,n$,  we obtain estimate \eqref{estimate 2}.

In the case of tensorized $C^{0,\gamma}$ data, one can easily see by induction that for any $Z_{s+2k},Z_{s+2k}'\in\mathbb{R}^{2d(s+2k)}$, we have
\begin{align*}
|f_0^{\otimes (s+2k)}(Z_{s+2k})-f_0^{\otimes (s+2k)}(Z_{s+2k}')|&\leq \|f_0\|_{L^\infty}^{s+2k-1}[f_0]_{C^{0,\gamma}}\sqrt{2d(s+2k)}|Z_{s+2k}-Z_{s+2k}'|^\gamma\\
&\leq C^{s+2k-1}|Z_{s+2k}-Z_{s+2k}'|^\gamma.
\end{align*}
Thus by \eqref{continuity for integral} we have
$$ |f_0^{(s+2k)}(Z_{s+2k}^N(0^+))-f_0^{(s+2k)}(Z_{s+2k}^\infty(0^+))|\leq C^{s+2k-1}\epsilon^\gamma,$$
and the estimate \eqref{estimate 2 rate} follows in a similar manner as estimate \eqref{estimate 2}.

 \end{proof}

 \subsection{Proof of Theorem \ref{convergence theorem}}
 
 We are now in the position to prove Theorem \ref{convergence theorem}. Fix $\sigma>0$, $s\in\mathbb{N}$, $\phi_s\in C_c(\mathbb{R}^{ds})$ and $t\in[0,T]$. Consider $n\in\mathbb{N}$ with $s<n$, and  parameters $\alpha,\epsilon_0,R,\eta,\delta$ satisfying \eqref{choice of parameters}. Let $\zeta>0$ small enough. Triangle inequality, Propositions \ref{reduction}, \ref{restriction to initially good conf}, \ref{truncated element estimate}, Remark \ref{no need for k=0}, estimates \eqref{estimate 1}, \eqref{estimate 2} and part \textit{(i)} of Definition  \ref{convergence of initial data}, yield that there is $N^*(\zeta)\in\mathbb{N}$ such that for all $N>N^*$,  we have
 \begin{equation}\label{final final bounds}
 \begin{aligned}
 & \|I_s^N(t)-I_s^\infty(t)\|_{L^\infty\left(\Delta_s^X\left(\epsilon_0\right)\right)}\leq C\left(2^{-n}+e^{-\frac{\beta_0}{3}R^2}+ \delta C^n\right)+C^n R^{4dn}\eta^{\frac{d-1}{4d+2}}+C^nR^{4dn}\zeta^2,
 \end{aligned}
 \end{equation}
 where 
$C>1$ is an appropriate  constant.

  We now choose  parameters satisfying \eqref{choice of parameters}, depending only on $\zeta$, such that the right hand side of \eqref{final final bounds} becomes less than $\zeta$. 

\textit{Choice of parameters}: For $\zeta$ sufficiently small, we choose $n\in\mathbb{N}$ and the parameters $\delta,\eta,R,\epsilon_0,\alpha $ in the following order:
\begin{align}
&\max\left\{s,\log_2(C\zeta^{-1})\right\}<< n,\quad \delta<< \zeta C^{-(n+1)}, \quad  \max\left\{1,\sqrt{3}\beta_0^{-1/2}\ln^{1/2}(C\zeta^{-1})\right\}< <R<<\zeta^{-1/4dn}C^{-1/4d},\nonumber\\
& \eta<<\zeta^{\frac{8d+4}{d-1}},\quad \epsilon_0<<\min\{\sigma,\eta\delta\},\quad\alpha<<\epsilon_0\min\{1,R^{-1}\eta\}.\label{final parameters}
\end{align}
 Relations \eqref{final parameters} imply the parameters chosen satisfy \eqref{choice of parameters} and depend only on $\zeta$.
Then,  \eqref{final final bounds}-\eqref{final parameters} imply that we may find $N_0(\zeta)\in\mathbb{N}$, such that for all $(N,\epsilon)$ in the scaling \eqref{scaling} with $N>N_0$, there holds
\begin{equation*}
\|I_s^N(t)-I_s^\infty(t)\|_{L^\infty\left(\Delta_s^X\left(\sigma\right)\right)}\overset{\epsilon_0<\sigma}\leq\|I_s^N(t)-I_s^\infty(t)\|_{L^\infty\left(\Delta_s^X\left(\epsilon_0\right)\right)}<\zeta,
\end{equation*}
and Theorem \ref{convergence theorem} is proved.

\textbf{Proof of Corollary \ref{derivation corollary}} 

By  Theorem \ref{theorem propagation of chaos} we have that $\bm{F}=(f^{\otimes s})_{s\in\mathbb{N}}$, where $f$ is the mild solution of the ternary Boltzmann equation. Therefore, in the same spirit as before (using estimates \eqref{estimate 1 on rate}, \eqref{estimate 2 rate} instead of \eqref{estimate 1}, \eqref{estimate 2}), for $N$ large enough we have
\begin{align}\label{final estimates tensorized}
\|I_{\phi_s}f_N^{(s)}(t)-I_{\phi_s}f^{\otimes s}(t)\|_{L^\infty(\Delta^X_s(\epsilon_0))}&\leq C\left(2^{-n}+e^{-\frac{\beta_0}{3}R^2}+ \delta C^n\right)+C^n R^{4dn}\eta^{\frac{d-1}{4d+2}}+C^nR^{4dn}\epsilon^{\gamma_*},
\end{align}
where $\gamma_*=\min\{1/2,\gamma\}\in (0,\frac{1}{2}]$ and $\gamma$ is the H\"older regularity of $f_0$. Consider $0<r<\gamma_*$.

\textit{Choice of parameters}: For $N$ large enough (or equivalently for $\epsilon$ small enough), we choose $n\in\mathbb{N}$ and the parameters $\delta,\eta,R,\epsilon_0,\alpha $ in the following order:
\begin{align}
&\max\left\{s,\log_2(C\epsilon^{\gamma_*})\right\}<< n,\quad \delta<< \epsilon^{\gamma_*} C^{-(n+1)}, \quad  \max\left\{1,\sqrt{3}\beta_0^{-1/2}\ln^{1/2}(C\epsilon^{-\gamma_*})\right\}< <R<<\epsilon^{\frac{r-\gamma_*}{4dn}}C^{-1/4d},\nonumber\\
& \eta<<\epsilon^{\frac{4d+2)}{d-1}\gamma_*},\quad \epsilon_0<<\min\{\sigma,\eta\delta\},\quad\alpha<<\epsilon_0\min\{1,R^{-1}\eta\}.\label{final parameters tensorized}
\end{align}
Then by \eqref{final estimates tensorized}, for $N$ large enough, we take
$$\|I_{\phi_s}f_N^{(s)}(t)-I_{\phi_s}f^{\otimes s}(t)\|_{L^\infty(\Delta^X_s(\sigma))}\overset{\epsilon_0<\sigma}\leq\|I_{\phi_s}f_N^{(s)}(t)-I_{\phi_s}f^{\otimes s}(t)\|_{L^\infty(\Delta^X_s(\epsilon_0))}<\epsilon^{r},$$
and Corollary \ref{derivation corollary} is proved.

\appendix
\section{Auxiliary results}
In this appendix, we state two auxiliary results. For the proofs, see \cite{thesis}.
\begin{lemma}\label{linear algebra lemma} Let $n\in\mathbb{N}$, $\lambda\neq 0$ and $w,u\in\mathbb{R}^n$. Denoting by $I_n$ the $n\times n$ identity matrix, we have
$$
\det(\lambda I_n+w u^T)=\lambda^n(1+\lambda^{-1}\langle w,u\rangle).
$$
\end{lemma}
\begin{lemma}\label{substitution lemma} 
Let $n\in\mathbb{N}$, $\Psi:\mathbb{R}^n\to\mathbb{R}$ be a $C^1$ function and $\gamma\in\mathbb{R}$. Assume there is $\delta>0$ with $\nabla\Psi(\bm{\omega})\neq 0$ for $\bm{\omega}\in [\gamma-\delta<\Psi<\gamma+\delta]$. Let $\Omega\subseteq\mathbb{R}^n$ be a domain and consider a $C^1$ map $F:\Omega\to\mathbb{R}^n$  of non-zero Jacobian in $\Omega
 $. Then for any measurable $g:\mathbb{R}^n\to[0,+\infty]$ or $g:\mathbb{R}^n\to\ [-\infty,+\infty]$ integrable
\begin{equation}\label{change of variables formula}
\int_{[\Psi=\gamma]}g(\nu)\mathcal{N}_F(\nu,[\Psi\circ F=\gamma])\,d\sigma(\nu)=\int_{[\Psi\circ F=\gamma]}(g\circ F)(\omega)|\jac F(\omega)|\frac{|\nabla\Psi(F(\omega))|}{|\nabla(\Psi\circ F)(\omega)|}\,d\sigma(\omega),
\end{equation}
where given $\nu\in\mathbb{R}^n$ and  $A\subseteq\Omega$, $\mathcal{N}_F(\nu,A):=\card(\{\omega\in A:F(\omega)=\nu\})$ is the Banach indicatrix of $A$.

\end{lemma}

\end{document}